\documentclass[a4paper,reqno]{amsart}
\usepackage[english]{babel}
\usepackage{amsmath,amsthm,amssymb,amsfonts}
\usepackage{mathrsfs}

\usepackage{enumitem}
\setlist[enumerate]{leftmargin=*}

\usepackage{hyperref}

\newtheorem{thm}{Theorem}[section]
\newtheorem{prp}[thm]{Proposition}
\newtheorem{cor}[thm]{Corollary}
\newtheorem{lem}[thm]{Lemma}
\theoremstyle{definition}
\newtheorem{dfn}[thm]{Definition}
\newtheorem{rem}[thm]{Remark}

\newtheorem{step}{Step}

\numberwithin{equation}{section}

\newcommand{\tc}{\,:\,}
\newcommand{\defeq}{\mathrel{:=}}
\newcommand{\eqdef}{\mathrel{=:}}

\newcommand{\RR}{\mathbb{R}}
\newcommand{\CC}{\mathbb{C}}
\newcommand{\NN}{\mathbb{N}}
\newcommand{\ZZ}{\mathbb{Z}}

\newcommand{\Npos}{\mathbb{N}_+}
\newcommand{\Rpos}{\mathbb{R}^+}
\newcommand{\Rnon}{\mathbb{R}^+_0}
\newcommand{\schur}{\odot}

\newcommand{\pot}{\mathcal{P}}
\newcommand{\halfpot}{\mathcal{HP}}
\newcommand{\Tpot}{\mathcal{P}^\mathrm{t}}
\newcommand{\Tpotone}{\mathcal{P}^\mathrm{t}_1}

\newcommand{\mm}{\mathrm{m}}

\newcommand{\matP}{\mathbf{P}}
\newcommand{\matA}{\mathbf{A}}
\newcommand{\matB}{\mathbf{B}}
\newcommand{\matM}{\mathbf{M}}
\newcommand{\matN}{\mathbf{N}}
\newcommand{\matF}{\mathbf{F}}

\newcommand{\auxC}{\mathfrak{C}}
\newcommand{\auxD}{\mathfrak{D}}

\newcommand{\banP}{\mathscr{P}}
\newcommand{\banR}{\mathscr{R}}
\newcommand{\banD}{\mathscr{D}}
\newcommand{\banK}{\mathscr{K}}
\newcommand{\banE}{\mathscr{E}}

\newcommand{\solU}{\mathcal{U}}

\newcommand{\formE}{\mathcal{E}}

\newcommand{\chr}{\mathbf{1}}

\DeclareMathOperator{\diag}{diag}
\DeclareMathOperator{\inc}{inc}

\DeclareMathOperator{\codim}{codim}

\DeclareMathOperator{\supp}{supp}
\DeclareMathOperator*{\esssup}{ess\,sup}
\newcommand{\loc}{\mathrm{loc}}

\newcommand{\sobolev}[2]{L^{#2}_{#1}}

\newcommand{\opL}{\mathcal{L}}
\newcommand{\opH}{\mathcal{H}}

\newcommand{\romI}{\mathrm{I}}
\newcommand{\romII}{\mathrm{II}}
\newcommand{\romIII}{\mathrm{III}}

\newcommand{\dist}{\mathrm{dist}}
\newcommand{\Vol}{\mathrm{Vol}}
\newcommand{\weight}{\varpi}
\newcommand{\Kern}{\mathcal{K}}

\newcommand{\doub}{\mathcal{D}}
\newcommand{\inv}{\leftarrow}

\renewcommand{\vartheta}{\gamma}

\begin{document}

\title[Grushin operators in the plane]{An optimal multiplier theorem for Grushin operators in the plane, I}

\author{Gian Maria Dall'Ara}
\address[G.\ M.\ Dall'Ara]{
Istituto Nazionale di Alta Matematica ``Francesco Severi'' \\ 
Research Unit Scuola Normale Superiore \\ Piazza dei Cavalieri 7 \\ 56126 Pisa \\ Italy}
\email{dallara@altamatematica.it}
\author{Alessio Martini}
\address[A.\ Martini]{School of Mathematics \\ University of Birmingham \\ Edgbaston \\ Birmingham \\ B15 2TT \\ United Kingdom}
\email{a.martini@bham.ac.uk}

\thanks{A substantial part of this project was developed while the first-named author was a Marie Sk\l odowska-Curie Research Fellow at the University of Birmingham. Both authors gratefully acknowledge the support of the European Commission via the Marie Sk\l odowska-Curie Individual Fellowship ``Harmonic Analysis on Real Hypersurfaces in Complex Space'' (ID 841094). The authors are members of the Gruppo Nazionale per l'Analisi Matematica, la Probabilit\`a e le loro Applicazioni (GNAMPA) of the Istituto Nazionale di Alta Matematica (INdAM)}


\begin{abstract}
Let $\opL = -\partial_x^2 - V(x) \partial_y^2$ be the Grushin operator on $\RR^2$ with coefficient $V : \RR \to [0,\infty)$. Under the sole assumptions that $V(-x) \simeq V(x) \simeq xV'(x)$ and $x^2 |V''(x)| \lesssim V(x)$, we prove a spectral multiplier theorem of Mihlin--H\"ormander type for $\opL$, whose smoothness requirement is optimal and independent of $V$. The assumption on the second derivative $V''$ can actually be weakened to a H\"older-type condition on $V'$. The proof hinges on the spectral analysis of one-dimensional Schr\"odinger operators, including universal estimates of eigenvalue gaps and matrix coefficients of the potential.
\end{abstract}

\maketitle

\section{Introduction}

\subsection{Statement of the results}

Let $X$ be a measure space and $\opL$ a nonnegative self-adjoint operator on $L^2(X)$. 
By the spectral theorem, $\opL$ admits a spectral resolution $E$, and a functional calculus for $\opL$ can be defined via spectral integration. In particular, for all bounded Borel functions $\mm : [0,\infty) \to \CC$, the operator
\[
\mm(\opL) = \int_{[0,\infty)} \mm(\lambda) \,dE(\lambda)
\]
is bounded on $L^2(X)$. Determining nontrivial sufficient conditions for the boundedness of operators of the form $\mm(\opL)$ on other function spaces, such as $L^p(X)$ for $p \neq 2$, in terms of properties of the ``spectral multiplier'' $\mm$, is in general a much more complicated problem.

In the case where $\opL=-\Delta$ is the Laplace operator on $\RR^n$, the classical Mihlin--H\"ormander multiplier theorem \cite{mihlin_multipliers_1956,hrmander_estimates_1960} implies that $\mm(\opL)$ is of weak type $(1,1)$ and bounded on $L^p$ for all $p \in (1,\infty)$ whenever $\mm$ is continuous on $(0,\infty)$ and satisfies a local scale-invariant Sobolev condition of the form
\begin{equation}\label{eq:mhcond_intro}
\sup_{t \geq 0} \| \mm(t \cdot) \, \eta \|_{\sobolev{s}{q}} < \infty
\end{equation}
for $q=2$ and some $s > n/2$; here $\sobolev{s}{q}(\RR)$ denotes the $L^q$ Sobolev space of (fractional) order $s$, and $\eta \in C^\infty_c((0,\infty))$ is any nontrivial cutoff. It is well known that the smoothness requirement $s>n/2$ in this result is optimal, in the sense that $n/2$ cannot be replaced by any smaller number.

Results analogous to the Mihlin--H\"ormander theorem have been obtained for more general ``Laplace-like'' operators in a variety of settings. For example, if $\opL$ is the Laplace--Beltrami operator on a compact Riemannian manifold (or, more generally, an elliptic pseudodifferential operator on a compact manifold), then the analogue of the aforementioned Mihlin--H\"ormander result was proved by Seeger and Sogge \cite{seeger_boundedness_1989}, with a smoothness condition of the form \eqref{eq:mhcond_intro} with $q=2$ and $s>n/2$, where $n$ is the dimension of the manifold. The proof of Seeger and Sogge is fundamentally based on a Fourier integral operator representation for the half-wave propagator associated to $\opL$, already exploited in \cite{hormander_spectral_1968}, and appears to break down when the ellipticity assumption on $\opL$ is weakened.

A more general and robust --- but not as sharp --- result of Mihlin--H\"ormander type, essentially due to Hebisch \cite{hebisch_functional_1995} (see also \cite{mauceri_vectorvalued_1990,christ_lpbounds_1991,alexopoulos} for some predecessors in particular cases, and \cite{cowling_spectral_2001,duong_plancherel-type_2002,martini_crsphere} for alternative approaches and variations), applies to arbitrary nonnegative self-adjoint operators $\opL$ on $L^2(X)$, where $X$ is a doubling metric measure space. Under the assumption that $\opL$ satisfies Gaussian-type heat kernel bounds, the result yields weak type $(1,1)$ and $L^p$ boundedness for $p \in (1,\infty)$ of $\mm(\opL)$ whenever $\mm$ satisfies a smoothness condition of the form \eqref{eq:mhcond_intro} with $q=\infty$ and $s>Q/2$, where $Q$ is the ``homogeneous dimension'' of the doubling space $X$. When applying this result to the Laplace operator on $\RR^n$ or the Laplace--Beltrami operator on a compact Riemannian $n$-manifold, one can take $Q=n$, thus recovering the aforementioned optimal results in those cases up to the type of Sobolev norm used (the results with $q=2$ are sharper than those with $q=\infty$). At the same time, the approach based on heat kernel bounds yields a wider applicability of the result, including cases where the operator $\opL$ fails to be elliptic. In several of these cases, however, the smoothness requirement $s>Q/2$ turns out not to be optimal.

Let us consider, for example, degenerate elliptic operators on $\RR^{n_1}_x \times \RR^{n_2}_y$ of the form
\begin{equation}\label{eq:grushin}
\opL = -\Delta_x - V(x) \Delta_y,
\end{equation}
where the coefficient $V : \RR^{n_1} \to [0,\infty)$ is measurable and bounded above and below by multiples of the function
\begin{equation}\label{eq:twopowerlaw}
x \mapsto \begin{cases}
|x|^d &\text{if $|x| \leq 1$,}\\
|x|^D &\text{if $|x| \geq 1$}
\end{cases}
\end{equation}
for some $d,D>0$. In the case where $V$ is a polynomial, these operators are among those studied in \cite{grushin_certain_1970}, whence the name ``Grushin operators'' commonly used to refer to operators of the form \eqref{eq:grushin}; sometimes the name ``Baouendi--Grushin operators'' is also used, due to the previous work \cite{baouendi} on degenerate elliptic operators. In \cite{robinson_analysis_2008} Robinson and Sikora develop a detailed analysis of degenerate elliptic operators of Grushin type, yielding among other things that Hebisch's multiplier theorem applies to operators of the form \eqref{eq:grushin}, with homogeneous dimension $Q = n_1 + (1+\max\{D,d\}/2) n_2$. No claim on the optimality of the condition $s>Q/2$ is made there, and indeed improvements turn out to be possible in some cases.

In the case where $V(x) = |x|^2$, in \cite{martini_grushin_2012,martini_sharp_2014} it is proved that the smoothness condition can be pushed down to $s>(n_1+n_2)/2$ and expressed in terms of an $L^2$ Sobolev norm. The same smoothness condition turns out to be enough also in the case where $V(x) = \sum_{j=1}^{n_1} |x_j|$ and $n_1 \geq n_2/2$ \cite{chen_sharp_2013}. The threshold $(n_1+n_2)/2$, that is, half the topological dimension of the underlying manifold, corresponds to the sharp threshold in the classical Mihlin--H\"ormander theorem for the Euclidean Laplacian on $\RR^{n_1+n_2}$, and therefore cannot be beaten in the case of Grushin operators, since they are locally elliptic where $x \neq 0$ \cite{mitjagin_divergenz_1974,kenig_divergence_1982}. In view of these examples, it is natural to ask whether the optimal smoothness requirement coincides with the Euclidean condition $s>(n_1+n_2)/2$ for any Grushin operator \eqref{eq:grushin}, and in particular whether the optimal threshold is independent of the ``degrees'' $d$ and $D$ of the coefficient $V$.

A partial positive answer in the case $d=D>1$ is presented in \cite{dallara_martini}. Under some structural and smoothness assumptions on $V$ (namely, $V(x) = \sum_{j=1}^{n_1} V_j(|x_j|)$, where each $V_j : (0,\infty) \to (0,\infty)$ is comparable to the power law $t \mapsto t^D$ up to the third derivative), a theorem of Mihlin--H\"ormander type is proved with $L^2$ smoothness condition $s>\max\{n_1+n_2, (1+ D/2) n_2\}/2$.
In addition to not requiring any analyticity or homogeneity of the coefficient $V$, this result yields the optimal degree-independent condition $s>(n_1+n_2)/2$ whenever $n_2 \leq 2n_1/D$; in particular, however large the degree $D$ is, we can reach the Euclidean ``half the topological dimension'' threshold, provided $n_2$ is sufficiently small compared to $n_1$. 

The constraint $n_2 \leq 2n_1/D$ from the previous discussion is somewhat unsatisfactory: for example it excludes the lowest dimensional case $n_1=n_2=1$ as soon as $D>2$. Indeed, strikingly enough, even in the apparently simplest case of a homogeneous Grushin operator on $\RR^2$,
\begin{equation}\label{eq:2dgrushin_power}
\opL = -\partial_x^2 - |x|^D \partial_y^2,
\end{equation}
the aforementioned results yield the Euclidean condition $s>2/2$ only when $1 \leq D \leq 2$. When $D>2$, the result from \cite{dallara_martini} gives the $L^2$ condition $s>(1+D/2)/2$, whose threshold becomes arbitrarily large as $D$ grows; if $0<D<1$, instead, among the existing results, only the general theorem from \cite{robinson_analysis_2008} with $L^\infty$ condition $s>Q/2 = (2+D/2)/2$ appears to be applicable. In light  of this, the analysis of two-dimensional cases appears to be a natural choice as a testing ground for the question whether the Euclidean condition is always the optimal smoothness requirement.

In this paper we focus on two-dimensional Grushin operators of the form
\begin{equation}\label{eq:2dgrushin}
\opL = -\partial_x^2 - V(x) \partial_y^2,
\end{equation}
where $V : \RR \to [0,\infty)$ is assumed to be continuous, not identically zero, $C^1$ off the origin, and satisfying, for some $\theta \in (0,1)$, the estimates
\begin{subequations}\label{eq:2dassumptions}
\begin{equation}\label{eq:2dassumptions_doubling}
V(-x) \simeq V(x) \simeq x V'(x),  
\end{equation}
\begin{equation}\label{eq:2dassumptions_holder}
|V'(xe^h)-V'(x)| \lesssim |V'(x)| \, |h|^\theta
\end{equation}
\end{subequations}
for all $x \in \RR \setminus\{0\}$ and $h \in [-1,1]$. Here $A \lesssim B$ means that there exists a constant $C>0$ such that $A \leq C B$, and $A \simeq B$ is the conjunction of $A \lesssim B$ and $B \lesssim A$; we also write $A \lesssim_s B$ or $A \simeq_s B$ to indicate that the implicit constants may depend on a parameter $s$.

The assumption \eqref{eq:2dassumptions_holder} is a scale-invariant H\"older-type condition on $V'$, which is verified, for example, whenever $V$ is twice differentiable on $\RR \setminus \{0\}$ and satisfies the estimate $|x V''(x)| \lesssim |V'(x)|$.
Clearly the assumptions \eqref{eq:2dassumptions} include the operators \eqref{eq:2dgrushin_power} among others. Under these sole assumptions, we confirm that the Euclidean condition $s>2/2$ is indeed the optimal smoothness requirement, at least when expressed in terms of an $L^\infty$ Sobolev norm.

\begin{thm}\label{thm:main}
Let $\opL$ be the Grushin operator on $\RR^2$ defined by \eqref{eq:2dgrushin}, where the coefficient $V$ satisfies the estimates \eqref{eq:2dassumptions}. Let $s>2/2$.
\begin{enumerate}[label=(\roman*)]
\item\label{en:main_l1} For all $\mm : \RR \to \CC$ such that $\supp \mm \subseteq [-1,1]$,
\[
\sup_{t > 0} \| \mm(t\opL) \|_{L^1 \to L^1} \lesssim_s \| \mm \|_{\sobolev{s}{\infty}}.
\]
\item\label{en:main_mh} Let $\eta \in C^\infty_c((0,\infty))$ be nonzero. For all $\mm : \RR \to \CC$ and $p \in (1,\infty)$,
\[
\| \mm(\opL) \|_{L^1 \to L^{1,\infty}} \lesssim_s \sup_{t>0} \| \mm(t \cdot) \eta \|_{\sobolev{s}{\infty}}, \qquad \| \mm(\opL) \|_{L^p \to L^p} \lesssim_{s,p} \sup_{t>0}\| \mm(t\cdot) \eta \|_{\sobolev{s}{\infty}}.
\]
\end{enumerate}
\end{thm}

It is worth pointing out that the assumptions \eqref{eq:2dassumptions} are substantially less restrictive than the corresponding ones in \cite{dallara_martini}. Among other things, the choices
\begin{align*}
V(x) &= |x|^d + |x|^D &&\qquad\text{for any } D,d>0,\\
V(x) &= 1/(|x|^{-d} + |x|^{-D}) &&\qquad\text{for any } D,d>0,\\
V(x) &= |x|^D \log(2+|x|) &&\qquad\text{for any } D > 0,
\end{align*}
are allowed here, thus showing that we do not require $V$ to be comparable to a power law with a specific degree.
The last example goes even beyond the class of coefficients considered in \cite{robinson_analysis_2008}, since it is not comparable to a function of the form \eqref{eq:twopowerlaw}.

Plenty of other examples can be constructed, 
by observing that the assumptions \eqref{eq:2dassumptions} admit a particularly simple rephrasing in terms of the functions $\upsilon_\pm(t) = \log V(\pm e^t)$. Namely, with \eqref{eq:2dassumptions_doubling} we are requiring that the $\upsilon_\pm : \RR \to \RR$ are continuously differentiable and
\[
\upsilon_\pm' \simeq 1, \qquad |\upsilon_+ - \upsilon_-| \lesssim 1,
\]
while the condition \eqref{eq:2dassumptions_holder} requires in addition that the $\upsilon_\pm'$ are $\theta$-H\"older continuous:
\[
|\upsilon_\pm'(t+h) - \upsilon'_\pm(t)| \lesssim |h|^\theta.
\]
In the case $V$ is homogeneous, the functions $\upsilon_\pm'$ are constant and equal to the homogeneity degree; so in a sense we could think of the assumptions \eqref{eq:2dassumptions} as allowing for potentials $V$ with non-constant, H\"older-continuous degree.

We remark that the smoothness condition in Theorem \ref{thm:main} is expressed in terms of an $L^\infty$ Sobolev norm, in contrast to the $L^2$ Sobolev condition used in \cite{martini_grushin_2012,chen_sharp_2013,martini_sharp_2014,dallara_martini}. In addition, the latter results also apply to higher-dimensional cases. As a matter of fact, the methods presented here can be adapted to treat some higher-dimensional cases as well, and
even to prove the sharper result with an $L^2$ Sobolev condition. However, the proof of Theorem \ref{thm:main} turns out to be relatively complicated as it is, and we believe that it may be of interest to present it in the simplest and cleanest form, instead of obfuscating the core ideas with additional technicalities. 
A separate future work \cite{DMfuture}, building on the present one, will be devoted to the $L^2$ Sobolev norm improvement of Theorem \ref{thm:main}.

In any case, the fact that here we can obtain a smoothness condition that is independent of the coefficient $V$ (and matches the Euclidean one, corresponding to $V \equiv 1$) is already a substantial improvement over the results in \cite{robinson_analysis_2008,dallara_martini}, where the smoothness threshold (for fixed dimension) grows with the degree of $V$. This is true even if we restrict our attention to $L^2$ Sobolev conditions: indeed, by Sobolev's embedding, Theorem \ref{thm:main} implies a corresponding result with $L^2$ condition $s>3/2$, which is anyway a huge improvement over the $L^2$ condition $s>(1+D/2)/2$ from \cite{dallara_martini} as soon as $D$ is large. Moreover, in its form with an $L^\infty$ condition, our result implies an essentially sharp bound on the growth of the norms of the imaginary powers:
\[
1+|\alpha| \lesssim  \| \opL^{i\alpha} \|_{L^1 \to L^{1,\infty}} \lesssim_{\epsilon} (1+|\alpha|)^{1+\epsilon}
\]
for all $\alpha \in \RR$ and $\epsilon>0$; the upper bound is a direct consequence of Theorem \ref{thm:main}, while the lower bound follows via transplantation from the corresponding one for the Laplacian on $\RR^2$ (see, e.g., \cite{sikora_imaginary_2001}).

In the case where $V$ is a homogeneous polynomial, Grushin operators of the form \eqref{eq:grushin} can be lifted to homogeneous left-invariant sub-Laplacians on stratified nilpotent Lie groups \cite{rothschild_hypoelliptic_1976,folland_lifting,robinson_grushin_2016,dziubanski_sikora}, and a number of their properties can be recovered by the analysis of these group-invariant sub-Laplacians. As a matter of fact, the discovery that the topological dimension and not the homogeneous dimension may determine the optimal Mihlin--H\"ormander smoothness condition for non-elliptic operators was first made by Hebisch and by M\"uller and Stein in the case of sub-Laplacians on the Heisenberg groups \cite{hebisch_multiplier_1993,mueller_spectral_1994}. Since then, the problem of determining the optimal smoothness condition for sub-Laplacians on stratified groups and more general sub-Riemannian manifolds has been extensively investigated (see, e.g., \cite{martini_necessary_2016,martini_mueller_golo} and references therein), but remains widely open. The results in the present paper contribute to this research programme, by confirming among other things that the Euclidean threshold can be reached in the case of the simplest ``sum-of-square operators'' $-\partial_x^2 - (x^k \partial_y)^2$ on $\RR^2$, irrespective of the degree $k \in \NN$. This is especially striking given that the corresponding result for the lifted operator on a stratified group (specifically, a filiform group of step $k+1$) is only known for $k\leq 1$, and gives hope of shedding some light on the problem for higher-step stratified groups and sub-Laplacians.

\subsection{Proof strategy}
By relatively standard arguments, the proof of Theorem \ref{thm:main} reduces to that of the $L^1$-estimate
\[
\sup_{r>0} \| \mm(r^2 \opL) \|_{L^1 \to L^1} \lesssim_s \| \mm \|_{\sobolev{s}{\infty}}
\]
for all $s > 2/2$ and all continuous functions $\mm : \RR \to \CC$ supported in $[1/4,1]$. In the case $s>Q/2$, this estimate follows from the general heat kernel argument of \cite{hebisch_functional_1995,duong_plancherel-type_2002}, as a consequence via the Cauchy--Schwarz inequality of a weighted $L^2$ estimate for the integral kernel of $\mm(r^2 \opL)$:
\[
\sup_{r>0} \esssup_{z' \in \RR^2} \Vol(z',r) \int_{\RR^2} |\Kern_{\mm(r^2 \opL)}(z,z')|^2 (1+\dist(z,z')/r)^{2\alpha} \,dz \lesssim_{\alpha,\beta} \| \mm \|_{\sobolev{\beta}{\infty}}^2
\]
for all $\beta > \alpha \geq 0$, where $\dist$ is the control distance on $\RR^2$ associated with $\opL$, and $\Vol(z,r)$ is the Lebesgue measure of the $\dist$-ball of centre $z$ and radius $r$. To improve on this, here we look for some extra gain, by means of weighted $L^2$ estimates involving different weights and Sobolev norms.

As in other works on the subject, our analysis is based on the fact that, via a partial Fourier transform in the $y$-variable, the Grushin operator \eqref{eq:2dgrushin} corresponds to a one-parameter family of one-dimensional Schr\"odinger operators
\[
\opH[\tau V] = -\partial_x^2 +\tau V(x) ,
\]
where the parameter $\tau$ ranges in $(0,\infty)$. In particular, if $E_n(\tau V)$ and $\psi_n(\cdot;\tau V)$ denote the eigenvalues and eigenfunctions of $\opH[\tau V]$, then we can write the integral kernel of the operator $\mm(\opL)$ as
\begin{equation}\label{eq:kernel_formula_intro}
\Kern_{\mm(\opL)}(z,z') = \frac{1}{2\pi} \int_\RR \sum_n \mm(E_n(\xi^2 V)) \psi_n(x;\xi^2 V) \psi_n(x';\xi^2 V) \,e^{i\xi (y-y')} \,d\xi,
\end{equation}
where $z=(x,y)$ and $z'=(x',y')$ range in $\RR^2$.
In order to obtain the aforementioned extra gain, the strategy in \cite{martini_grushin_2012,chen_sharp_2013,dallara_martini} is to look for weighted $L^2$ estimates where the weight is essentially a power of $V(x)$ and therefore depends on the variable $x$ only; while this strategy leads to sharp results in some cases, it does not seem effective to treat cases where $V(x)$ is not comparable to a power of $|x|$ or when the degree is large.

To overcome this, here we use a different strategy, already exploited in \cite{martini_sharp_2014} in the case $V(x)=x^2$, and look for estimates with weights involving powers of $y-y'$. From the formula \eqref{eq:kernel_formula_intro}, one immediately sees that multiplication by $y-y'$ corresponds to differentiation in the dual Fourier variable $\xi$. As a consequence, not only do we need estimates for eigenvalues $E_n(\tau V)$ and eigenfunctions $\psi_n(\cdot;\tau V)$ which are suitably uniform in the parameter $\tau$ (this was one of the challenges tackled in \cite{dallara_martini}), but we also require analogous estimates for the $\tau$-derivatives of these objects. Because of the focus on $\tau$-derivatives, this strategy may be compared to that in \cite{jotsaroop_sanjay_thangavelu_2013}, where the case $V(x)=x^2$ is considered, and the operator-valued Fourier multiplier theorem from \cite{weis} is exploited to reduce the problem for $\opL$ to corresponding ones for the $\opH[\tau V]$; however, the strategy of \cite{jotsaroop_sanjay_thangavelu_2013} does not appear to be efficient enough to obtain the optimal condition $s>2/2$ even in the particular case considered there.

It is worth pointing out that, in the cases considered in \cite{martini_grushin_2012,chen_sharp_2013,jotsaroop_sanjay_thangavelu_2013,martini_sharp_2014}, the coefficient $V$ is homogeneous, so the operators $\opH[\tau V]$ for different values of $\tau$ are intertwined one another by suitable scalings; in these cases, uniformity in $\tau$ is automatically verified, and the analysis reduces to that of a single Schr\"odinger operator $\opH[V]$. Moreover, in the special case of the quadratic potential $V(x)=x^2$, the eigenfunctions $\psi_n(\cdot;\tau V)$ are suitably scaled Hermite functions, and known identities for Hermite functions allow one to write
\[
\tau\partial_\tau \psi_n(x;\tau V) = a_n \psi_{n+2}(x;\tau V) + b_n \psi_{n}(x;\tau V) + c_n \psi_{n-2}(x;\tau V)
\]
for certain explicit $\tau$-independent coefficients $a_n,b_n,c_n$; this identity plays a fundamental role in the analysis of \cite{martini_sharp_2014}. For arbitrary functions $V$ satisfying \eqref{eq:2dassumptions}, we have instead an infinite expansion
\[
\tau\partial_\tau \psi_n(x;\tau V) = \sum_{m} \matA_{nm}(\tau V) \psi_m(x;\tau V),
\]
where the coefficients $\matA_{nm}(\tau V)$ are not explicitly known, and need to be estimated in the three parameters $\tau,n,m$. An important reduction is provided by the formula
\begin{equation}\label{eq:ef_der_reduction}
\matA_{nm}(\tau V) = \frac{\matP_{nm}(\tau V)}{E_n(\tau V) - E_m(\tau V)}
\end{equation}
for $n\neq m$, where the $\matP_{nm}(\tau V)$ are the matrix coefficients of the operator of multiplication by the potential $\tau V(x)$ of the Schr\"odinger operator $\opH[\tau V]$:
\[
\tau V(x) \psi_n(x;\tau V) = \sum_{m} \matP_{nm}(\tau V) \psi_m(x;\tau V).
\]
These matrix coefficients are also related to the $\tau$-derivatives of the eigenvalues:
\begin{equation}\label{eq:ev_der_reduction}
\tau\partial_\tau E_n(\tau V) = \matP_{nn}(\tau V).
\end{equation}
Thanks to the formulas \eqref{eq:ef_der_reduction} and \eqref{eq:ev_der_reduction}, the study of $\tau$-derivatives is effectively reduced to the study of quantities that, while depending on $\tau$, refer to a specific operator $\opH[\tau V]$ and do not explicitly involve $\tau$-differentiation. In view of this, estimates for the matrix coefficients $\matP_{nm}(\tau V)$ of the potential, as well as estimates for the eigenvalue gaps $E_n(\tau V)-E_m(\tau V)$, become crucial for our analysis.

A substantial part of this paper is therefore devoted to the proof of suitable bounds for eigenvalues, eigenfunctions, eigenvalue gaps and matrix coefficients of the potential for one-dimensional Schr\"odinger operators $\opH[V]$ whose potential $V$ satisfies the assumptions \eqref{eq:2dassumptions}. Since the rescaled potentials $\tau V$ satisfy the same assumptions as $V$ (with the same implicit constants) for any $\tau \in (0,\infty)$, the desired uniformity in $\tau$ will be an immediate consequence of the ``universality'' of the estimates that we obtain here (that is, the fact that they only depend on the implicit constants in \eqref{eq:2dassumptions}). Some of these estimates are already available in the literature, at least in particular cases, but not always do the existing references provide information about their uniformity, which is crucial here. For other estimates, including those of eigenvalue gaps and matrix coefficients of the potential, we could not find existing references directly addressing them. 

For this reason, here we present an essentially self-contained derivation of the 	universal bounds for one-dimensional Schr\"odinger operators $\opH[V]$ that we require, which may be of independent interest, also due to the generality of our assumptions \eqref{eq:2dassumptions} on the potential. It is worth pointing out that many of the bounds are actually obtained under weaker assumptions, and most of them do not require the H\"older assumption \eqref{eq:2dassumptions_holder} on the first derivative; the H\"older assumption is effectively only used to prove the off-diagonal decay estimate
\[
|\matP_{nm}(V)| \lesssim \frac{E_n(V)}{1+|n-m|^{1+\epsilon}}
\]
for some $\epsilon>0$, in the range where $n$ and $m$ are comparable. Without the H\"older assumption on $V'$ we can obtain the above matrix bound with $\epsilon = 0$, but the lack of summability of that bound appears to prevent us from using it in our proof of the multiplier theorem.

\subsection{Structure of the paper}
In Section \ref{s:abstracttheorem} we show, in the generality of doubling metric spaces and operators satisfying heat kernel bounds, how the multiplier theorem of \cite{hebisch_functional_1995} can be sharpened under the assumption of a suitable $L^1$ estimate.

In Section \ref{s:grushin} Grushin operators on $\RR^2$ are introduced and some of their fundamental properties described, including precise estimates for the associated control distance. In Section \ref{s:redweightedplancherel} it is shown how, in case of Grushin operators, the $L^1$ estimate required for the sharpened multiplier theorem of Section \ref{s:abstracttheorem} follows from suitable ``weighted Plancherel estimates'', which then become the target of our approach.

Sections \ref{s:halflineschroedinger}, \ref{s:halflineregular}, \ref{basic_sec} and \ref{s:matrixbounds} are devoted to the analysis of one-dimensional Schr\"odinger operators. Due to our ``approximate parity'' assumption $V(-x) \simeq V(x)$, a number of results are based on the analysis of second-order differential equations on a half-line, developed in Sections \ref{s:halflineschroedinger} and \ref{s:halflineregular}. Spectral results for operators on the real line, including estimates for eigenvalues and their gaps, are discussed in Section \ref{basic_sec}, while the bounds on the matrix coefficients of the potential are discussed in Section \ref{s:matrixbounds}.

Finally, in Section \ref{s:proofweightedplancherel}, the crucial ``weighted Plancherel estimates'' are proved, yielding the ``extra gain'' that allows us to obtain the required $L^1$ estimate and our optimal multiplier theorem.

Most of the sections start with a summary including the statements of the main results, while the details of the proofs are discussed in the later parts of the sections. 
This structure should allow the reader to skip many technical details on a first reading and quickly access the proof in Section \ref{s:proofweightedplancherel}.

\subsection{Notation}
$\chr_A$ denotes the characteristic function of the set $A$. We set $\Rpos = (0,\infty)$ and $\Rnon = [0,\infty)$. $\NN$ denotes the set of natural numbers (including zero), while $\Npos = \NN \setminus \{0\}$ is the set of the positive integers. For an invertible function $V$, we write $V^\inv$ to denote its compositional inverse. 
For a measurable subset $A \subseteq \RR$ we denote by $|A|$ its Lebesgue measure.
We write $\Kern_T$ to denote the integral kernel of the operator $T$.

\subsection{Acknowledgments}
The second-named author would like to thank Detlef M\"uller and Adam Sikora for several stimulating discussions related to the subject of this work.

\section{A sharpened multiplier theorem}\label{s:abstracttheorem}

We state here a general multiplier theorem in the context of metric measure spaces and nonnegative self-adjoint operators on $L^2$, showing how the general result of \cite{hebisch_functional_1995} can be sharpened under the assumption of a suitable $L^1$ estimate.

\begin{thm}\label{thm:abstract}
Let $(X,\dist,\mu)$ be a metric measure space and $\opL$ be a nonnegative self-adjoint operator on $L^2(X)$. Let $k \in (0,\infty)$, $q \in [2,\infty]$, and $\varsigma \geq 1/q$. Assume that the following conditions hold.
\begin{enumerate}[label=(\alph*)]
\item\label{en:abstract_assdoubling} The doubling condition: for all $r >0$ and $z \in X$,
\[
\mu(B(z,2r)) \lesssim \mu(B(z,r)).
\]
\item\label{en:abstract_assheat} Heat kernel bounds: for all $N>0$, $r>0$ and $z,z' \in X$,
\[
|\Kern_{e^{-r^k \opL}}(z,z')| \lesssim_N \mu(B(z',r))^{-1} (1+\dist(z,z')/r)^{-N}.
\]
\item\label{en:abstract_assL1} $L^1$ estimate: for all $s>\varsigma$ and all continuous $\mm : \RR \to \CC$ with $\supp \mm \subseteq [1/4,1]$,
\[
\sup_{t>0}\|\mm(t \opL) \|_{1 \to 1} \lesssim_s \|\mm\|_{\sobolev{s}{q}}.
\]
\end{enumerate}
Then the following estimates hold.
\begin{enumerate}[label=(\roman*)]
\item\label{en:abstract_cpt} For all $s>\varsigma$, all continuous $\mm : \RR \to \CC$ such that $\supp \mm \subseteq [-1,1]$ and all $p \in [1,\infty]$,
\[
\sup_{t > 0} \| \mm(t\opL) \|_{L^p \to L^p} \lesssim_{p,s} \| \mm \|_{\sobolev{s}{q}}.
\]
\item\label{en:abstract_mh} Let $\eta \in C^\infty((0,\infty))$ be nonzero. For all $s>\varsigma$, all $\mm : \RR \to \CC$ continuous on $(0,\infty)$, and all $p \in (1,\infty)$,
\[
\| \mm(\opL) \|_{L^1 \to L^{1,\infty}} \lesssim \sup_{t\geq 0} \| \mm(t \cdot) \eta \|_{\sobolev{s}{q}}, \qquad \| \mm(\opL) \|_{L^p \to L^p} \lesssim_{p} \sup_{t\geq 0}\| \mm(t\cdot) \eta \|_{\sobolev{s}{q}}.
\]\end{enumerate}
\end{thm}

\begin{rem}\label{rem:abstract}
Under the assumptions \ref{en:abstract_assdoubling} and \ref{en:abstract_assheat}, the remaining assumption \ref{en:abstract_assL1} is automatically satisfied for $q=\infty$ and $\varsigma = Q/2$, where $Q$ is the homogeneous dimension of the doubling space $X$, that is,
\begin{equation}\label{eq:doubling_Q}
\mu(B(z,\lambda r)) \lesssim \lambda^Q \mu(B(z,r))
\end{equation}
for all $z \in X$, $\lambda \geq 1$ and $r >0$ (see \eqref{eq:standard_sobemb_lq} below). This leads to the general multiplier theorem with $L^\infty$ condition $s>Q/2$. The point of the above statement is that, if we can prove the $L^1$ estimate \ref{en:abstract_assL1} for some smaller values of $\varsigma$ and/or $q$, then this improvement automatically transfers to the multiplier theorem.
\end{rem}

The strategy used here should be compared to that in, e.g., \cite{cowling_spectral_2001,duong_plancherel-type_2002}, where the role of assumption \ref{en:abstract_assL1} is played by suitable $L^2$ estimates (``Plancherel-type estimates''); the approach in those works, however, does not appear to yield optimal results for the entire class of Grushin operators considered here. In addition, those works require Gaussian-type heat kernel bounds for $\opL$ (that is, superexponential spatial decay for the heat kernel) or finite propagation speed for the associated wave equation (which implies Gaussian-type bounds \cite{sikora_wave_2004}), while assumption \ref{en:abstract_assheat} only requires polynomial decay (of arbitrary order), as in \cite{hebisch_functional_1995}.

The above result will not be surprising to experts, and the proof is a combination of a number of techniques available in the literature. For the reader's convenience, we include a brief sketch of the proof here.

\begin{proof}[Proof of Theorem \ref{thm:abstract}]
Since $(X,\dist,\mu)$ is doubling by assumption \ref{en:abstract_assdoubling}, it has a homogeneous dimension $Q$ satisfying \eqref{eq:doubling_Q}.
Thanks to the heat kernel bound in assumption \ref{en:abstract_assheat}, we can apply \cite[Theorem 6.1(ii)]{martini_crsphere} to the operator $\opL$ and deduce that
\begin{equation}\label{eq:standard_l1}
\sup_{r>0} \esssup_{z' \in X} \int_X |\Kern_{\mm(r^k \opL)}(z,z')|  \, (1+\dist(z,z')/r)^\alpha \,dz \lesssim_{\alpha,\beta} \|\mm\|_{\sobolev{\beta}{\infty}}.
\end{equation}
for all continuous $\mm : \RR \to \CC$ supported in $[-1,1]$ and all $\alpha \geq 0$ and $\beta > \alpha+Q/2$.
By Sobolev's embedding, this implies that
\begin{equation}\label{eq:standard_sobemb_lq}
\sup_{r>0} \esssup_{z' \in X} \int_X |\Kern_{\mm(r^k \opL)}(z,z')|  \, (1+\dist(z,z')/r)^\alpha \,dz \lesssim_{\alpha,\beta} \|\mm\|_{\sobolev{\beta}{q}}.
\end{equation}
for all continuous $\mm : \RR \to \CC$ supported in $[-1,1]$ and all $\alpha \geq 0$ and $\beta > \alpha+Q/2+1/q$.
On the other hand, the assumption \ref{en:abstract_assL1} can be rewritten as
\[
\sup_{r>0} \esssup_{z' \in X} \int_X |\Kern_{\mm(r^k \opL)}(z,z')| \,dz \lesssim_{s} \|\mm\|_{\sobolev{s}{q}}.
\]
for all $s > \varsigma$ and all continuous $\mm : \RR \to \CC$ supported in $[1/4,1]$. Interpolation of the last two estimates (cf.\ \cite[proof of Lemma 1.2]{mauceri_vectorvalued_1990}) gives
\begin{equation}\label{eq:improved_l1_off0}
\sup_{r>0} \esssup_{z' \in X} \int_X |\Kern_{\mm(r^k \opL)}(z,z')| \, (1+\dist(z,z')/r)^\alpha \,dz \lesssim_{\alpha,\beta} \|\mm\|_{\sobolev{\beta}{q}}.
\end{equation}
for all continuous $\mm : \RR \to \CC$ supported in $[1/4,1]$ and all $\alpha \geq 0$ and $\beta > \alpha+\varsigma$.

To prove part \ref{en:abstract_cpt}, by duality and interpolation it is enough to discuss the case $p=1$. However this estimate readily reduces to the case $\alpha=0$ of the estimates \eqref{eq:standard_l1} and \eqref{eq:improved_l1_off0} via a dyadic decomposition (cf.\ \cite[eqs.\ (7.10)-(7.12)]{casarino_ultrasphericalgrushin}).

As for part \ref{en:abstract_mh}, again by duality and interpolation we only need to prove the weak type $(1,1)$ bound. For this, we can employ the boundedness result for singular integral operators in \cite[Theorem 1]{duong_mcintosh_1999}. Indeed, thanks to the assumption \ref{en:abstract_assheat}, we can use the heat propagator $e^{-t^k \opL}$ as the ``approximate identity'' $A_t$ in \cite[Theorem 1]{duong_mcintosh_1999}. Moreover, from \eqref{eq:improved_l1_off0} we obtain that for all $s > \varsigma$ there exists $\epsilon>0$ such that
\[
\esssup_{z' \in X} \int_{X \setminus B(z',t)} |\Kern_{\mm(\opL) (1-A_t)}(z,z')| \,dz \lesssim_{s} \frac{\min\{1,(Rt)^k\}}{(1+Rt)^{\epsilon}} \|\mm(R^{k} \cdot)\|_{\sobolev{s}{q}} 
\]
for all $R,t>0$ and all continuous $\mm : \RR \to \CC$ supported in $[R^k/4,R^k]$. Now, for an arbitrary $\mm : \RR \to \CC$ continuous on $(0,\infty)$, a dyadic decomposition (cf.\ \cite[eqs.\ (4.18)-(4.19)]{duong_plancherel-type_2002}) finally yields the required off-diagonal bound for $\Kern_{\mm(\opL)}$ \cite[eq.\ (7)]{duong_mcintosh_1999}.
\end{proof}

\section{Grushin operators and their geometry}\label{s:grushin}

\subsection{Summary of the results}

Let $V : \RR \to \Rnon$ be continuous and not identically zero.
Let $\opL$ be the Grushin operator on $\RR^2$ associated to $V$, that is, the unbounded second-order differential operator defined on the domain $C^\infty_c(\RR^2)$ by
\begin{equation}\label{eq:grushin_def}
\opL=-\partial_x^2-V(x)\partial_y^2.
\end{equation}
Of course, $\opL : C^\infty_c(\RR^2) \to L^2(\RR^2)$. 

\begin{prp}[Essential self-adjointness]\label{prp:grushin_selfadj}
The Grushin operator $\opL$ just defined is essentially self-adjoint on $L^2(\RR^2)$.
\end{prp}

In light of the above result, $\opL$ has a unique self-adjoint extension, which we will still denote by $\opL$.

As $\opL$ is a nonnegative second-order differential operator on $\RR^2$, its symbol, thought of as a quadratic form on the cotangent space, defines a degenerate Riemannian cometric on $\RR^2$, which in turn induces a distance function $\dist$ on $\RR^2$ (see Section \ref{ss:controldistance} below for details), also known as the \emph{control distance} for $\opL$.

An important property relating the operator $\opL$ and the control distance $\dist$ is the \emph{finite propagation speed} for solutions of the wave equation associated to $\opL$.

\begin{prp}[Finite propagation speed]\label{prp:fps}
For all $t \in \RR$ and $f \in L^2(\RR^2)$,
\[
\supp \cos(t \sqrt{\opL}) f \subseteq \{ z \in \RR^2 \tc \dist(z,\supp f) \leq |t| \}.
\]
\end{prp}

For the next results, we assume
that $V : \RR \to \Rnon$ is continuous, not identically zero, and satisfies the estimates
\begin{equation}\label{eq:V_doubling}
V(-x) \simeq V(x) \leq V(\lambda x) \lesssim \lambda^D V(x)
\end{equation}
for some $D > 0$ and all $x \in \RR$ and $\lambda \geq 1$.

Under the assumption \eqref{eq:V_doubling}, we can give a precise estimate of the distance $\dist$ associated to $\opL$. To this purpose, we introduce the auxiliary function $U : \Rnon \to \Rnon$, defined by
\begin{equation}\label{eq:U_def}
U(t) = |\{ x \in \RR \tc |x|V(x)^{1/2} \leq t \}| .
\end{equation}
We point out that the implicit constants in the estimates below may depend on $D$ and the implicit constants in \eqref{eq:V_doubling}.

\begin{prp}[Distance and volume estimates]\label{prp:distance}
Assume that $V$ satisfies \eqref{eq:V_doubling} for some $D > 0$. Then
\begin{equation}\label{eq:distance}
\dist(z,z') \simeq |x-x'| + \min \left\{ \frac{|y-y'|}{\max\{V(x),V(x')\}^{1/2}} , U(|y-y'|) \right\}.
\end{equation}
for all $z=(x,y),z'=(x',y') \in \RR^2$, where $U$ is as in \eqref{eq:U_def}.
In particular, the volume $\Vol(z,r)$ of the ball centred at $z =(x,y) \in \RR^2$ and radius $r \geq 0$ satisfies
\begin{equation}\label{eq:volume}
\Vol(z,r) \simeq r^2 \max\{V(r),V(x)\}^{1/2}.
\end{equation}
Consequently
\begin{equation}\label{eq:vol_doubling}
\Vol(z,\lambda r) \lesssim \lambda^{Q} \Vol(z,r)
\end{equation}
for all $r \geq 0$, $z \in \RR^2$ and $\lambda \geq 1$, where
\begin{equation}\label{eq:homdim}
Q \defeq 2+D/2.
\end{equation}
\end{prp}

The estimate \eqref{eq:vol_doubling} tells us in particular that $\RR^2$ with the control distance $\dist$ and the Lebesgue measure is a doubling metric measure space of homogeneous dimension $Q$ given by \eqref{eq:homdim} (see Remark \ref{rem:abstract}).

Define the family of weights $\weight_r : \RR^2 \times \RR^2 \to \Rnon$ by
\begin{equation}\label{eq:weight}
\weight_r(z,z') = \frac{|y-y'|}{r \max\{V(r),V(x')\}^{1/2}} 
\end{equation}
for all $r > 0$ and $z=(x,y),z'=(x',y') \in \RR^2$. Effectively these are just the weight $|y-y'|$, suitably scaled in accordance with the volume growth of balls described in Proposition \ref{prp:distance}.

A crucial property of these weights is the following integrability property, relating them to the distance $\dist$. 
Again, the implicit constants in the estimates below may depend on $D$ and the implicit constants in \eqref{eq:V_doubling}.

\begin{prp}[Weight estimates]\label{prp:weights}
Under the same assumptions as in Proposition \ref{prp:distance}, 
the weights $\weight_r$ satisfy the estimate
\begin{equation}\label{eq:weight_int}
\int_{\RR^2} (1+\dist(z,z')/r)^{-\alpha} (1+\weight_r(z,z'))^{-\beta}  \,dz \lesssim_{\alpha,\beta} \Vol(z',r)
\end{equation}
for all $z' \in \RR^2$ and $r > 0$, whenever $\beta \in [0,1)$ and $\alpha > 1+(1+D/2)(1-\beta)$.
\end{prp}

We note that, as $\beta$ grows from $0$ to $1$, the condition on $\alpha$ in the previous lemma goes from $\alpha > Q$ to $\alpha>1$, and correspondingly the lower bound to $\alpha+\beta$ goes from $Q$ to $2$. This indicates how the introduction of these weights is instrumental in pushing down the smoothness condition in the multiplier theorem from half the homogeneous dimension to half the topological dimension.

\begin{rem}
All the above results are already known under different assumptions on $V$. For example, when $V$ is smooth, the finite propagation speed property is contained in the results of \cite{melrose1986}, and also of \cite{cowling_martini_2013}. Moreover both Propositions \ref{prp:fps} and \ref{prp:distance} are proved in \cite{robinson_analysis_2008} in the case $V$ is comparable to a function of the form \eqref{eq:twopowerlaw}. Finally Proposition \ref{prp:weights} is proved in \cite{martini_sharp_2014} in the case $V(x) = x^2$.
\end{rem}

The rest of the section is devoted to the proofs of the above results.

\subsection{Essential self-adjointness}

Here we prove Proposition \ref{prp:grushin_selfadj}, that is, the essential self-adjointness of $\opL : C_c^\infty(\RR^2) \to L^2(\RR^2)$. The proof presented here works more generally for any nonnegative, locally square-integrable potential $V$ on $\RR$.

By a well-known criterion \cite[Corollary to Theorem VIII.3]{reed-simon-I}, it is enough to check that $\opL^*-\overline{\lambda}$ is injective when $\lambda\in \CC\setminus \RR$. Here $\opL^*$ denotes the Hilbert space adjoint. In other words, we have to show that a function $f\in L^2(\RR^2)$ such that 
\begin{equation*}
\int_{\RR^2} f(x,y) \, \overline{\left(-\partial_x^2-V(x)\partial_y^2-\lambda\right)\varphi(x,y)} \,dx\,dy = 0   \qquad\forall \varphi\in C^\infty_c(\RR^2),
\end{equation*} 
must necessarily vanish. The choice $\varphi(x,y)=u(x)v(y)$ for $u,v\in C^\infty_c(\RR)$ gives
\[\begin{split}
0 
&= \int_{\RR^2} f(x,y) \, \overline{\left(-\partial_x^2-V(x)\partial_y^2-\lambda\right)\varphi(x,y)} \,dx \,dy \\
&= \int_{\RR} \left( \int_{\RR} \widetilde{f}(x,\xi) \, \overline{(\opH[\xi^2 V]-\lambda) u(x)} \,dx \right) \overline{\widehat{v}(\xi)} \,d\xi,
\end{split}\]
where $\opH[\xi^2 V] \defeq -\partial_x^2+\xi^2 V(x)$ is the Schr\"odinger operator with potential $\xi^2 V$ on $\RR$, while $\hat v$ denotes the Fourier transform of $v$ (with a suitable normalisation) and $\widetilde{f}$ is the partial Fourier transform of $f$ in the variable $y$.
Hence, for all $u \in C^\infty_c(\RR)$,
the locally integrable function
\[
\xi\mapsto \int_{\RR} \widetilde{f}(x,\xi) \, \overline{(\opH[\xi^2 V]-\lambda)u(x)} \,dx
\]
vanishes as a temperate distribution, and consequently it vanishes almost everywhere, that is,
\begin{equation*}
\int_{\RR} \widetilde{f}(x,\xi) \, \overline{(\opH[\xi^2 V]-\lambda) u(x)} \,dx = 0 \qquad\forall u\in C^\infty_c(\RR), \ \text{for a.e. } \xi\in \RR.
\end{equation*}
Since the Schr\"odinger operator $\opH[\xi^2 V]$ is essentially self-adjoint on $L^2(\RR)$ for every $\xi\in\RR$ \cite[Theorem X.28]{reed-simon-II}, $\widetilde{f}(\cdot,\xi)$ must vanish as an element of $L^2(\RR)$ for almost every $\xi$. This immediately implies that $f=0$ almost everywhere.

\subsection{The control distance}\label{ss:controldistance}

Let $V : \RR \to \Rnon$ be continuous and not identically zero.
The Grushin operator $\opL$ associated to $V$ defined in \eqref{eq:grushin_def} is a second-order differential operator on $\RR^2$, whose symbol is a nonnegative quadratic form on the cotangent space of $\RR^2$. This quadratic form can be thought of as a degenerate Riemannian cometric on $\RR^2$:
\[
|(\xi,\eta)|^*_{(x,y)} = \sqrt{|\xi|^2 + V(x) |\eta|^2}
\]
for $(\xi,\eta) \in \RR^2  \cong T^*_{(x,y)} \RR^2$. Polarisation gives the degenerate Riemannian metric
\[
|v|_{(x,y)} = \sup \{ \omega(v) \tc \omega \in \RR^2, \, |\omega|^*_{(x,y)} \leq 1 \}
\]
for all $v \in \RR^2 \cong T_{(x,y)} \RR^2$. This in turn induces a distance $\dist$ on $\RR^2$, by setting
\[
\dist(z,z') = \inf \{ r > 0 \tc \exists \gamma \in AC([0,r]; \RR^2), \, |\gamma'|_{\gamma} \leq 1 \text{ a.e.}, \gamma(0) = z, \, \gamma(r) = z'\}
\]
for all $z,z' \in \RR^2$. Here $AC([0,r]; \RR^2)$ is the set of all absolutely continuous curves $\gamma : [0,r] \to \RR^2$; those curves that satisfy $|\gamma'|_{\gamma} \leq 1$ a.e.\ are called \emph{subunit curves}.

For more details on this construction, we refer the reader to \cite[Section 4]{cowling_martini_2013}, where a more general theory is developed for continuous cometrics on manifolds. In particular, we recall that the topology induced by $\dist$ is as least as fine as the Euclidean topology, and that a curve $\gamma : [0,T] \to \RR^2$ is subunit if and only if it is $1$-Lipschitz with respect to $\dist$ \cite[Proposition 4.6]{cowling_martini_2013}.

We remark that, since $V$ is not identically zero, any pair of points $z,z' \in \RR^2$ can be joined by a subunit curve (made of suitable line segments parallel to the coordinate axes), so $\dist(z,z')$ is finite. Moreover, if $\gamma : [0,r]  \to \RR^2$ is subunit, then
\begin{equation}\label{eq:subunit_bound}
\gamma([0,r]) \subseteq R_V(\gamma(t),r) \qquad \text{for all } t \in [0,r],
\end{equation}
where
\begin{equation}\label{eq:rectangle_V}
R_V(z,r) \defeq [x-r,x+r] \times \left[y-r \max_{[x-r,x+r]} V^{1/2}, y+r\max_{[x-r,x+r]} V^{1/2}\right]
\end{equation}
for all $z=(x,y) \in \RR^2$; this is easily seen by observing that, if $|(v_1,v_2)|_{(x,y)} \leq 1$, then $|v_1| \leq 1$ and $|v_2| \leq V(x)^{1/2}$, and by applying these estimates to the components of the velocity $\gamma'$ of a subunit curve $\gamma$. From this remark and the local boundedness of $V$ one easily deduces that $\dist$-bounded sets in $\RR^2$ are also Euclidean-bounded, and also that $\dist$ is a complete metric on $\RR^2$ \cite[Proposition 4.16]{cowling_martini_2013}.

We now state an equivalent characterisation of the distance $\dist$. In the statement, $W^{1,\infty}(\RR^2;\RR)$ denotes the space of bounded real-valued Lipschitz functions on $\RR^2$.

\begin{prp}\label{prp:distance_gradient_char}
For all $z,z' \in \RR^2$,
\begin{multline}\label{eq:distance_gradient_char}
\dist(z,z') \\
= \sup \{ \psi(z)-\psi(z') \tc \psi \in W^{1,\infty}(\RR^2;\RR), \, (\partial_x \psi)^2 + V(x) (\partial_y \psi)^2 \leq 1 \text{ a.e.} \}.
\end{multline}
\end{prp}

Similar equivalent characterisations of control distances associated to second-order differential operators can be found in many places in the literature (see, e.g., \cite{jerison_sanchez_calle_1987,garofalo_nhieu_1998,cowling_martini_2013}), though often with additional smoothness assumptions on the coefficients, and with different choices of the space of ``test functions'' $\psi$. For this reason, we give a proof tailored to our case.

\begin{proof}
Let $\dist_*(z,z')$ denote the right-hand side of \eqref{eq:distance_gradient_char}.
We would like to show that $\dist_* = \dist$.

If $\gamma = (\gamma_1,\gamma_2) \in AC([0,r];\RR^2)$ is a subunit curve joining $z$ to $z'$, then
\[\begin{split}
|\psi(z)-\psi(z')| &= |\psi(\gamma(r))-\psi(\gamma(0))| \leq \int_{0}^r |(\psi\circ \gamma)'(t)| \,dt \\
&= \int_{0}^r | \partial_x \psi(\gamma(t)) \gamma_1'(t) + \partial_y \psi(\gamma(t)) \gamma_2'(t) | \,dt \leq r
\end{split}\]
whenever $\psi \in C^1(\RR^2;\RR)$ satisfies $(\partial_x \psi)^2 + V(x) (\partial_y \psi)^2 \leq 1$. Thanks to the approximation result in \cite[Corollary 3.5]{cowling_martini_2013}, the same inequality extends to all $\psi \in W^{1,\infty}_\loc(\RR^2;\RR)$ satisfying $(\partial_x \psi)^2 + V(x) (\partial_y \psi)^2 \leq 1$. Hence $\dist_*(z,z') \leq \dist(z,z')$ for all $z,z' \in \RR^2$.

It remains to discuss the opposite inequality.
Let us first show 
that, if $\psi \in W^{1,\infty}_\loc(\RR^2;\RR)$ is $1$-Lipschitz with respect to $\dist$, then $(\partial_x \psi)^2 + V(x) (\partial_y \psi)^2 \leq 1$ almost everywhere. Note that $\psi$ is locally Lipschitz, so by the Rademacher theorem it is almost everywhere differentiable. Let $z \in \RR^2$ be such that $\psi$ is differentiable at $z$; to conclude it will be enough to show that $|\alpha \partial_x \psi(z) + \beta V(x)^{1/2} \partial_y \psi(z)| \leq 1$ for all unit vectors $(\alpha,\beta) \in \RR^2$. Let $Z = \alpha\partial_x + \beta V(x)^{1/2}\partial_y$; this is a vector field on $\RR^2$ with continuous coefficients, hence by Peano's Theorem there exist $\delta>0$ and an integral curve $\gamma : (-\delta,\delta) \to \RR^2$ of $Z$ (of class $C^1$) such that $\gamma(0) = z$. Since $\psi$ is differentiable at $z$, by the Chain Rule,
\[
Z\psi(z) = \left.\frac{d}{dt}\right|_{t=0} \psi(\gamma(t)) = \lim_{t \to 0} \frac{\psi(\gamma(t))-\psi(\gamma(0))}{t}.
\]
Since $\psi$ is $1$-Lipschitz with respect to $\dist$ and $\gamma$ is subunit, the quotient in the right-hand side is bounded in absolute value by one, hence we conclude that $|Z\psi(z)| \leq 1$ as desired.

We now consider, for $\epsilon \in (0,1]$, distances $\dist_\epsilon$ associated with the cometrics
\[
|(\xi,\eta)|^{*,\epsilon}_{(x,y)} = \sqrt{|\xi|^2+ (\epsilon+V(x)) |\eta|^2 }.
\]
It is easily checked that the induced Riemannian distances $\dist_\epsilon$ on $\RR^2$ are locally bi-Lipschitz-equivalent with the Euclidean distance. 
Moreover clearly
\[
\dist_{\epsilon'} \leq \dist_\epsilon \leq \dist \quad\text{whenever } 0 < \epsilon \leq \epsilon' \leq 1.
\]
Furthermore, by using the argument of \cite[proof of Proposition 5.9]{cowling_martini_2013} (see also \cite[proof of Proposition 1.3]{jerison_sanchez_calle_1987}), one can show that 
\[
\dist(z,z') = \sup_{\epsilon \in (0,1]} \dist_\epsilon(z,z').
\]
The idea is that quasi-minimising subunit curves for $\dist_\epsilon$ joining two fixed points $z,z' \in \RR^2$ stay in a fixed compact set (one can take $R_{1+V}(z,2\dist(z,z'))$, as defined in \eqref{eq:rectangle_V}); so, by the Ascoli--Arzel\`a Theorem, one can form a convergent sequence of such curves, including $\dist_\epsilon$-subunit curves for arbitrary small $\epsilon>0$, and the limit curve is easily proved to be subunit for $\dist$.

Now, for fixed $z' \in \RR^2$, the function $\dist_\epsilon(\cdot,z')$ is in $W^{1,\infty}_\loc(\RR^2)$ and is $1$-Lipschitz with respect to $\dist$. Consequently, for given $z,z' \in \RR^2$, the function $\psi(w) = (\dist_\epsilon(z,z')-\dist_\epsilon(w,z'))_+$ is in $W^{1,\infty}(\RR^2)$ and is $1$-Lipschitz with respect to $\dist$, and therefore $(\partial_x \psi)^2 + V(x) (\partial_y \psi)^2 \leq 1$; from the definition of $\dist_*$ we then obtain
\[
\dist_*(z,z') \geq \psi(z') - \psi(z) = \dist_\epsilon(z,z'),
\]
and by taking the supremum in $\epsilon$ we finally obtain the inequality $\dist_* \geq \dist$.
\end{proof}

\subsection{Finite propagation speed}

Thanks to the characterisation of the control distance $\dist$ in Proposition \ref{prp:distance_gradient_char}, we can now apply the results of \cite[Section 4]{robinson_analysis_2008} to obtain that the Grushin operator $\opL$ satisfies the finite propagations speed property for the associated wave equation, thus proving Proposition \ref{prp:fps}.

Indeed, the bilinear form
\begin{equation}\label{eq:dirichlet}
\formE(f,g) = \langle f, \opL g \rangle = \int_{\RR^2} (\partial_x f) (\partial_x g) + V(x) (\partial_y f) (\partial_y g) \,dx\,dy,
\end{equation}
initially defined on $C_c^\infty(\RR^2)$, is a strong local, regular Markovian symmetric form on $L^2(\RR^2)$ \cite[Example 1.2.1, page 6]{fukushima_oshima_takeda_2011}, which is easily seen to be closable \cite[Exercise 1.1.2, page 4]{fukushima_oshima_takeda_2011}, hence its closure (which we still denote by $\formE$) is a strong local, regular Dirichlet form \cite[Theorems 3.1.1 and 3.1.2 and Exercise 3.1.1, pages 109-111]{fukushima_oshima_takeda_2011}. In view of the correspondence between closed symmetric forms and nonnegative definite self-adjoint operators \cite[Theorem 1.3.1 and Corollary 1.3.1, pages 20-21]{fukushima_oshima_takeda_2011} and the essential self-adjointness of $\opL$ (Proposition \ref{prp:grushin_selfadj}), we conclude that the Dirichlet form $\formE$ is the form associated with the self-adjoint operator $\opL$, so the domain of $\formE$ is the domain of $\sqrt{\opL}$ and the first identity in \eqref{eq:dirichlet} holds for all $f$ in the domain of $\formE$ and $g$ in the domain of $\opL$.

Thanks to the just-discussed closability of $\formE$, its closure coincides with its relaxation $\formE_0$ (as defined in \cite[Section 4]{robinson_analysis_2008}), hence \cite[Proposition 4.1]{robinson_analysis_2008} gives Davies--Gaffney estimates for the heat semigroup associated with $\opL$:
\[
|\langle f, e^{-t\opL} g \rangle| \leq e^{-\dist(A,B)^2/(4t)} \|f\|_2 \|g\|_2
\]
for all $t>0$, open sets $A,B \subseteq \RR^2$, and $f,g \in L^2(\RR^2)$ such that $\supp f \subseteq A$ and $\supp g \subseteq B$. In turn, Davies--Gaffney estimates are equivalent to the desired finite propagation speed property \cite{sikora_wave_2004,coulhon_sikora_2008}.

\subsection{The auxiliary function \texorpdfstring{$U$}{U}}

We record here some useful estimates involving the auxiliary function $U$ defined in \eqref{eq:U_def}. The implicit constants in the estimates below may depend on the constants in \eqref{eq:V_doubling}.

\begin{lem}
Assume that $V$ satisfies \eqref{eq:V_doubling} for some $D>0$. Then the function $U$ defined in \eqref{eq:U_def} is continuous, increasing and invertible. Moreover, for all $t,r \in [0,\infty)$ and $\lambda \geq 1$,
\begin{equation}\label{eq:U_iden}
U(t) \, V(U(t))^{1/2} \simeq t, \qquad U(r V(r)^{1/2}) \simeq r
\end{equation}
and
\begin{equation}\label{eq:U_doubling}
\lambda U(t) \geq U(\lambda t) \gtrsim \lambda^{1/(1+D/2)} U(t).
\end{equation}
\end{lem}
\begin{proof}
Define $W_\pm : \Rnon \to \Rnon$ by
\[
W_\pm(x) = x V(\pm x)^{1/2} .
\]
Then, by \eqref{eq:V_doubling}, the functions $W_\pm$ are not identically zero and satisfy
\begin{gather}
W_\pm(x) \leq \kappa W_\mp(x), \label{eq:W_parity}\\
 \lambda W_\pm(x) \leq W_\pm(\lambda x) \leq \kappa \lambda^{1+D/2} W_\pm(x) \label{eq:W_doubling}
\end{gather}
for all $x \in \Rnon$ and $\lambda \geq 1$, where $\kappa \geq 1$ depends on the implicit constants in \eqref{eq:V_doubling}. In particular, the $W_\pm : \Rnon \to \Rnon$ are strictly increasing and invertible. Moreover
\[
U(t) =  |\{ W_+ \leq t \}| + |\{ W_- \leq t \}| = W_+^\inv(t) + W_-^\inv(t),
\]
which proves that $U$ is continuous, increasing and invertible too.

Now, for all $t \geq 0$ and $\lambda \geq 1$, if $x = W_\pm^\inv(t)$, then
\begin{equation}\label{eq:Winv_first}
W_\pm^\inv(\lambda t) = W_\pm^\inv(\lambda W_\pm(x)) \leq W_\pm^\inv (W_\pm(\lambda x)) = \lambda x = \lambda W_\pm^\inv(t),
\end{equation}
where we used the first inequality in \eqref{eq:W_doubling} and the fact that $W_\pm^\inv$ is increasing. Similarly
\begin{multline}\label{eq:Winv_second}
\lambda^{1/(1+D/2)} W_\pm^\inv(t) = W_\pm^\inv(W_\pm(\lambda^{1/(1+D/2)} x)) \\
\leq W_\pm^\inv(\kappa \lambda W_\pm(x)) \leq \kappa W_\pm^\inv(\lambda t),
\end{multline}
where we used the second inequality in \eqref{eq:W_doubling} and \eqref{eq:Winv_first}. Together, \eqref{eq:Winv_first} and \eqref{eq:Winv_second} give
\begin{equation}\label{eq:Winv_together}
\kappa^{-1} \lambda^{1/(1+D/2)} W_\pm^\inv(t) \leq W_\pm^\inv(\lambda t) \leq \lambda W_\pm^\inv(t).
\end{equation}
By summing the estimates in \eqref{eq:Winv_together}, we obtain 
\eqref{eq:U_doubling}.

Moreover, by \eqref{eq:W_parity} and \eqref{eq:Winv_first},
\[
W_\pm^\inv(t) = |\{ W_\pm \leq t\}| \leq |\{ W_\mp \leq \kappa t \}| = W_\mp^\inv(\kappa t) \leq \kappa W_\mp^\inv(t),
\]
which implies that
\begin{equation}\label{eq:U_Winv}
W_\pm^\inv(t) \leq U(t) \leq (1+\kappa) W_\pm^\inv(t).
\end{equation}

Consequently, for all $t,r \geq 0$, by \eqref{eq:W_doubling} and \eqref{eq:U_Winv},
\[
U(t) V(U(t))^{1/2} = W_+(U(t)) \simeq W_+(W_+^\inv(t)) = t
\]
and
\[
U(rV(r)^{1/2}) = U(W_+(r)) \simeq W_+^\inv(W_+(r)) = r,
\]
which proves \eqref{eq:U_iden}.
\end{proof}

\subsection{Distance and volume estimates}
Here we prove Proposition \ref{prp:distance}, under the assumption that $V$ satisfies \eqref{eq:V_doubling} for some $D >0$.

We start by proving the distance estimate
\begin{equation}\label{eq:distance_bis}
\dist(z,z') \simeq |x-x'| + \min \left\{ \frac{|y-y'|}{\max\{V(x),V(x')\}^{1/2}} , U(|y-y'|) \right\}.
\end{equation}

Let us first prove the inequality $\lesssim$.
In view of \eqref{eq:U_iden}, we are reduced to proving that
\[
\dist(z,z') \lesssim |x-x'|+ \frac{|y-y'|}{\max\{V(x),V(x'),V(U(|y-y'|))\}^{1/2}}. 
\]

If the maximum in the denominator is $V(x)$, then the upper bound to $\dist(z,z')$ is given by the length of the concatenation of the line segments through the points
\[
(x,y) \rightsquigarrow (x,y') \rightsquigarrow (x',y'),
\]
that is, $|y-y'|/V(x)^{1/2} + |x-x'|$.

Similarly, if the maximum in the denominator is $V(x')$, then we consider the concatenation of the line segments through the points
\[
(x,y) \rightsquigarrow (x',y) \rightsquigarrow (x',y'),
\]
whose length is $|x-x'|+|y-y'|/V(x')^{1/2}$.

Finally, if the maximum in the denominator is $V(U(|y-y'|))$, then $U(|y-y'|) \gtrsim |x|+|x'| \geq |x-x'|$ by \eqref{eq:V_doubling}. Hence, if we set $x_* = U(|y-y'|)$, then the concatenation of the line segments through the points
\[
(x,y) \rightsquigarrow (x_*,y) \rightsquigarrow (x_*,y') \rightsquigarrow (x',y')
\]
has length $|x-x_*| + |y-y'|/V(x_*)^{1/2} + |x_*-x'| \simeq U(|y-y|) \simeq |y-y|/V(U(|y-y'|))^{1/2}$ by \eqref{eq:U_iden}.

We now prove the inequality $\gtrsim$ in \eqref{eq:distance_bis}.
Let us consider an arbitrary subunit curve $\gamma = (\gamma_1,\gamma_2) \in AC([0,r];\RR^2)$ joining $z$ to $z'$. Since $\gamma$ is subunit, we deduce that
\[
|\gamma_1'(t)| \leq 1, \qquad |\gamma_2'(t)| \leq V(\gamma_1(t))^{1/2}.
\]

Let $t_* \in [0,r]$ be such that $|\gamma_1(t_*)| = \max_{[0,r]} |\gamma_1|$, and set $x_* = \gamma_1(t_*)$. Then
\begin{equation}\label{eq:subunit_dist_est1}
|x-x'| \leq |x-x_*|+|x_*-x'| \leq \int_0^r |\gamma_1'(t)| \,dt \leq r
\end{equation}
and
\begin{equation}\label{eq:subunit_dist_est2}
|y-y'| \leq \int_0^r |\gamma_2'(t)| \,dt \leq r \max_{t \in [0,r]} V(\gamma_1(t))^{1/2} \simeq r V(x_*)^{1/2}.
\end{equation}

If $|x_*| \leq 2\max\{|x|,|x'|\}$, then $V(x_*) \simeq \max\{V(x),V(x')\}$ and the previous inequalities \eqref{eq:subunit_dist_est1} and \eqref{eq:subunit_dist_est2} imply that
\[
r \gtrsim |x-x'| + \frac{|y-y'|}{\max\{V(x),V(x')\}^{1/2}}.
\]
If instead $|x_*| \geq 2\max\{|x|,|x'|\}$, then $|x-x_*|+|x_*-x'| \geq |x_*|$, whence, by \eqref{eq:subunit_dist_est1} and \eqref{eq:subunit_dist_est2},
\[
r \gtrsim |x-x'|+ |x_*| + \frac{|y-y'|}{V(x_*)^{1/2}} \gtrsim |x-x'|+ \min_{t \geq 0} \left( t + \frac{|y-y'|}{V(t)^{1/2}} \right) \simeq |x-x'| + U(|y-y'|),
\]
where \eqref{eq:U_iden} was used in the last step.

So in any case
\[
r \gtrsim |x-x'|+\min\left\{ \frac{|y-y'|}{\max\{V(x),V(x')\}^{1/2}}, U(|y-y'|) \right\} .
\]
Since $\dist(z,z')$ is the infimum of such $r$, this proves the inequality $\gtrsim$ in \eqref{eq:distance_bis}.

To conclude the proof of Proposition \ref{prp:distance}, we must prove the volume estimates 
\begin{align}
\Vol(z,r) &\simeq r^2 \max\{V(r),V(x)\}^{1/2}, \label{eq:volume_bis} \\
\Vol(z,\lambda r) &\lesssim \lambda^{Q} \Vol(z,r), \label{eq:vol_doubling_bis}
\end{align}
where $z = (x,y) \in \RR^2$ and $Q$ is as in \eqref{eq:homdim}.

In order to prove \eqref{eq:volume_bis} from \eqref{eq:distance_bis}, it is enough to show that the ball of centre $z=(x,y)$ and radius $r\geq 0$ is comparable to the rectangle
\[
(x,y)+[-r,r] \times [-r\max\{V(x),V(r)\}^{1/2},r\max\{V(x),V(r)\}^{1/2}],
\]
in the sense that the ball is contained and contains suitably scaled versions of the above rectangle (with respect to the centre). Indeed, from \eqref{eq:distance_bis} we obtain that
\begin{equation}\label{eq:ball_dist}
\dist(z,z') \lesssim r
\end{equation}
if and only if both
\begin{equation}\label{eq:ball_x}
|x-x'| \lesssim r
\end{equation}
and
\begin{equation}\label{eq:ball_y}
|y-y'| \lesssim r \max\{V(x),V(x')\}^{1/2} \quad\text{or}\quad U(|y-y'|) \lesssim r.
\end{equation}
In view of \eqref{eq:U_iden} and \eqref{eq:U_doubling}, the condition \eqref{eq:ball_y} is equivalent to
\begin{equation}\label{eq:ball_y2}
|y-y'| \lesssim r \max\{V(x),V(x'),V(r)\}^{1/2}.
\end{equation}
However, in view of \eqref{eq:V_doubling}, under the condition \eqref{eq:ball_x} we also have
\[
\max\{V(x),V(x'),V(r)\} \simeq \max\{V(x),V(r)\}
\]
and therefore, under \eqref{eq:ball_x}, the condition \eqref{eq:ball_y2} is equivalent to
\begin{equation}\label{eq:ball_y3}
|y-y'| \lesssim r \max\{V(x),V(r)\}^{1/2}.
\end{equation}
In conclusion, \eqref{eq:ball_dist} is equivalent to the conjunction of \eqref{eq:ball_x} and \eqref{eq:ball_y3}, which proves the statement about comparability of balls and rectangles, and therefore \eqref{eq:volume_bis}.

Finally, from \eqref{eq:volume_bis} and \eqref{eq:V_doubling} we deduce that, for all $\lambda \geq 1$,
\[\begin{split}
\Vol(z,\lambda r) &\simeq (\lambda r)^2 \max\{V(\lambda r),V(x)\}^{1/2} \\
&\lesssim \lambda^{2+D/2} r^2 \max\{V(r),V(x)\}^{1/2} \\
&\simeq \lambda^Q \Vol(z,r),
\end{split}\]
which proves \eqref{eq:vol_doubling_bis}.

\subsection{Weight estimates}
Here we assume again that $V$ satisfies \eqref{eq:V_doubling}, and prove Proposition \ref{prp:weights} about the family of weights
$\weight_r : \RR^2 \times \RR^2 \to [0,\infty)$ defined in \eqref{eq:weight}.

According to Proposition \ref{prp:distance},
\[
\dist(z,z') \simeq \min \{ \dist_1(z,z'), \dist_2(z,z') \},
\]
where
%\begin{equation*}
%\begin{aligned}
%\dist_1(z,z') &\defeq |x-x'|+\frac{|y-y'|}{V(x,x')^{1/2}}, \\
%\dist_2(z,z') &\defeq |x-x'|+ U(|y-y'|) ,\\
%V(x,x') &\defeq \max\{V(x),V(x')\}.
%\end{aligned}
%\end{equation*}
\begin{equation*}
\dist_1(z,z') \defeq |x-x'|+\frac{|y-y'|}{V(x,x')^{1/2}}, \qquad \dist_2(z,z') \defeq |x-x'|+ U(|y-y'|) 
\end{equation*}
and
\begin{equation*}
V(x,x') \defeq \max\{V(x),V(x')\}.
\end{equation*}
Hence it is enough to prove the desired estimate
\eqref{eq:weight_int} with $\dist$ replaced by $\dist_i$ for $i=1,2$.

In the case $i=1$, we note that, under the assumptions $\beta \in [0,1)$ and $\alpha > 1+(1+D/2)(1-\beta)$, we can decompose $\alpha =\alpha'+\alpha''$ so that $\alpha''>1-\beta$ and $\alpha' > 1+ \alpha'' D/2$. Therefore
\[\begin{split}
&\int_{\RR^2} (1+\dist_1(z,z')/r)^{-\alpha} (1+\weight_r(z,z'))^{-\beta} \,dz \\
&\leq \int_{\RR^2} \left(1+\frac{|x-x'|}{r}\right)^{-\alpha'} \left(1+\frac{|y-y'|}{r V(x,x')^{1/2}}\right)^{-\alpha''}
\left(1+\frac{|y-y'|}{r V(r,x')^{1/2}}\right)^{-\beta} \,dz \\
&= r V(r,x')^{1/2} \int_{\RR^2} \left(1+\frac{|x-x'|}{r}\right)^{-\alpha'} 
\left(1+|y|\frac{V(r,x')^{1/2}}{V(x,x')^{1/2}}\right)^{-\alpha''} \left(1+|y|\right)^{-\beta} \,dz \\
&\lesssim_{\alpha,\beta} r V(r,x')^{1/2} \int_{\RR} \left(1+\frac{|x-x'|}{r}\right)^{-\alpha'} 
\left(1+\frac{V(x,x')^{1/2}}{V(r,x')^{1/2}}\right)^{\alpha''} \,dx ,
\end{split}\]
since $\alpha''+\beta>1$.

We now observe that, by \eqref{eq:V_doubling}, $V(x,x') \simeq V(x') + V(|x-x'|)$ and
\[
1+\frac{V(x,x')}{V(r,x')} \lesssim 1 + \frac{V(|x-x'|)}{V(r)} \lesssim \left(1+\frac{|x-x'|}{r}\right)^D.
\]
Hence we can continue the previous series of inequalities and obtain that
\[\begin{split}
&\int_{\RR^2} (1+\dist_1(z,z')/r)^{-\alpha} (1+\weight_r(z,z'))^{-\beta} \,dz \\
&\lesssim r V(r,x')^{1/2} \int_{\RR} \left(1+\frac{|x-x'|}{r}\right)^{-\alpha'+\alpha''D/2} \,dx \\
&\lesssim_{\alpha} r^2 V(r,x')^{1/2},
\end{split}\]
as desired; in the last inequality we used that $\alpha'-\alpha''D/2>1$.

In the case $i=2$, instead, we observe that, under the assumptions $\beta \in [0,1)$ and $\alpha > 1+(1+D/2)(1-\beta)$, we can decompose $\alpha =\alpha'+\alpha''$ so that $\alpha' > 1$ and $\alpha'' > (1+D/2)(1-\beta)$. Hence
\[\begin{split}
&\int_{\RR^2} (1+\dist_2(z,z')/r)^{-\alpha} (1+\weight_r(z,z'))^{-\beta} \,dz \\
&\leq \int_{\RR^2} \left(1+\frac{|x-x'|}{r}\right)^{-\alpha'} \left(1+\frac{U(|y-y'|)}{r}\right)^{-\alpha''} 
\left(1+\frac{|y-y'|}{r V(r,x')^{1/2}}\right)^{-\beta} \,dz \\
&= \int_{\RR^2} \left(1+\frac{|x|}{r}\right)^{-\alpha'} \left(1+\frac{U(|y|)}{r}\right)^{-\alpha''} 
\left(1+\frac{|y|}{r V(r,x')^{1/2}}\right)^{-\beta} \,dz \\
&\lesssim
\left(\frac{V(r,x')}{V(r)}\right)^{\beta/2}  \int_{\RR^2} \left(1+\frac{|x|}{r}\right)^{-\alpha'} 
\left(1+\frac{U(|y|)}{r}\right)^{-\alpha''} \left(1+\frac{|y|}{r V(r)^{1/2}}\right)^{-\beta} \,dz \\
\end{split}\]
by translation-invariance. Now, by \eqref{eq:U_iden} and \eqref{eq:U_doubling},
\[
1+\frac{U(|y|)}{r} \simeq 1+\frac{U(|y|)}{U(rV(r)^{1/2})} \gtrsim \left(1+\frac{|y|}{r V(r)^{1/2}}\right)^{1/(1+D/2)}.
\]
Since $\beta < 1$, from the previous inequalities we deduce that
\[\begin{split}
&\int_{\RR^2} (1+\dist_2(z,z')/r)^{-\alpha} (1+\weight_r(z,z'))^{-\beta} \,dz \\
&\lesssim_\alpha  \left(\frac{V(r,x')}{V(r)}\right)^{1/2} \int_{\RR^2} \left(1+\frac{|x|}{r}\right)^{-\alpha'} 
\left(1+\frac{|y|}{r V(r)^{1/2}}\right)^{-\beta-\alpha''/(1+D/2)} \,dz \\
&\simeq_{\alpha,\beta}
 r^2 V(r,x')^{1/2} ,
\end{split}\]
as $\alpha',\beta+\alpha''/(1+D/2)>1$,
and we are done.

\section{Reduction to weighted Plancherel estimates}\label{s:redweightedplancherel}

In this section we show that the proof of our optimal multiplier theorem (Theorem \ref{thm:main}) for Grushin operators $\opL$ on $\RR^2$ reduces to that of a ``weighted Plancherel estimate'' involving the weights $\weight_r$ from Section \ref{s:grushin}.

\begin{thm}\label{thm:conditional_weightedplancherel}
Let $V : \RR \to \Rnon$ be continuous and satisfy the assumption \eqref{eq:V_doubling}.
Let $\opL$ be the Grushin operator on $\RR^2$ associated to $V$, as in \eqref{eq:grushin_def}. Let the weights $\weight_r$ on $\RR^2$ be defined as in \eqref{eq:weight}.
Let $q \in [2,\infty]$. 
Assume that:
\begin{enumerate}[label=(\Alph*)]
\item\label{en:ass_w_planch} For all $\vartheta \in [0,1/2)$, the estimate
\begin{equation}\label{eq:weightedpl_abstract}
\sup_{r>0} \, \esssup_{z' \in \RR^2} \, \Vol(z',r) \int_{\RR^2} \left|(1+\weight_r(z,z'))^\vartheta \Kern_{\mm(r^2 \opL)}(z,z') \right|^2 \,dz \lesssim_\vartheta \|\mm\|_{\sobolev{\vartheta}{q}}^2, 
\end{equation}
holds for all continuous $\mm : \RR \to \CC$ with $\supp \mm \subseteq [1/4,1]$.
\item\label{en:ass_planch} The estimate analogous to \eqref{eq:weightedpl_abstract} with $\vartheta=0$ and $q=\infty$ also holds for all continuous $\mm : \RR \to \CC$ with $\supp \mm \subseteq [-1,1]$.
\end{enumerate}
Then the two estimates 
\ref{en:abstract_cpt} and \ref{en:abstract_mh}
of Theorem \ref{thm:abstract}
hold for $\opL$ with $\varsigma=2/2$.
\end{thm}

\begin{rem}\label{rem:conditional_weightedplancherel}
In view of Proposition \ref{prp:distance}, the estimate \eqref{eq:weightedpl_abstract} is equivalent to
\begin{multline}\label{eq:equivalentlyrewritten}
\sup_{r>0} \, \esssup_{z' \in \RR^2} \, r^{2-2\vartheta} \max\{V(r),V(x')\}^{1/2-\vartheta} \int_{\RR^2} |y-y'|^{2\vartheta} \left|\Kern_{\mm(r^2 \opL)}(z,z') \right|^2 \,dz \\
\lesssim_\vartheta \|\mm\|_{\sobolev{\vartheta}{\infty}}^2.
\end{multline}
More precisely, the left-hand side of \eqref{eq:weightedpl_abstract} is comparable to the sum of the left-hand side of \eqref{eq:equivalentlyrewritten} and the corresponding one for $\vartheta=0$.
\end{rem}

\begin{proof}[Proof of Theorem \ref{thm:conditional_weightedplancherel}]
From the estimate in assumption \ref{en:ass_planch}, together with a dyadic decomposition, one easily obtains (cf., e.g., \cite[proof of Corollary 4.5]{casarino_grushinsphere}) that
\[
\sup_{r>0} \, \esssup_{z' \in \RR^2} \, \Vol(z',r) \int_{\RR^2} \left| \Kern_{e^{-r^2 \opL}}(z,z') \right|^2 \,dz \lesssim 1. 
\]
This information, combined with the fact that $\dist$ is doubling (Proposition \ref{prp:distance}) and $\opL$ has the finite propagation speed property (Proposition \ref{prp:fps}), implies Gaussian-type heat kernel bounds for $\opL$ \cite{sikora_wave_2004,coulhon_sikora_2008}: there exists $b>0$ such that
\begin{equation}\label{eq:gaussian_heat_kernel}
|\Kern_{e^{-r^2 \opL}}(z,z')| \lesssim \Vol(z',r)^{-1} \exp(-b\,\dist(z,z')^2/r^2).
\end{equation}
for all $r>0$ and $z,z' \in \RR^2$. As a consequence, assumptions \ref{en:abstract_assdoubling} and \ref{en:abstract_assheat} of Theorem \ref{thm:abstract} are satisfied with $k=2$; to conclude, we only need to check assumption \ref{en:abstract_assL1} with $\varsigma = 2/2$.

 As in the proof of Theorem \ref{thm:abstract}, the doubling condition and the heat kernel bounds allow us to apply \cite[Theorem 6.1(ii)]{martini_crsphere} to the operator $\opL$ and deduce that
\begin{equation}\label{eq:standard_l1_grushin}
\sup_{r>0} \esssup_{z' \in \RR^2} \int_{\RR^2} |\Kern_{\mm(r^2\opL)}(z,z')|  (1+\dist(z,z')/r)^\alpha \,dz \lesssim_{\alpha,\beta} \|\mm\|_{\sobolev{\beta}{\infty}}.
\end{equation}
for all continuous $\mm : \RR \to \CC$ supported in $[1/4,1]$ and all $(\alpha,\beta)$ such that
\begin{equation}\label{eq:standard_range}
\alpha \geq 0 \qquad\text{and}\qquad \beta > \alpha+Q/2+1/q.
\end{equation}
On the other hand, from assumption \ref{en:ass_w_planch}, Proposition \ref{prp:weights} and the Cauchy--Schwarz inequality, we easily deduce that \eqref{eq:standard_l1_grushin} also holds for all $(\alpha,\beta)$ such that
\begin{equation}\label{eq:special_range}
\alpha<-1/2 \qquad\text{and}\qquad \beta=1/2.
\end{equation}
Indeed, if $\alpha<-1/2$, then $-2\alpha>1$. So, if we choose $\vartheta \in [0,1/2)$ sufficiently large that $-2\alpha>1+(1+D/2)(1-2\vartheta)$, then
\[\begin{split}
&\int_{\RR^2} |\Kern_{\mm(r^2\opL)}(z,z')|  (1+\dist(z,z')/r)^\alpha  \\
&\leq \left(\int_{\RR^2} (1+\dist(z,z')/r)^{2\alpha} (1+\weight_r(z,z'))^{-2\vartheta} \,dz \right)^{1/2}\\
&\qquad\times \left(\int_{\RR^2} \left |(1+\weight_r(z,z'))^{\vartheta} \Kern_{\mm(r^2\opL)}(z,z') \right|^2\,dz \right)^{1/2}\\
&\lesssim_{\alpha} \|\mm\|_{\sobolev{\vartheta}{q}} \\
&\lesssim_\alpha \|\mm\|_{\sobolev{1/2}{q}}.
\end{split}\]
Interpolation between the two ranges \eqref{eq:standard_range} and \eqref{eq:special_range} finally yields that \eqref{eq:standard_l1_grushin} also holds for
\[
\alpha=0 \qquad\text{and}\qquad \beta>1,
\]
thus verifying assumption \ref{en:abstract_assL1} of Theorem \ref{thm:abstract} with $\varsigma = 2/2$.
\end{proof}

The rest of the paper is aimed at proving that, under the assumptions \eqref{eq:2dassumptions} on $V$, the weighted Plancherel estimate \eqref{eq:weightedpl_abstract} indeed holds in the case $q=\infty$; in light of Theorem \ref{thm:conditional_weightedplancherel}, this will prove Theorem \ref{thm:main}. As we shall see, the unweighted estimate ($\vartheta=0$) of assumption \ref{en:ass_planch} holds in greater generality, and the full strength of \eqref{eq:2dassumptions} will only be needed for the weighted estimate ($\vartheta > 0$) in assumption \ref{en:ass_w_planch}.

As mentioned in the introduction, the proof of the weighted Plancherel estimate will be given in Section \ref{s:proofweightedplancherel}. Instrumental in the proof is the theory of one-dimensional Schr\"odinger operators developed in Sections \ref{s:halflineschroedinger} to \ref{s:matrixbounds}.

\section{Schr\"odinger equations on a half-line: generalities}\label{s:halflineschroedinger}

This section is devoted to establishing some general properties of real-valued solutions of a second-order ODE of the form
\[
-u''+Vu=w
\]
on a half-line, where the real-valued potential $V$ is assumed to be continuous. In particular, in Section \ref{ss:existence_recessive} we discuss the existence of recessive solutions for appropriate choices of the inhomogeneous term $w$. In Section \ref{ss:agmon} the exponential decay at infinity of such solutions is established. Finally, in Section \ref{ss:smoothdependence} some results on the smooth dependence on parameters are obtained.

While most proofs are relatively straightforward adaptations of standard techniques, the generality and uniformity of the obtained results will be crucial in the subsequent developments.

\subsection{Existence of solutions vanishing at infinity}\label{ss:existence_recessive}

The proof of the following result is an adaptation of arguments in \cite[Section 2.3]{berezin-shubin}, where the case of the homogeneous ODE ($w=0$) is considered.

\begin{thm}\label{thm:existence}
Let $I \subseteq \RR$ be an interval with $\sup I = +\infty$, and $V : I \to \RR$ be continuous and such that 
\begin{equation}\label{eq:potpos}
\liminf_{x\to+\infty} V > 0.
\end{equation}
Let $w : I \to \RR$ be continuous and such that 
\begin{equation}\label{eq:rhs}
\lim_{x \to +\infty} \frac{w(x)}{V(x)} = 0.
\end{equation}
Then, for all solutions $u$ on $I$ of
\begin{equation}\label{eq:ode}
-u'' + V u = w,
\end{equation}
the limit $\lim_{x \to \infty} u(x)$ exists and is one of $+\infty,-\infty,0$. Solutions with each of these limits exist, and those with limit zero form a one-dimensional affine subspace of $C^2(I;\RR)$.
\end{thm}
\begin{proof}
Note first that, by the classical theory of linear ODE, solutions on $I$ of the equation \eqref{eq:ode} exist and form a two-dimensional affine subspace of $C^2(I;\RR)$.

\setcounter{step}{-1}
\begin{step}\label{step:zero}
Let $u$ be a solution of \eqref{eq:ode} on $I$. If there exists $x_* \in I$ such that $V>0$ on $[x_*,+\infty)$, $u'(x_*) > 0$, and $u(x_*) \geq \sup_{[x_*,+\infty)} |w/V|$, then $\lim_{x \to +\infty} u(x) = +\infty$.
\end{step}
\begin{proof}
We claim that $u(x) \geq u(x_*)$ for all $x \geq x_*$. Indeed, if this is not the case and $x_2 = \inf \{ x > x_* \tc u(x) < u(x_*) \}$, then $x_2 > x_*$ (because $u'(x_*) > 0)$ and $u(x) \geq u(x_*)  \geq \sup_{[x_*,+\infty)} |w/V|$ for all $x \in (x_*,x_2)$, whence $u''(x) = V(x) ( u(x) - w(x)/V(x) ) \geq 0$ and $u'(x) \geq u'(x_*) > 0$ for all $x \in (x_*,x_2)$, so $u(x_2) > u(x_*)$, which contradicts (by continuity of $u$) the definition of $x_2$.
		
Now, since $u \geq u(x_*) \geq \sup_{[x_*,+\infty)} |w/V|$ on $[x_*,+\infty)$, from \eqref{eq:ode} we deduce again that $u'' \geq 0$ on $[x_*,+\infty)$, whence $u' \geq u'(x_*) > 0$ on $[x_*,+\infty)$, which clearly implies that $\lim_{x \to +\infty} u(x) = +\infty$.
\end{proof}
	
Let $\solU$ be the set of all solutions $u : I \to \RR$ of \eqref{eq:ode} which do not tend to $\pm\infty$ as $x \to +\infty$ (note that we do not assume \emph{a priori} that $\lim_{x \to +\infty} u(x)$ exists for $u \in \solU$; we only ask that, if the limit exists, then it is neither $+\infty$ nor $-\infty$).
	
\begin{step}\label{step:one}
If $u \in \solU$, then $\limsup_{x \to +\infty} u(x) \geq 0$ and $\liminf_{x \to +\infty} u(x) \leq 0$.
\end{step}
\begin{proof}
Assume for a contradiction that $\liminf_{x \to +\infty} u(x) > 0$. Then, by \eqref{eq:ode} and \eqref{eq:rhs}, $\liminf_{x\to+\infty} \frac{u''(x)}{V(x)} > 0$, whence, by \eqref{eq:potpos}, $\liminf_{x\to+\infty} u''(x) > 0$ too. This clearly implies that $\lim_{x \to +\infty} u(x) =+ \infty$, thus contradicting that $u \in \solU$. In a similar way one rules out that $\limsup_{x \to +\infty} u(x) < 0$.
\end{proof}
	
\begin{step}\label{step:two}
If $u \in \solU$, then $\lim_{x \to +\infty} u(x) = 0$.
\end{step}
\begin{proof}
Assume for a contradiction that $u$ does not vanish at $+\infty$. From Step \ref{step:one} we deduce that $\limsup_{x \to +\infty} u(x)$ and $\liminf_{x \to +\infty} u(x)$ must be different, so one of them must be nonzero. Without loss of generality, we may assume that $\delta \defeq \limsup_{x \to +\infty} u(x) > 0$ and $\liminf_{x \to +\infty} u(x) \leq 0$. By \eqref{eq:potpos} and \eqref{eq:rhs} we can find $M \in I$ such that $|w/V| \leq \delta/3$ and $V > 0$ on $[M,+\infty)$. In addition, we may find $x_1 > x_0 > M$ such that $u(x_0) \leq 2\delta/3$ and $u(x_1) > 2\delta/3$. If $\bar x = \sup \{x \in [x_0,x_1) \tc u(x) \leq 2\delta/3 \}$, then, by continuity of $u$, $u(\bar x) = 2\delta/3 < u(x_1)$; consequently, by Lagrange's Mean Value Theorem there exists $x_* \in (\bar x,x_1)$ such that $u'(x_*) > 0$ and $u(x_*) > 2\delta/3$. By Step \ref{step:zero}, this implies that $\lim_{x \to +\infty} u(x) = +\infty$, thus contradicting that $u \in \solU$.
\end{proof}
	
Note that Step \ref{step:two} shows that all solutions $u$ of \eqref{eq:ode} on $I$ have limit at $+\infty$ and the limit must be one of $+\infty,-\infty,0$. We now show that solutions with each of these limits exist.
	
Take $x_0 \in I$ such that $\inf_{[x_0,+\infty)} V > 0$, and for all $\lambda \in \RR$ let $u_\lambda$ be the solution of \eqref{eq:ode} on $I$ such that $u(x_0) = 0$ and $u'(x_0) = \lambda$.
	
\begin{step}\label{step:three}
If $\lambda > \lambda'$ then $u_\lambda(x) > u_{\lambda'}(x)$ for all $x > x_0$.
\end{step}
\begin{proof}
The function $v = u_\lambda-u_{\lambda'}$ is a solution of the homogeneous ODE $-v'' + Vv = 0$ with $v(x_0) = 0$ and $v'(x_0) = \lambda-\lambda' > 0$, hence $v$ cannot vanish on $(x_0,+\infty)$ \cite[Corollary of Theorem 3.2, Section 2.3]{berezin-shubin} and consequently it must be strictly positive there.
\end{proof}
	
From Step \ref{step:two} we deduce that $\RR = \Lambda_{-\infty} \cup \Lambda_0 \cup \Lambda_{+\infty}$, where $\Lambda_\mu = \{ \lambda \in \RR \tc \lim_{x \to \infty} u_\lambda(x) = \mu \}$; moreover, by Step \ref{step:three}, the $\Lambda_\mu$ are intervals and $\Lambda_{\mu} < \Lambda_{\mu'}$ (in the sense that the inequality holds for any respective elements) whenever $\mu < \mu'$.
	
\begin{step}\label{step:four}
$\Lambda_{-\infty}$ and $\Lambda_{+\infty}$ are open and nonempty.
\end{step}
\begin{proof}
We only consider $\Lambda_{+\infty}$, the argument for $\Lambda_{-\infty}$ being analogous.
		
Let $M = \sup_{[x_0,x_0+1]} |u_0''(x)|$. Then, by Step \ref{step:three}, for all $\lambda>0$,
\[
u''_\lambda = Vu_\lambda - w \geq Vu_0-w = u_0'' \geq -M
\]
on $[x_0,x_0+1]$. Consequently, by Taylor's Theorem,
\[
u_\lambda'(x_0+1) \geq \lambda - M, \qquad u_\lambda(x_0+1) \geq \lambda - M/2.
\]
This shows that, if we take $\lambda$ sufficiently large, then we may ensure that $u_\lambda'(x_0+1) > 0$ and $u_\lambda(x_0+1) > \sup_{[x_0+1,+\infty)} |w/V|$, and Step \ref{step:zero} yields that $\lim_{x \to +\infty} u_\lambda(x) = +\infty$ in this case. This shows that $\Lambda_{+\infty}$ is nonempty.
		
Let now $\lambda_0 \in \Lambda_{+\infty}$, that is $\lim_{x \to +\infty} u_{\lambda_0}(x) = +\infty$. Then from \eqref{eq:ode} and \eqref{eq:rhs} we deduce that $\lim_{x \to+\infty} u_{\lambda_0}''(x) = +\infty$, whence also $\lim_{x \to +\infty} u_{\lambda_0}'(x) = +\infty$. In particular, we can find $x_* > x_0$ such that $u_{\lambda_0}'(x_*) > 0$ and $u_{\lambda_0}(x_*) > \sup_{[x_0,+\infty)} |w/V|$. By classical ODE theory, $u_\lambda(x_*)$ and $u'_\lambda(x_*)$ depend continuously on $\lambda$, hence the inequalities $u_{\lambda}'(x_*) > 0$ and $u_{\lambda}(x_*) > \sup_{[x_0,+\infty)} |w/V|$ also hold for all $\lambda$ sufficiently close to $\lambda_0$, and, for these values of $\lambda$, Step \ref{step:zero} yields that $\lim_{x \to +\infty} u_\lambda(x) = +\infty$, that is, $\lambda \in \Lambda_{+\infty}$. This shows that $\Lambda_{+\infty}$ is open.
\end{proof}
	
In conclusion, since $\RR$ is the disjoint union of $\Lambda_{-\infty}$, $\Lambda_0$ and $\Lambda_{+\infty}$, and moreover $\Lambda_{-\infty}$ and $\Lambda_{+\infty}$ are open and nonempty by Step \ref{step:four}, from the fact that $\RR$ is connected we deduce that $\Lambda_0 \neq \emptyset$.
	
This shows in particular that $\solU$ is nonempty. We also note that, for all $u_1,u_2 \in \solU$, the difference $v = u_1 - u_2$ is a solution of the homogeneous ODE $-v''+Vu = 0$ such that $\lim_{x \to +\infty} v(x) = 0$. Since the space of these solutions is one-dimensional \cite[Theorem 3.3, Section 2.3]{berezin-shubin}, we deduce that $\solU$ is a one-dimensional affine subspace of $C^2(I;\RR)$.
\end{proof}

\subsection{Agmon estimates}\label{ss:agmon}
In this section we exploit a classical technique due to Agmon \cite{agmon} to obtain information about the decay of recessive solutions.

\begin{thm}\label{agmon_general_thm}
Let $V : \Rpos \to \RR$ be continuous. Let $\rho : \Rpos\to \RR$ be Lipschitz and such that
\begin{equation}\label{agmon_metric_bound}
(\rho')^2\leq V_+,
\end{equation}
where $V_+(x) \defeq \max\{V(x),0\}$.
Assume that $w : \Rpos \to \RR$ is continuous and $u$ is a solution of
\begin{equation}\label{eq:agmon_ODE}
-u''(x) + V(x) \, u(x) = w(x) \quad (x\in \Rpos),
\end{equation}  
such that $\limsup_{x\to +\infty}|u(x)|<+\infty$. Then
\begin{equation}\label{agmon}
\int_b^{+\infty} e^{2\gamma\rho} V u^2 \lesssim_\gamma \int_a^{+\infty} e^{2\gamma\rho} \frac{w^2}{V} + \frac{1}{(b-a)^2} \int_a^b e^{2\gamma\rho} u^2
\end{equation}
for every $\gamma\in [0,1)$ and $a,b \in \Rpos$ such that $b \in (a,+\infty) \subseteq \{V>0\}$.
\end{thm}
\begin{proof}
Let $\formE(u_1,u_2) \defeq \int_{\Rpos} u_1' u_2'+\int_{\Rpos} Vu_1u_2$ and $\formE(u) \defeq \formE(u,u)$. One first observes that, by the Leibniz rule,
the \emph{localisation identity}
\begin{equation*}
\formE(\eta v) = \formE(\eta^2 v, v) + \int_{\Rpos}(\eta')^2 v^2
\end{equation*}
holds for any $v$ Lipschitz and $\eta$ Lipschitz and compactly supported. Additionally, by \eqref{eq:agmon_ODE} and integration by parts,
\[
\formE(\zeta,u) = \int_{\Rpos} \zeta w
\]
for any $\zeta$ Lipschitz and compactly supported.
We apply the above identities to $v = u$, $\zeta = \eta^2 u$, $\eta=\chi_c e^{\gamma\rho}$, with $\gamma$ and $\rho$ as in the statement, and
\begin{equation*}
\chi_c(x) \defeq \begin{cases}
0 &\text{ if }x\leq a,\\
(x-a)/(b-a) &\text{ if }a\leq x\leq b,\\
(c-x)/(c-b) &\text{ if } b\leq x\leq c,\\
0 &\text{ if }x\geq c
\end{cases}
\end{equation*}
for $c>b>a$. The result is:
\[\begin{split}
	\int_{\Rpos} \chi_c^2 e^{2\gamma\rho} V u^2 
	&\leq \formE(\eta u)	\\
	&= \int_{\Rpos}\chi_c^2 e^{2\gamma\rho}uw +\int_{\Rpos}\left(\chi_c'+\gamma\rho'\chi_c\right)^2e^{2\gamma\rho}u^2 \\
	&\leq  \varepsilon\int_{\Rpos} \chi_c^2 e^{2\gamma\rho}Vu^2+ \frac{1}{4\varepsilon}\int_{\Rpos} \chi_c^2 e^{2\gamma\rho} \frac{w^2}{V} \\
	&\quad+ (1+\varepsilon^{-1})\int_{\Rpos} (\chi_c')^2e^{2\gamma\rho} u^2 + (1+\varepsilon)\gamma^2\int_{\Rpos} (\rho')^2\chi_c^2e^{2\gamma\rho}u^2\\
	&\leq  (\varepsilon+(1+\varepsilon)\gamma^2)\int_{\Rpos} \chi_c^2 e^{2\gamma\rho}Vu^2+ \frac{1}{4\varepsilon}\int_{\Rpos} \chi_c^2 e^{2\gamma\rho} \frac{w^2}{V} \\
	&\quad+ \frac{(1+\varepsilon^{-1})}{(b-a)^2}\int_a^b e^{2\gamma\rho}u^2 + \frac{(1+\varepsilon^{-1})}{(c-b)^2}\int_b^c e^{2\gamma\rho}u^2 ,
\end{split}\]
where $\varepsilon>0$ is arbitrary, and we have used that $(a,+\infty) \subseteq \{V>0\}$ and \eqref{agmon_metric_bound}. Since $\gamma<1$, we can choose $\varepsilon$ so small that $(\varepsilon+(1+\varepsilon)\gamma^2)<1$. This allows to absorb the first term of our estimate into the left-hand side. Moreover, \emph{if $\rho$ is bounded above}, then we can let $c$ tend to $+\infty$ and use the assumption $\limsup_{x\to +\infty}|u(x)|<+\infty$ to conclude that \eqref{agmon} holds. To remove the additional boundedness assumptions, one may just apply the estimate to $\min\{\rho,T\}$ in place of $\rho$, where $T\in \RR$, and let $T$ tend to $+\infty$. We omit the easy details. 
\end{proof}

A particularly important instance of the previous estimate is given in the following statement.

\begin{thm}\label{agmon_thm}
Let $V : \Rpos \to \RR$ be continuous and strictly increasing. Assume that $w : \Rpos\to \RR$ is continuous and $u$ is a solution of \eqref{eq:agmon_ODE}
such that $\limsup_{x\to +\infty}|u(x)|<+\infty$. Let $A<B$ be nonnegative values of $V$. If $\gamma\in [0,1)$, then
\[\begin{split}
&\int_{V>B}\exp\left(2\gamma \int_{V^{\inv}(B)}^x \sqrt{V}\right) V(x) \, u(x)^2 \,dx \\
&\lesssim_{\gamma} \int_{V>B} \exp\left(2\gamma \int_{V^{\inv}(B)}^x \sqrt{V}\right) \frac{w(x)^2}{V(x)} \,dx \\
&\quad+ \int_{A<V<B} \frac{w(x)^2}{V(x)} \,dx + \frac{1}{|\{A<V<B\}|^2} \int_{A<V<B} u(x)^2 \,dx .
\end{split}\]
\end{thm}
\begin{proof}	
Apply Theorem \ref{agmon_general_thm} with $a=V^{\inv}(A)$, $b=V^{\inv}(B)$, and
$\rho(x) = \int_{V^{\inv}(B)}^x \sqrt{V}$.
\end{proof}

The $L^2$ estimates of the previous theorem imply corresponding pointwise bounds, contained in the following corollary. 

\begin{cor}\label{pointwise_agmon_cor}
Let $V : \Rpos \to \RR$ be continuous and strictly increasing. Assume that $w : \Rpos \to \RR$ is continuous and $u$ is a solution of \eqref{eq:agmon_ODE}
such that $\lim_{x \to +\infty} u(x) = 0$. 
Let $A<B<C$ be nonnegative values of $V$ and $\varepsilon,\gamma\in (0,1)$. Define
\begin{equation*}
\auxC_j \defeq \int_{V>C} \exp\left(-2\varepsilon\int_{V^{\inv}(B)}^y\sqrt{V}\right) V(y)^j \, dy \qquad (j = 0,1)
\end{equation*} 
and 
\[\begin{split}
\auxD &\defeq \int_{V>B} \exp\left(2\gamma \int_{V^{\inv}(B)}^y \sqrt{V} \right) \frac{w(y)^2}{V(y)} \,dy \\
&\quad +\int_{A<V<B} \frac{w(y)^2}{V(y)} \,dy + \frac{1}{|\{A<V<B\}|^2} \int_{A<V<B} u(y)^2 \,dy.
\end{split}\]
If $x \in \{ V > C \}$ and $\beta \geq 0$ are such that $\beta+\varepsilon \leq \gamma$, then
\begin{equation}\label{eq:pointwise_agmon_der}
|u'(x)| \lesssim_{\gamma} \sqrt{\auxC_1\auxD} \exp\left(-\beta\int_{V^{\inv}(B)}^x\sqrt{V}\right);
\end{equation}	
if moreover $\beta+3\varepsilon \leq \gamma$, then
\begin{equation}\label{eq:pointwise_agmon}
|u(x)| \lesssim_{\gamma} \auxC_0\sqrt{\auxC_1\auxD} \exp\left(-\beta\int_{V^{\inv}(B)}^x\sqrt{V}\right).
\end{equation}
\end{cor}

\begin{proof}
Define $I(y) = \exp\left(\int_{V^{\inv}(B)}^y \sqrt{V}\right)$ for $y \in \{ V > 0 \}$.

Let us first prove the estimate for $u'$. 
We may assume that $\auxC_1\auxD  < \infty$, otherwise the desired estimate holds trivially.

Let $x \in \{ V > C \}$. For every $x_*>x$,
\begin{multline}\label{eq:est_diffder}
|u'(x_*) - u'(x)|^2
\leq \left(\int_{x}^{+\infty} |u''| \right)^2 \\
\leq \int_{x}^{+\infty} I^{-2\gamma} V  
\cdot \int_{x}^{+\infty} I^{2\gamma} \frac{(u'')^2}{V} 
\eqdef J_1(x) \cdot J_2(x).
\end{multline}
Now, trivially, if $\varepsilon \leq \gamma$, then
\begin{equation}\label{eq:est_J1}
J_1(x) \leq \auxC_1 I(x)^{-2(\gamma-\varepsilon)},
\end{equation}
as $I$ is an increasing function. Moreover, as $u$ is a solution of \eqref{eq:agmon_ODE}, we deduce that
\[
\frac{(u'')^2}{V} \leq 2 V u^2 +2 \frac{w^2}{V},
\]
and therefore Theorem \ref{agmon_thm} gives that
\begin{equation}\label{eq:est_J2}
J_2(x) \leq \int_{V>B} \exp\left(2\gamma \int_{V^{\inv}(B)}^y \sqrt{V}\right) \frac{u''(y)^2}{V(y)} \,dy \lesssim_\gamma \auxD.
\end{equation}
As $\varepsilon \leq \gamma$, the estimates \eqref{eq:est_J1} and \eqref{eq:est_J2} imply in particular that the integrals $J_1(x)$ and $J_2(x)$ are finite and tend to zero as $x \to +\infty$. From \eqref{eq:est_diffder} we then deduce that $\lim_{y \to \infty} u'(y)$ exists and is finite. On the other hand, $\lim_{y\to+\infty} u(y) =0$, so $u'$ must vanish at $+\infty$. Taking the limit as $x_* \to +\infty$ in \eqref{eq:est_diffder} and bounding the right-hand side by \eqref{eq:est_J1} and \eqref{eq:est_J2} finally gives the desired bound \eqref{eq:pointwise_agmon_der} for $u'(x)$.

Analogously we can prove the estimate for $u$. Again, we may assume that $\auxC_0 \sqrt{\auxC_1 \auxD} < \infty$.
Since $\lim_{y \to +\infty}u(y)=0$, by \eqref{eq:pointwise_agmon_der},
\[
|u(x)|
\leq \int_x^{+\infty} |u'|
\lesssim_{\gamma} \sqrt{\auxC_1 \auxD} \int_x^{+\infty} I^{-(\gamma-\varepsilon)} 
\lesssim_{\gamma} \auxC_0 \sqrt{\auxC_1 \auxD} I(x)^{-(\gamma-3\varepsilon)} 
\]
whenever $3\varepsilon \leq \gamma$, and \eqref{eq:pointwise_agmon} is proved.
\end{proof}

We now apply the above estimates to solutions of
\begin{equation}\label{equation}
-u''(x) + V(x) \, u(x) = E \, u(x) \qquad (x\in \Rpos),
\end{equation}
where the potential $V$ belongs to the following class.

\begin{dfn}
Let $\halfpot$ be the class of the functions $V : \Rpos \to \Rpos$ which are continuous, strictly increasing, and such that $\lim_{x \to 0^+}V(x)=0$ and $\lim_{x \to +\infty}V(x)=+\infty$.
\end{dfn}

In other words, a $V \in \halfpot$ is ``half of a potential'' in the class $\pot$ introduced in Section \ref{basic_sec} below.

\begin{cor}\label{agmon_second_cor}
Let $V\in \halfpot$ and $E>0$, and set $\delta = E \, |\{E< V<2E\}|^2$. Let $u$ be a solution of \eqref{equation} such that $\limsup_{x \to +\infty} |u(x)| < \infty$.
Then
\begin{align}
\int_{V\geq 2E} V  u^2  &\lesssim \delta^{-1} E \int_{E < V < 2E} u^2 , \label{eq:auxenbd} \\
\int_{V\geq 2E} u^2  &\lesssim \delta^{-1} \int_{E<V<2E} u^2 . \label{eq:agmon_second_cor}
\end{align}
\end{cor}
\begin{proof}
By applying Theorem \ref{agmon_thm} to the potential $V-E$ with $w=0$, $\gamma=0$, $A=0$, and $B=E$, we immediately obtain that
\begin{equation}\label{eq:agmon_first_cor}
\int_{E<V<2E} u^2  \gtrsim |\{E< V<2E\}|^2 \int_{V\geq 2E} V u^2 .
\end{equation}
Multiplying both sides of \eqref{eq:agmon_first_cor}
 by $E$ gives \eqref{eq:auxenbd}, and
 \eqref{eq:agmon_second_cor} follows immediately from \eqref{eq:auxenbd} and the trivial bound $\int_{V\geq 2E} V  u^2  \geq 2E \int_{V\geq 2E} u^2$.
\end{proof}

Finally, we record here an elementary proposition for solutions $u$ of \eqref{equation}, which will be used in Section \ref{ss:gaps}.

\begin{prp}\label{prp:elementaryids}
Let $V : \Rpos \to \RR$ be continuous and $E \in \RR$. Let $u$ be a solution of \eqref{equation} that is recessive at $+\infty$ (that is, $\lim_{x \to +\infty} u(x) = \lim_{x \to +\infty} u'(x) = 0$).
\begin{enumerate}[label=(\roman*)]
\item\label{en:half_energy_prp} For every $a>0$,
\begin{equation}
\int_a^{+\infty} (u')^2 + \int_a^{+\infty} V u^2  = E \int_a^{+\infty} u^2 - u(a) \, u'(a).
\end{equation}
\item\label{en:v(0)_prp}
If moreover $V \in C^1(\Rpos)$ and $\lim_{x\to +\infty} V(x) u(x)^2 = 0$, then, for every $a>0$,
\begin{equation}\label{eq:sonin}
(E-V(a)) \, u(a)^2 + u'(a)^2 = \int_a^{+\infty} V' \, u^2 .
\end{equation}
\end{enumerate}
\end{prp}

\begin{proof}
The equation \eqref{equation} and an integration by parts yields
\[
\int_a^b (u')^2 + \int_a^b V u^2 = E \int_a^b u^2 
+ u(b) u'(b) - u(a) u'(a).
\]
Letting $b$ tend to $+\infty$ gives part \ref{en:half_energy_prp}.

As for part \ref{en:v(0)_prp}, by integrating the identity
\begin{equation}
\left( (E-V) u^2 + (u')^2 \right)' = -V' u^2,
\end{equation}
we get, for all $a<b$,
\begin{equation*}
\left. (E-V) u^2 + (u')^2\right|_a^b = - \int_a^b V' u^2 .
\end{equation*}
Letting $b$ tend to $+\infty$ and using that $\lim_{x\to +\infty} \left( u(x)^2 + u'(x)^2 + V(x) u(x)^2 \right) = 0$ gives \eqref{eq:sonin}. 
\end{proof}

\subsection{Smooth dependence on parameters}\label{ss:smoothdependence}

Let $I \subseteq \RR$ be a closed upper half-line, $V : I \to \RR$ be continuous, and $\alpha \geq 0$. Define the spaces of real-valued functions
\begin{gather*}
\banP_V = \Biggl\{ W \in C^0(I;\RR) \tc \sup_I \frac{|W|}{1+|V|} < \infty \Biggr\},\\
\banR_{V,\alpha} = \Biggl\{ w \in C^0(I;\RR) \tc \lim_{x \to \infty} \frac{w(x)}{1+|V(x)|} = 0,\, \int_I e^{2\alpha x} \frac{w(x)^2}{1+|V(x)|} \,dx < \infty \Biggr\}, \\
\banD_{V,\alpha} = \{ u \in C^2(I;\RR) \tc (1+|V|) u \in \banR_{V,\alpha}, \, u'' \in \banR_{V,\alpha} \}.
\end{gather*}
Clearly $\banP_V$, $\banR_{V,\alpha}$ and $\banD_{V,\alpha}$ are Banach spaces with the norms
\begin{gather*}
\|W\|_{\banP_V} = \|W/(1+|V|)\|_\infty,\\
\|w\|_{\banR_{V,\alpha}} = \|w/(1+|V|)\|_\infty + \| e^{\alpha \cdot} w/\sqrt{1+|V|} \|_2,\\
\|u\|_{\banD_{V,\alpha}} = \|u\|_\infty + \|u''/(1+|V|)\|_\infty + \|e^{\alpha \cdot}  u \sqrt{1+|V|} \|_2 + \| e^{\alpha \cdot} u''/\sqrt{1+|V|} \|_2.
\end{gather*}
Note that $V \in \banP_V$.

\begin{prp}\label{u'_prp}
Let $I$ be a closed upper half-line, $V : I \to \RR$ be continuous, and $\alpha \geq 0$. Assume that $K_{V,\alpha} \defeq \left(\int_I (1+|V(x)|) \, e^{-2\alpha x} \,dx\right)^{1/2} < \infty$. Then, for all $u \in \banD_{V,\alpha}$, $\lim_{x \to \infty} u'(x) = 0$. Moreover, for all $\beta \in (0,\alpha)$, if $K_{V,\alpha-\beta} < \infty$, then
\[
|u(x)| \leq \beta^{-1} K_{V,\alpha-\beta} \|u\|_{\banD_{V,\alpha}} e^{-\beta x}, \qquad |u'(x)| \leq K_{V,\alpha-\beta} \|u\|_{\banD_{V,\alpha}} e^{-\beta x}
\]
for all $x \in I$.
\end{prp}
\begin{proof}
Note first that $\liminf_{x \to \infty} |u'(x)| = 0$ (otherwise $\lim_{x \to \infty} u(x) = \pm\infty$, which is not the case, since $u \in \banD_{V,\alpha}$). Now, for all $x_1,x_2 \in I$ with $x_1 \leq x_2$, by the Cauchy--Schwarz inequality,
\[
|u'(x_2)-u'(x_1)| \leq \int_{x_1}^{x_2} |u''|  \leq \|u\|_{\banD_{V,\alpha}} \left(\int_{x_1}^\infty (1+|V(t)|) \, e^{-2\alpha t} \,dt \right)^{1/2} , 
\]
and the last integral tends to zero as $x_1 \to \infty$ under our assumptions, so the limit $\lim_{x \to \infty} u'(x)$ exists in $\RR$, and therefore it must be zero.

Assume now that $K_{V,\alpha-\beta} <\infty$ for some $\beta \in (0,\alpha)$; then, by taking the limit as $x_2 \to \infty$ in the previous inequality we obtain that
\[
|u'(x_1)| \leq \|u\|_{\banD_{V,\alpha}} \left(\int_{x_1}^\infty (1+|V(t)|) \, e^{-2\alpha t} \,dt \right)^{1/2} \leq \|u\|_{\banD_{V,\alpha}} K_{V,\alpha-\beta} e^{-\beta x_1}, 
\]
and consequently
\[
|u(x_1)| \leq \int_{x_1}^\infty |u'|  \leq \|u\|_{\banD_{V,\alpha}} K_{V,\alpha-\beta} \int_{x_1}^\infty e^{-\beta t} \,dt = \beta^{-1} \|u\|_{\banD_{V,\alpha}} K_{V,\alpha-\beta} e^{-\beta x_1}
\]
for all $x_1 \in I$.
\end{proof}

In order to present the results below about smooth dependence on parameters, it is convenient to make use of real analyticity of maps between Banach spaces; for basic definitions and results about real-analytic maps in the context of Banach spaces, we refer to \cite[Chapter 4]{buffoni_toland}.

\begin{lem}\label{lem:banach}
Let $I \subseteq \RR$ be a closed upper half-line, $V : I \to \RR$ be continuous, and let $\alpha \geq 0$.
\begin{enumerate}[label=(\roman*)]
\item\label{en:corex_bdd} For all $W \in \banP_V$, the Schr\"odinger operator $\opH[W] = -\partial_x^2+ W$ maps $\banD_{V,\alpha}$ into $\banR_{V,\alpha}$ boundedly.
\item\label{en:corex_smooth} The map $\Theta : \banP_V \times \banD_{V,\alpha} \ni (W,u) \mapsto \opH[W] u \in \banR_{V,\alpha}$ is real-analytic and
\[
d\Theta_{(W,u)}(H,f) = \opH[W] f + Hu
\]
for all $(W,u),(H,f) \in \banP_V \times \banD_{V,\alpha}$.
\item\label{en:corex_gamma} The map $\RR \times \RR \ni (t,E) \mapsto tV-E \in \banP_V$ is linear and bounded.
\end{enumerate}
Assume now that $\liminf_{x \to +\infty} V(x) > \alpha^2$.
\begin{enumerate}[label=(\roman*),resume]
\item\label{en:corex_surj} The operator $\opH[V] : \banD_{V,\alpha} \to \banR_{V,\alpha}$ is surjective.
\item\label{en:corex_kernel} The kernel of $\opH[V]$ in $\banD_{V,\alpha}$ is one-dimensional, and coincides with the set of the $u \in C^2(I;\RR)$ such that $-u''+Vu = 0$ and $\lim_{x \to \infty} u(x) = 0$.
\end{enumerate}
\end{lem}
\begin{proof}
Parts \ref{en:corex_bdd} to \ref{en:corex_gamma} are easily checked; we only remark that the map $\Theta$ can be written as the sum of $(W,u) \mapsto -u''$ and $(W,u) \mapsto Wu$, which are a bounded linear map and a bounded bilinear map respectively.
	 	
Assume now that $\liminf_{x \to +\infty} V(x) > \alpha^2$. Let $w \in \banR_{V,\alpha}$. Then, by Theorem \ref{thm:existence}, there exists $u \in C^2(I;\RR)$ such that $\opH[V] u = -u''+Vu = w$ and $\lim_{x \to \infty} u(x) = 0$. We now show that every such $u$ is in $\banD_{V,\alpha}$. Since $w \in \banR_{V,\alpha}$, we know that $\lim_{x \to +\infty} w(x)/V(x) = 0$, so from the ODE we deduce that $\lim_{x \to +\infty} u''(x)/V(x) =0$ too. Moreover, from Theorem \ref{agmon_general_thm}, applied with $\gamma\rho(x) = \alpha x$ and $\gamma$ sufficiently close to $1$, and the fact that $\int_I e^{2\alpha x} w(x)^2/(1+|V(x)|) \,dx < +\infty$, we deduce that $\int_I e^{2\alpha x} (1+|V(x)|) u(x)^2 \,dx < +\infty$, so the ODE implies that $\int_I e^{2\alpha x} u''(x)^2/(1+|V(x)|)  \,dx < +\infty$. Hence $u \in \banD_{V,\alpha}$.
	 	
The above argument shows that $\opH[V] : \banD_{V,\alpha} \to \banR_{V,\alpha}$ is surjective, and that, moreover, the kernel of $\opH[V]$ in $\banD_{V,\alpha}$ coincides with the set of the $u \in C^2(I;\RR)$ such that $-u''+Vu = 0$ and $\lim_{x \to \infty} u(x) = 0$, which we know to be a one-dimensional subspace of $C^2(I;\RR)$. This proves parts \ref{en:corex_surj} and \ref{en:corex_kernel}.
\end{proof}
	 
\begin{thm}\label{thm:smoothdependence}
Let $I \subseteq \RR$ be a closed upper half-line. Let $V : I \to \RR$ be continuous, $\alpha \geq 0$, and set $\Omega_{V,\alpha} = \{ (t,E) \in \RR^+ \times \RR \tc \liminf_{x \to \infty} tV(x) > E + \alpha^2 \}$. Then, for all $(t,E) \in \Omega_{V,\alpha}$, there exists a unique $u_{t,E} \in C^2(I;\RR)$ such that
\[
-u_{t,E}'' + tV u_{t,E} = E u_{t,E}
\]
on $I$, $\int_I u_{t,E}^2 = 1$, and $u_{t,E}(x) > 0$ for all sufficiently large $x$. 
Moreover, $u_{t,E} \in \banD_{V,\alpha}$ and the map $\Omega_{V,\alpha} \ni (t,E) \mapsto u_{t,E} \in \banD_{V,\alpha}$ is real-analytic.
\end{thm}
\begin{proof}
Note first that, if $(t,E) \in \Omega \defeq \Omega_{V,\alpha}$, then a solution $u \in C^2(I;\RR)$ of $-u''+tVu = Eu$ tends to $0$ or $\pm\infty$ at $+\infty$, so if $u \in L^2(I)$ then necessarily $\lim_{x \to +\infty} u(x) = 0$. 
	 	
Note also that, for all $(t,E) \in \RR^+ \times \RR$, $\banD_{tV-E,\alpha}=\banD_{V,\alpha}$ and $\banR_{tV-E,\alpha}=\banR_{V,\alpha}$, with equivalent norms. In particular, from Lemma \ref{lem:banach}, if $(t,E) \in \Omega$, then the set
\[
\banK_{tV-E} = \left\{ u \in C^2(I;\RR) \tc -u''+tVu=Eu, \, \lim_{x\to+\infty}u(x) = 0 \right\}
\]
is the kernel of $\opH[tV-E]$ in $\banD_{V,\alpha}$ and is a one-dimensional subspace thereof; moreover, by classical Sturm--Liouville theory \cite[Corollary of Theorem 3.2, Section 2.3]{berezin-shubin}, every nonzero $u \in \banK_{tV-E}$ is strictly positive or strictly negative in a neighbourhood of $+\infty$. Since $\banD_{V,\alpha} \subseteq L^2(I)$, there exist exactly two elements $u \in \banK_{tV-E}$ with $\int_{I} u^2 = 1$, and $u_{t,E}$ is uniquely determined by additionally requiring that $u_{t,E} > 0$ in a neighbourhood of $+\infty$.
	 	
Define now the map $\Phi : \RR^2 \times \banD_{V,\alpha} \to \RR \times \banR_{V,\alpha}$ by
\[
\Phi(t,E;u) = \left(\int_I u^2, \opH[tV-E] u \right).
\]
By Lemma \ref{lem:banach}, the map $\Phi$ is real-analytic and
\[
d\Phi_{(t,E;u)}(s,F;v) = \left(2\int_I u v, \opH[tV-E] v + (sV-F)u \right)
\]
for all $(t,E),(s,F) \in \RR^2$ and $u,v \in \banD_{V,\alpha}$. 
	 	
Let now $(t,E) \in \Omega$ and set $u=u_{t,E}$. We claim that the map
\[
\banD_{V,\alpha} \ni v \mapsto d\Phi_{(t,E;u)}(0,0;v) \in \RR \times \banR_{V,\alpha}
\]
is an isomorphism of Banach spaces, i.e., it is injective and surjective. Indeed, for all $v \in \banD_{V,\alpha}$, if $d\Phi_{(t,E;u)}(0,0;v) = 0$, then $\int_I u v = 0$ and $\opH[tV-E]v = 0$; the second condition tells us that $v \in \banK_{tV-E} = \RR u$, and together with the first condition this implies that $v = 0$. This shows that $\Phi$ is injective. As for surjectivity, if $w \in \banR_{V,\alpha}$ and $h \in \RR$, then by Lemma \ref{lem:banach} there exists $v_0 \in \banD_{V,\alpha}$ such that $\opH[tV-E] v_0 = w$; in particular, for all $\lambda \in \RR$, we also have $\opH[tV-E](v_0+\lambda u) = w$, and moreover $\int_I u(v_0+\lambda u) = \lambda +\int_I u v_0$, hence we can choose $\lambda \in \RR$ so that $v = v_0+\lambda u$ satisfies $2\int_I u v = h$.

By the Implicit Function Theorem (see, e.g., \cite[Theorem 4.5.4]{buffoni_toland}), there exist open neighbourhoods $A$ of $(t,E)$ in $\Omega$ and $B$ of $u_{t,E}$ in $\banD_{V,\alpha}$, and a real-analytic map $\Psi : A \to B$ such that $\Psi(t,E) = u_{t,E}$ and, for all $(s,F;v) \in A \times B$,
\[
\Phi(s,F;v) = (1,0) \qquad\text{if and only if}\qquad v = \Psi(s,F).
\]
In particular, for all $(s,F) \in A$, $\opH[sV-F] \Psi(s,F) = 0$ and $\int_I \Psi(s,F)^2 = 1$, which implies that $\Psi(s,F) = \pm u_{s,F}$. Up to shrinking $A$, we may assume that $A \subseteq [t_0,+\infty) \times (-\infty,E_0]$ for some $(t_0,E_0) \in \Omega$; in particular, we can find $x_* \in I$ such that $[x_*,+\infty) \subseteq \{ x \in I \tc t_0 V(x)-E_0>0 \} \subseteq \{ x \in I \tc sV(x)-F>0 \}$ for all $(s,F) \in A$. Hence, by Sturm--Liouville theory, $u_{s,F} > 0$ on $[x_*,+\infty)$; since $\Psi(s,F)(x_*)$ is a continuous function of $(s,F)$, we conclude that $\Psi(s,F) = u_{s,F}$, that is, $(s,F) \mapsto u_{s,F}$ is real-analytic in a neighbourhood of $(t,E)$.
\end{proof}

\section{Schr\"odinger equations on a half-line: regular potentials}\label{s:halflineregular}

\subsection{Summary of the results}
The general results of Section \ref{s:halflineschroedinger} are applied here to obtain pointwise and integral bounds for real-valued $L^2$ solutions of
\begin{equation}\label{eq:ODE_regularpot}
-u''(x) + V(x) \, u(x) = E \, u(x) \qquad (x\in \Rpos)
\end{equation} 
under a quantitative $C^1$ assumption on the potential.

\begin{dfn}
Let $\kappa>0$. We denote by $\halfpot_1(\kappa)$ the class of positive functions $V\in C^1(\Rpos)$ such that
\begin{equation*}
	\kappa^{-1} V(x) \leq xV'(x) \leq \kappa V(x) \qquad \forall x\in\Rpos.
\end{equation*}
\end{dfn}

Our estimates will have the desired uniformity as long as the natural ``adimensional'' quantity
\[
\sqrt{E} \, |\{V\leq E\}|
\]
is bounded away from zero. This condition will be automatically satisfied in later applications to eigenfunctions of Schr\"odinger operator on $\RR$, thanks to the eigenvalue estimates of Theorem \ref{thm:eigenvalues} below. We point out, however, that the following bounds apply to arbitrary recessive solutions of \eqref{eq:ODE_regularpot} on $\Rpos$, which need not correspond to eigenfunctions of Schr\"odinger operators on $\RR$; this generality will be crucial in the proof of the eigenvalue gaps in Section \ref{ss:gaps}.

\begin{thm}[Pointwise bounds]\label{thm:pointwise_eigen}
Let $V\in \halfpot_1(\kappa)$. If $\delta>0$, $E>0$ are such that 
\begin{equation}\label{sigma_bound}
	\sqrt{E} \, |\{V\leq E\}| \geq \delta,
\end{equation} 
and $u\in L^2(\Rpos)$ is a real-valued solution of \eqref{eq:ODE_regularpot},
then 
\begin{equation}\label{classical_eigen}
\sup_{V < E/2} \left\{ u^2 + \frac{(u')^2}{E} \right\} 
\lesssim_{\kappa,\delta} \frac{1}{|\{V<E\}|} \int_{E/2 < V < 2 E} u^2.
\end{equation}
Moreover, there exists $c=c(\kappa)>0$ such that the inequality
\begin{equation}\label{exponential_decay_eigen}
u(x)^2+\frac{u'(x)^2}{E}\lesssim_{\kappa,\delta}\frac{e^{-cx\sqrt{V(x)}}}{|\{V\leq E\}|}\int_{E<V<2E}u^2
\end{equation} holds for every $x$ such that $V(x)\geq 4E$.
\end{thm}

\begin{thm}[Integral bounds]\label{thm:Vpowers_prp}
Let $V\in \halfpot_1(\kappa)$, $\delta>0$, $E>0$, and $u\in L^2(\Rpos)$ be as in Theorem \ref{thm:pointwise_eigen}. Then:
\begin{enumerate}[label=(\roman*)]
\item\label{en:Vpowers_upp} If $W : \Rpos \to \Rnon$, $a,b \in \RR$ and $C > 0$ are such that
\begin{equation}\label{eq:xV_int_assumption}
W(x) + \frac{1}{x}\int_0^x W   \leq C x^a V(x)^b 
 \qquad\forall x > 0,
\end{equation}
then
\[
\int_{\Rpos} W  \left( u^2 + \frac{(u')^2}{E} \right) 
\lesssim_{a,b,\kappa,\delta} C \, |\{V\leq E\}|^a E^b \int_{\Rpos} u  .
\]
\item\label{en:Vpowers_ass} The assumption \eqref{eq:xV_int_assumption} holds whenever $W(x) = x^a V(x)^b$ and $a+\kappa^{-1} b>-1$, with $C = C(a,b,\kappa)$.
\item\label{en:Vpowers_low} For all $a,b \in \RR$,
\[
\int_{\Rpos} x^a V(x)^b \, u(x)^2 \,dx
\gtrsim_{a,b,\kappa, \delta} |\{V\leq E\}|^a E^b \int_{\Rpos} u^2 .
\]
\end{enumerate}
\end{thm}

The rest of the section is devoted to the proofs of the above results.

\subsection{An auxiliary class of functions on a half-line}

Let $d_1\leq d_2$ be real numbers. We denote by $\doub(d_1,d_2)$ the collection of positive functions $U\in C^1(\Rpos)$ such that
\begin{equation}\label{eq:doub_def}
	d_1 U(x) \leq xU'(x) \leq d_2 U(x) \qquad \forall x\in\Rpos.
\end{equation}
Notice that $\halfpot_1(\kappa)= \doub(\kappa^{-1}, \kappa)$. The next proposition collects a few elementary algebraic properties of the classes $\doub(d_1,d_2)$. 

\begin{prp}\label{trivial_doubling_prp}
Let $d_1 \leq d_2$ be real numbers. The following hold.
\begin{enumerate}[label=(\roman*)]
\item\label{en:doubling_cone} $\doub(d_1,d_2)$ is a convex cone: if $U,W\in \doub(d_1,d_2)$, then $\lambda U+\mu W\in \doub(d_1,d_2)$ for every $\lambda, \mu\geq 0$.
\item\label{en:doubling_pow} If $U\in \doub(d_1,d_2)$ and $b>0$,  then $U^b\in \doub(bd_1,bd_2)$.
\item\label{en:doubling_prod} If $U\in \doub(d_1,d_2)$ and $W\in \doub(d'_1,d'_2)$, then $UW\in \doub(d_1+d_1',d_2+d_2')$.
\item\label{en:doubling_inv} If $U\in \doub(d_1,d_2)$ with $d_1>0$, then $U : \Rpos \to \Rpos$ is invertible, and the inverse function $U^{\inv}$ is in $\doub(d_2^{-1},d_1^{-1})$.
\end{enumerate}
\end{prp}

\begin{proof}
Parts \ref{en:doubling_cone} to \ref{en:doubling_prod} follow immediately from the definition and elementary differentiation rules.
	
Part \ref{en:doubling_inv} is also elementary, once one realizes that $U$ is invertible. To see this, notice that, if we define $\upsilon : \RR \to \RR$ by $\upsilon(t) = \log U(e^t)$, then the condition \eqref{eq:doub_def} can be equivalently rewritten as
\begin{equation}\label{eq:doub_equiv}
d_1 \leq \upsilon'(t) \leq d_2 \qquad \forall t \in \RR.
\end{equation}
In particular, if $d_1>0$, then $\upsilon$ is strictly increasing and $\lim_{t \to \pm \infty} \upsilon(t) = \pm\infty$, thus $\upsilon : \RR \to \RR$ is invertible by continuity, and therefore $U : \Rpos \to \Rpos$ is invertible too.
\end{proof}

We now discuss a number of ``doubling" properties satisfied by functions in the classes $\doub(d_1,d_2)$. 

\begin{prp}\label{doubling_prp}
Let $d_1 \leq d_2$ be real numbers. The following hold.
\begin{enumerate}[label=(\roman*)]
\item\label{en:doubling_doub} If $U\in \doub(d_1,d_2)$, then we have the \emph{doubling inequality}
\begin{equation}\label{doubling}
U\left(\lambda x\right)\leq \lambda^{d_2} U(x)\qquad\forall x\in \Rpos, \ \forall\lambda\geq 1,
\end{equation} 
and the \emph{reverse doubling inequality}
\begin{equation}\label{reverse_doubling}
U(\lambda x)\geq \lambda^{d_1}U(x) \qquad\forall x\in \Rpos, \ \forall\lambda\geq 1.
\end{equation} 
\item\label{en:sublevel_measure} If $U \in \doub(d_1,d_2)$ with $d_1>0$, then
\begin{gather*}
	c^{1/d_1} |\{U\leq E\}| \leq |\{U\leq cE\}| \leq c^{1/d_2} |\{U\leq E\}| \qquad\forall E>0,\ \forall c\in(0,1),\\
	|\{c E \leq U \leq E\}| \geq \frac{1-c^{1/d_2}}{T^{1/d_1}} |\{U\leq T E\}| \qquad\forall E>0,\ \forall c\in(0,1),\ \forall T\geq 1.
\end{gather*} 
\item\label{en:doubling_int} If $U\in \doub(d_1,d_2)$ with $d_1>-1$, then 
\begin{equation*}
\int_0^a U  \lesssim_{d_1,\lambda} \int_{a/\lambda}^a U  \qquad\forall a>0, \ \forall \lambda> 1. 
\end{equation*}
\item\label{en:auxiliary} If $U \in \doub(d_1,d_2)$ with $d_1>0$, and $E>0$ is such that 
\begin{equation*}
\sqrt{E} \, |\{U\leq E\}|\geq \delta,
\end{equation*} 
then 
\begin{equation*}
	\int_{U\geq E} x^a U(x)^b e^{-\varepsilon x\sqrt{U(x)}} \,dx \lesssim_{d_1,d_2, \delta, a, b, \varepsilon} |\{U\leq E\}|^{a+1} E^b
\end{equation*}
for every $a,b\in \RR$ and $\varepsilon > 0$.
\end{enumerate}
\end{prp}
			
\begin{proof}
Part \ref{en:doubling_doub} follows from the observation that, if we define $\upsilon : \RR \to \RR$ by $\upsilon(t) = \log U(e^t)$, then the conditions \eqref{doubling} and \eqref{reverse_doubling} can be equivalently rewritten as
\[
\upsilon(t) + d_1 s \leq \upsilon(t+s) \leq \upsilon(t) + d_2 s \qquad\forall t \in \RR, \, \forall s \geq 0,
\]
and in turn these inequalities are elementary consequences of the rephrasing \eqref{eq:doub_equiv} of the condition $U \in \halfpot_1(d_1,d_2)$.

Part \ref{en:sublevel_measure} follows from part \ref{en:doubling_doub} and Proposition \ref{trivial_doubling_prp}\ref{en:doubling_inv} plus the trivial observation that $|\{U\leq E\}|=U^{\inv}(E)$ for all $E>0$.
		
Part \ref{en:doubling_int} follows from the reverse doubling inequality \eqref{reverse_doubling}: 
\[\begin{split}
\int_0^a U(x) \,dx
&=\sum_{j=0}^{\infty} \int_{a/\lambda^{j+1}}^{a/\lambda^j} U(x) \,dx\\
&= \sum_{j=0}^{\infty} \lambda^{-j} \int_{a/\lambda}^{a} U(x/\lambda^j) \,dx\\
&\leq \sum_{j=0}^{\infty} \lambda^{-j(1+d_1)} \int_{a/\lambda}^{a} U(x) \,dx.
\end{split}\]

Let us prove part \ref{en:auxiliary}. Set $x_0 \defeq U^{\inv}(E)=|\{U\leq E\}|$ and $D \defeq \max\{bd_1,bd_2\}$. If $\lambda \geq 1$, then Proposition \ref{trivial_doubling_prp}\ref{en:doubling_pow} and part \ref{en:doubling_doub} give that
\[\begin{split}
		&(\lambda x_0)^a U(\lambda x_0)^b \exp\left(-\varepsilon\lambda x_0 \sqrt{U(\lambda x_0)}\right) \\
		&\leq  \lambda^{a+D} x_0^a U(x_0)^b \exp\left(-\varepsilon\lambda^{1+d_1/2} x_0 \sqrt{U(x_0)}\right) \\
		&= |\{U\leq E\}|^a E^b \lambda^{a+D} \exp\left(-\varepsilon\lambda^{1+d_1/2} \sqrt{E} \,|\{U\leq E\}|\right) \\
		&\leq |\{U\leq E\}|^a E^b\lambda^{a+D} \exp\left(-\delta\varepsilon\lambda^{1+d_1/2}\right) .
\end{split}\]
Hence
\[\begin{split}
	&\int_{U\geq E} x^a U(x)^b \exp\left(-\varepsilon x\sqrt{U(x)}\right) \,dx \\
	&= x_0\int_1^{+\infty} (\lambda x_0)^a U(\lambda x_0)^b \exp\left(-\varepsilon \lambda x_0\sqrt{U(\lambda x_0)}\right) \,d\lambda \\
	&\leq |\{U\leq E\}|^{a+1} E^b \int_1^{+\infty}\lambda^{a+D} \exp\left(-\delta\varepsilon\lambda^{1+d_1/2}\right) \,d\lambda \\
	&\lesssim_{d_1,d_2, \delta,a,b, \varepsilon} |\{U\leq E\}|^{a+1} E^b,
\end{split}\]
as desired.
\end{proof}

\subsection{Pointwise bounds in the exponential region}\label{ss:half_exponential}

The goal of this section is to prove the exponential decay part of Theorem \ref{thm:pointwise_eigen}. More precisely, we are given $V\in \halfpot_1(\kappa)$, $\delta>0$, $E>0$ such that	$\sqrt{E} \, |\{V\leq E\}|\geq \delta$, and a real-valued solution $u\in L^2(\Rpos)$ of
\begin{equation*}
	-u''(x) + V(x) u(x) = E u(x) \quad (x\in \Rpos),
\end{equation*}
and we want to prove that, for every $x$ such that $V(x)\geq 4E$, 
\begin{equation*} 
u(x)^2+\frac{u'(x)^2}{E}\lesssim_{\kappa,\delta}\frac{1}{|\{V\leq E\}|}e^{-c(\kappa)x\sqrt{V(x)}}\int_{E<V<2E}u^2. 
\end{equation*}
To achieve this, we invoke Agmon theory, in the form of Corollary \ref{pointwise_agmon_cor}, applied to the potential $V-E$, with $w=0$ and with the following choices of parameters:
\begin{equation*} 
	A=0,\ B=E,\ C=3E,\ \beta=1/2, \ \varepsilon=1/12,\ \gamma=3/4.
\end{equation*} 
So we obtain that, for all $x$ such that $V(x) \geq 4E$,
\begin{equation}\label{eq:prel_est_psi}
u(x)^2 + \frac{u'(x)^2}{E} \lesssim \left(\auxC_0^2 + \frac{1}{E} \right) \auxC_1 \auxD \exp\left(-\int_{V^{\inv}(2E)}^x \sqrt{V-E}\right) ,
\end{equation}
where
\begin{align*}
\auxC_j &= \int_{V > 4E} \exp\left(-\frac{1}{6} \int_{V^\inv(2E)}^y \sqrt{V-E} \right) (V(y)-E)^j \,dy,\\
\auxD &= \frac{1}{|\{E < V < 2E \}|^2} \int_{E < V < 2E} u^2.
\end{align*}

To bound the above quantities, it is useful to notice that
\begin{equation}\label{x_sqrtV}
 \int_{V^{\inv}(2E)}^x \sqrt{V-E} \simeq_\kappa	x \sqrt{V(x)} \qquad \forall x\colon \ V(x) \geq 4 E.
\end{equation}
In fact, the upper bound is trivial, and the reverse doubling inequality \eqref{reverse_doubling} of Proposition \ref{doubling_prp} applied to $V^\inv$ gives (for $V(x) \geq 4 E$) that
\[\begin{split}
\int_{V^{\inv}(2E)}^x \sqrt{V-E}
&\geq 2^{-1/2} \int_{V^{\inv}(2E)}^x \sqrt{V} \\
&\geq  2^{-1} \sqrt{V(x)} \, (x-V^{\inv}(V(x)/2))\\
& \geq 2^{-1} (1-2^{-1/\kappa}) \, x \sqrt{V(x)}.
\end{split}\]

Now, by \eqref{x_sqrtV} and Proposition \ref{doubling_prp}\ref{en:auxiliary}, we see immediately that 
\begin{equation*}
\auxC_0 \lesssim_{\kappa,\delta} |\{V\leq E\}|, \qquad  \auxC_1 \lesssim_{\kappa,\delta} |\{V\leq E\}| \, E
\end{equation*}
for every $x$ such that $V(x)\geq 4E$. Moreover, part \ref{en:sublevel_measure} of Proposition \ref{doubling_prp} yields 
\[
\auxD \lesssim_{\kappa} |\{V\leq E\}|^{-2}\int_{E<V<2E}u^2.
\]
Hence, by \eqref{eq:prel_est_psi}, in the range where $V(x)\geq 4E$,
\[\begin{split}
&u(x)^2+\frac{u'(x)^2}{E} \\
&\lesssim_{\kappa,\delta} \left( E \, |\{V\leq E\}| + \frac{1}{|\{V\leq E\}|} \right) \exp\left(-\int_{V^{\inv}(2E)}^x \sqrt{V-E}\right) \int_{E<V<2E}u^2 \\
&\leq \frac{x^2V(x)+1}{|\{V\leq E\}|} e^{-c(\kappa)x\sqrt{V(x)}} \int_{E<V<2E}u^2 \\
&\lesssim_{\kappa} \frac{1}{|\{V\leq E\}|} e^{-\frac{c(\kappa)}{2} x \sqrt{V(x)}} \int_{E<V<2E} u^2,
\end{split}\]
where we used \eqref{x_sqrtV} and the fact that $t^2 \leq 2a^{-2}\, e^{at}$ for every $t, a>0$. This completes the proof of inequality \eqref{exponential_decay_eigen} of Theorem \ref{thm:pointwise_eigen}. 

\subsection{Pointwise bounds in the classical region}

We now come to the first half of Theorem \ref{thm:pointwise_eigen}. Under the usual assumptions that $V\in \halfpot_1(\kappa)$, $\delta>0$, $E>0$, $\sqrt{E} \,|\{V\leq E\}|\geq \delta$, this time we want to prove that every real-valued solution $u\in L^2(\Rpos)$ of
\begin{equation}\label{equation2}
	-u''(x) + V(x) \, u(x) = E \, u(x) \qquad (x\in \Rpos)
\end{equation}
satisfies the bound
\begin{equation*}
\sup_{V < E/2} \left\{ u^2 + \frac{(u')^2}{E} \right\} 
\lesssim_{d_1,d_2,\delta} \frac{1}{|\{V<E\}|} \int_{E/2 < V < 2 E} u^2.
\end{equation*}
Before giving the details of its proof, let us remark that the key ingredient is the positivity of a certain derivative (see \eqref{eq:titchmarsh_derivative} below), already exploited in \cite[Theorem 8.3]{titchmarsh} (see also \cite[Theorem 7.3.1]{szego}) to deduce bounds for eigenfunctions of Schr\"odinger operators in the classical region. Since here we do not assume any boundary condition on the solutions of \eqref{equation2}, we cannot directly invoke the theory in \cite{titchmarsh}.	

We need the following energy estimate. 

\begin{lem}\label{kinetic_lem}
Let $V\in \halfpot_1(\kappa)$, $\delta>0$, $E>0$ be such that $\sqrt{E} \, |\{V\leq E\}|\geq \delta$ holds, and let $u\in L^2(\Rpos)$ be a real-valued solution of \eqref{equation2}. Then
\begin{equation*}
\int_{V\geq E/2} (u')^2 \lesssim_{\kappa, \delta} E \int_{E/2 \leq V\leq 2E} u^2.
\end{equation*}	
\end{lem}

\begin{proof}
By \cite[Inequality (2.6)]{nirenberg_1959} and \eqref{equation2}, 
\[\begin{split}
\int_{V \geq E/2} (u')^2 &\lesssim \sqrt{\int_{V \geq E/2} u^2 \int_{V\geq E/2} (u'')^2}\\
&\lesssim \sqrt{\int_{V \geq E/2} u^2 \int_{V\geq E/2} V^2 u^2} + E \int_{V \geq E/2} u^2.
\end{split}\]
By the already established exponential decay \eqref{exponential_decay_eigen}, 
\[\begin{split}
\int_{V\geq E/2} V^2 u^2
&\leq 16 E^2 \int_{E/2 \leq V \leq 4 E} u^2 + \int_{V\geq 4 E} V^2 u^2 \\
&\lesssim_{\kappa,\delta} E^2 \int_{E/2 \leq V \leq 4 E} u^2 \\
&\qquad+ \frac{1}{|\{V\leq E\}|} \int_{V\geq 4E} V(x)^2 e^{-c x \sqrt{V(x)}} \,dx \cdot \int_{E \leq V \leq 2 E} u^2 \\
&\lesssim_{\kappa,\delta} E^2 \int_{V\geq E/2} u^2,
\end{split}\]
where Proposition \ref{doubling_prp}\ref{en:auxiliary} was used in the last inequality. Thus, 
\begin{equation*}
\int_{V\geq E/2} (u')^2 \lesssim_{\kappa,\delta} E \int_{V\geq E/2} u^2\lesssim_{\kappa, \delta} E \int_{E/2\leq V\leq 2E} u^2,
\end{equation*} 
where we used once more Agmon theory, in the form of Corollary \ref{agmon_second_cor}.
\end{proof}

Thanks to the energy estimate of Lemma \ref{kinetic_lem}, we can run the positivity argument alluded to above. Consider the function 
\begin{equation*}
g \defeq u^2+\frac{(u')^2}{E-V}
\end{equation*}
and notice that 
\[\begin{split}
\int_{E/2 < V < 3E/4} g 
&\leq	\int_{E/2 < V < 3E/4} u^2 + \frac{4}{E} \int_{E/2 < V < 3E/4} (u')^2 \\
&\lesssim_{\kappa,\delta} \int_{E/2 < V < 2 E} u^2,
\end{split}\]
where the last bound follows from Lemma \ref{kinetic_lem}. A simple computation using the equation \eqref{equation2} yields
\begin{equation}\label{eq:titchmarsh_derivative}
g'(x) = \frac{V'(x) \, u'(x)^2}{(E-V(x))^2} >0 \qquad \forall x \tc V(x) \neq E.
\end{equation}
It follows that the maximum of $g$ on the interval $\{V < E/2\}$ is attained at the right endpoint and it must be less than or equal to
\[
\frac{1}{|\{E/2 < V < 3E/4\}|}\int_{E/2 < V < 3E/4} g 
\lesssim_{\kappa,\delta} \frac{1}{|\{V<E\}|} \int_{E/2 < V < 2 E} u^2
\]
(here Proposition \ref{doubling_prp}\ref{en:sublevel_measure} was also used). This is exactly the content of inequality \eqref{classical_eigen} of Theorem \ref{thm:pointwise_eigen}, whose proof is now complete. 

\subsection{Integral bounds}
Here we prove Theorem \ref{thm:Vpowers_prp}.

Let $W : \Rpos \to \Rnon$, $a,b \in \RR$ and $C > 0$ be such that
\begin{equation}\label{eq:xV_int_assumption2}
	W(x) \leq C x^a V(x)^b \quad\text{and}\quad
	\int_0^x W  \leq C x^{a+1} V(x)^b \qquad\forall x > 0.
\end{equation}
Theorem \ref{thm:Vpowers_prp}\ref{en:Vpowers_upp}, namely the integral inequality 
\[
\int_{\Rpos} W  \left( u^2 + \frac{(u')^2}{E} \right) 
\lesssim_{a,b,\delta,\kappa} C \, |\{V\leq E\}|^a E^b \int_{\Rpos} u^2 ,
\]
follows from the pointwise bounds of Theorem \ref{thm:pointwise_eigen}. Let us see the details. 

Clearly we may assume $C=1$. Define
\[
v \defeq u^2 + \frac{(u')^2}{E}.
\]
We decompose our integral as follows:
\[
	\int_{\Rpos} W  v 
	= \int_{V<E/2} + \int_{E/2<V<4E} + \int_{V>4E}
	\eqdef \romI + \romII + \romIII.
\]
For the first term, the classical-region bound \eqref{classical_eigen} of Theorem \ref{thm:pointwise_eigen} gives that
\[
\romI
\lesssim_{\kappa,\delta} 
\frac{1}{|\{V<E\}|} \int_{\Rpos} u^2 \cdot \int_{V<E/2} W .
\]
Now, by \eqref{eq:xV_int_assumption2}
and Proposition \ref{doubling_prp},
\[
\int_{V<E/2} W \leq |\{V<E/2\}|^{a+1} (E/2)^b \simeq_{\kappa,a,b} |\{V \leq E\}|^{a+1} E^b.
\]
Hence
\[
\romI
\lesssim_{\kappa,\delta,a,b} \{V\leq E\}|^a E^b \int_{\Rpos} u^2.
\]
Next, since $W(x) \leq x^a V(x)^b$,
\[\begin{split}
	\romII
	&\lesssim_{b}  E^b\max\{|\{V\leq E/2\}|^a, |\{V\leq 4E\}|^a\} \int_{E/2<V<4E} v\\
	&\lesssim_{\kappa,a,b} |\{V\leq E\}|^a E^b \int_{\Rpos} u^2 .
\end{split}\]
In the last step we used the energy estimate of Lemma \ref{kinetic_lem} and the usual doubling property of Proposition \ref{doubling_prp}. Finally, the exponential decay part of Theorem \ref{thm:pointwise_eigen} and Proposition \ref{doubling_prp}\ref{en:auxiliary} take care of the tail of the integral: 
\[\begin{split}
	\romIII
	&\lesssim_{\kappa, \delta, a,b}  \frac{1}{|\{V\leq E\}|}\int_{V\geq 4E} x^a V(x)^b e^{-cx\sqrt{V(x)}} \,dx \cdot \int_{\Rpos} u^2 \\
	&\lesssim_{\kappa, \delta, a,b}  |\{V\leq E\}|^a E^b \int_{\Rpos} u^2 ,
\end{split}\]
and this completes the proof of part \ref{en:Vpowers_upp} of Theorem \ref{thm:Vpowers_prp}.

\smallskip

We now verify part \ref{en:Vpowers_ass}, that is, the assertion that the assumption \eqref{eq:xV_int_assumption2} holds whenever $W(x) = x^a V(x)^b$ and $a+\kappa^{-1} b>-1$, with $C = C(a,b,\kappa)$.
Indeed, note that such $W$ lies in $\doub(a+b\kappa^{-1}, a+b\kappa)$ by parts \ref{en:doubling_pow} and \ref{en:doubling_prod} of Proposition \ref{trivial_doubling_prp}. Hence, if 
$a+b\kappa^{-1}>-1$, then the desired assertion 
follows from parts \ref{en:doubling_doub} and \ref{en:doubling_int} of Proposition \ref{doubling_prp}

\smallskip

We are left with the proof of part \ref{en:Vpowers_low} of Theorem \ref{thm:Vpowers_prp}. For all $a,b \in \RR$,
\[\begin{split}
	\int_{\Rpos} x^a V(x)^b u(x)^2 \,dx 
	&\geq \int_{E/2 \leq V \leq 2 E} x^a V(x)^b u(x)^2 \,dx \\
	&\gtrsim_b E^b \min\{|\{V\leq E/2\}|^a,|\{V\leq 2E\}|^a\} \int_{E/2 \leq V \leq 2 E} u^2 \\
	&\gtrsim_{\kappa,a,b} E^b \, |\{V\leq E\}|^a  \int_{E/2 \leq V\leq 2E} u^2.
\end{split}\]
The desired estimate follows by Corollary \ref{agmon_second_cor} and the pointwise bounds in the classical region, which imply that
\[
\int_{E/2 \leq V \leq 2 E} u^2 \gtrsim_{\kappa,\delta} \int_{\Rpos} u^2.
\]

\section{One-dimensional Schr\"odinger operators}\label{basic_sec}

\subsection{Summary of the results}\label{summary_line_sec}

We now switch to the study of Schr\"odinger equations on the real line $\RR$. We begin by introducing three useful classes of ``single-well" potentials on $\RR$.

\begin{dfn}\label{dfn:potentials_singlewell}
Let $\pot$ be the class of continuous nonnegative potentials $V : \RR \to \Rnon$ such that:
\begin{enumerate}
	\item $V(0)=0$,
	\item $V$ is strictly increasing for $x\geq0$ and strictly decreasing for $x\leq 0$,
	\item $\lim_{x \to \pm\infty}V(x)=+\infty$.
\end{enumerate}
Let $\Tpot\subset \pot$ be the subclass of potentials satisfying the additional temperate growth condition
\begin{enumerate}[resume]
	\item $V(x) \leq Ce^{M|x|}$ for some $C,M > 0$ and all $x \in \RR$.
\end{enumerate}
Finally, let $\Tpotone$ be the class of potentials $V \in \Tpot$ which are continuously differentiable on $\RR \setminus \{0\}$ and such that
\begin{enumerate}[resume]
\item $xV'(x) \leq C e^{M|x|}$ for some $C,M>0$ and all $x \in \RR\setminus \{0\}$.
\end{enumerate}
\end{dfn}

Fix a potential $V\in\pot$. As is well-known (see Remark \ref{rem:schr_selfadj} below), the Schr\"odinger operator 
\begin{equation}\label{eq:def_opH}
	\opH[V] \defeq -\partial_x^2+V(x),
\end{equation}
initially defined on $C^\infty_c(\RR)$, is essentially self-adjoint and its unique self-adjoint extension has pure point spectrum consisting of a divergent sequence of simple eigenvalues $\{E_n(V) \tc n\geq 1\}$. We denote by $\psi_n(x;V)$ the eigenfunction of $\opH[V]$ corresponding to the eigenvalue $E_n(V)$, i.e., 
\begin{equation*}
-\psi_n''(x;V) + V(x) \, \psi_n(x;V) = E_n(V) \, \psi_n(x;V) \qquad (x\in \RR),
\end{equation*} 
normalised by the two conditions 	
\begin{equation}\label{eigenfunction_normalization}
\begin{aligned}
\int_\RR\psi_n(x;V)^2 \,dx &= 1, \\
 \psi_n(x;V)&>0 \quad\text{for all $x>0$ large enough}
\end{aligned}
\end{equation} 
(see Section \ref{sec:sturm} for details). Whenever the potential $V$ is clear from the context, the lighter notation $E_n$ and $\psi_n(x)$ will be used. 

\begin{rem}\label{rem:schr_selfadj}
Essential self-adjointness of $\opH[V]$ holds more generally for locally square-integrable nonnegative potentials \cite[Theorem X.28]{reed-simon-II}, and the spectrum of $\opH[V]$ consists entirely of a divergent sequence of simple eigenvalues if and only if
\[
\lim_{x \to \infty} \int_{x-r}^{x+r} V(y) \,dy = +\infty
\]
for every $r>0$ (\cite{molchanov_1953}, cited in the introduction of \cite{mazya_shubin_2005}). 
\end{rem}

Under the aforementioned temperate growth conditions on the potential $V$, we shall derive the following important relations.

\begin{thm}[Virial theorem]\label{thm:virial}
Let $V \in \Tpot$ and $n \in \Npos$. Then $E_n(t V)$ is real-analytic as a function of $t>0$ and 
\begin{equation}\label{virial_F}
F_n(V) \defeq \partial_t|_{t=1} E_n(t V)  = \int_\RR V(x) \,\psi_n(x;V)^2 \,dx.
\end{equation}
If moreover $V \in \Tpotone$, then
\begin{equation}\label{virial_identity}
\int_\RR x V'(x) \,\psi_n(x;V)^2 \,dx = 2 \int_\RR \psi_n'(x;V)^2 \,dx.
\end{equation}
\end{thm}

\begin{rem} Under the additional assumption $x V'(x) \simeq V(x)$, the identity \eqref{virial_identity} tells us that potential and kinetic energy associated to eigenfunctions are comparable. This justifies the terminology ``virial theorem'', cf.\ \cite{fock,weidmann,georgescu_virial_1999}.
\end{rem}

We now introduce a more stringent growth condition on potentials $V$, which corresponds to the assumption \eqref{eq:2dassumptions_doubling} discussed in the introduction.
In what follows, for any $f : \RR \to \CC$, we denote by $f_\oplus,f_\ominus : \Rnon \to \CC$ the functions defined by
\begin{equation}\label{eq:pm_parts}
	f_\oplus(x) = f(x), \qquad f_\ominus(x) = f(-x)
\end{equation}
for all $x \in \Rnon$.

\begin{dfn}
Let $\kappa\geq 1$. We denote by $\pot_1(\kappa)$ the class of potentials $V : \RR \to [0,+\infty)$ such that:
\begin{enumerate}[label=(\arabic*)]
	\item $V(0)=0$;
	\item\label{en:p1_compar} $V(-x)\leq \kappa V(x)$ for every $x\in \RR$;
	\item\label{en:p1_doubl} both $V_\oplus$ and $V_\ominus$ lie in $\halfpot_1(\kappa)$.
\end{enumerate} 
\end{dfn}
Notice that condition \ref{en:p1_compar} amounts to comparability of $V_\oplus$ and $V_\ominus$, while condition \ref{en:p1_doubl} amounts to comparability of $x V'(x)$ and $V(x)$ (see the definition of $\halfpot_1(\kappa)$ at the beginning of Section \ref{s:halflineregular}). From the doubling properties discussed in Proposition \ref{doubling_prp}, one immediately deduces that $\pot_1(\kappa) \subseteq \Tpotone$.

Under the assumption $V \in \pot_1(\kappa)$, we shall prove the following fundamental estimates for eigenvalues and eigenfunctions of $\opH[V]$.

\begin{thm}[Eigenvalue and eigenvalue gaps estimates]\label{thm:eigenvalues}
Let $V\in \pot_1(\kappa)$. Then, for every $n \geq 1$,
\begin{equation}\label{approx_bohr_sommerfeld}
E_n(V)^{1/2} \, |\{V \leq E_n(V)\} | \simeq_{\kappa} n
\end{equation}
and, if $1 \leq m \leq n$, then
\begin{equation*}
E_n(V) - E_m(V) \simeq_{\kappa} \frac{E_n(V)}{n} (n-m).
\end{equation*}
\end{thm}

\begin{thm}[Eigenfunction estimates]\label{thm:pointwise_eigenfcts}
Let $V\in \pot_1(\kappa)$.
\begin{enumerate}[label=(\roman*)]
\item\label{en:pointwise_eigenfcts} For every $n\in\NN$ and $x\in \RR$ such that $V(x)\geq 4E_n(V)$,
\begin{equation*}
	\psi_n(x;V)^2 + \frac{\psi_n'(x;V)^2}{E_n(V)} \lesssim_{\kappa}\frac{1}{|\{V\leq E_n(V)\}|} e^{-c(\kappa)|x|\sqrt{V(x)}}.
\end{equation*}
\item\label{en:integral_eigenfcts} If $W : \RR \to \Rnon$,  $a,b\in \RR$ and $C > 0$ are such that
\begin{equation}\label{eq:xV_int_assumption_bil}
	W(x) \leq C \, |x|^a V(x)^b \quad\text{and}\quad
	\int_{-|x|}^{|x|} W(t) \,dt \leq C \, |x|^{a+1} V(x)^b \qquad\forall x \in \RR,
\end{equation}
then
\begin{equation*}
	\int_{\RR} W(x) \left(\psi_n(x;V)^2 + \frac{\psi_n'(x;V)^2}{E_n(V)}\right) \,dx \lesssim_{a,b,\kappa}  C \, |\{V\leq E_n(V)\}|^a E_n(V)^b.
\end{equation*}
\item\label{en:ass_weight_eigenfcts} The assumption \eqref{eq:xV_int_assumption_bil} holds whenever $W(x) = |x|^a V(x)^b$ and $a+\kappa^{-1} b>-1$, with $C = C(a,b,\kappa)$.
\end{enumerate}
\end{thm}

The rest of this section is devoted to the proofs of the above results.

\subsection{Auxiliary solutions of one-dimensional Schr\"odinger equations}\label{sec:sturm}

Let $V \in \pot$. By Sturm--Liouville theory, for every $E>0$, there exists a unique global solution $v(x)=v(x;E)$ of the stationary Schr\"odinger equation
\begin{equation}\label{eq:ODE_RR}
-v''(x) + V(x) \, v(x) = E \, v(x) \quad (x\in \RR)
\end{equation}
satisfying the following properties:
\begin{enumerate}[label=(\alph*)]
\item $v$ is recessive at $+\infty$, that is
\begin{equation*}
\lim_{x \to +\infty} v(x;E) = \lim_{x \to +\infty} \partial_x v(x;E) = 0.
\end{equation*} 
\item $v(x;E) > 0$ and $\partial_x v(x;E) < 0$ for $x$ large enough.
\item $\int_0^\infty v(x;E)^2 \,dx = 1$. 
\end{enumerate}
The spectrum of the self-adjoint operator $\opH[V] \defeq -\partial_x^2+V(x)$ coincides with the set of $E>0$ with the property that $v(\cdot;E)\in L^2(\RR)$. The eigenfunction $\psi_n(x; V)$, normalized as in \eqref{eigenfunction_normalization}, is the unique positive multiple of $v(x;E_n)$ whose $L^2$ norm on the whole real line is $1$. 

\subsection{The virial theorem}
Here we prove Theorem \ref{thm:virial}. Let us introduce the Banach spaces
\[
\banE^k_\beta = \left\{ u \in C^k(\RR;\RR) \tc \sup_{x \in \RR} e^{\beta |x|} \sum_{j=0}^k |u^{(j)}(x)| < \infty \right\},
\] where $k \in \NN$ and $\beta \geq 0$.

\begin{prp}\label{prp:one_parameter_tV}
Let $V \in \pot$. 
\begin{enumerate}[label=(\roman*)]
\item\label{en:one_parameter_tV_E} For all $n \in \Npos$, the function $t \mapsto E_n(tV)$ is real-analytic from $\Rpos$ to $\Rpos$.
\item\label{en:one_parameter_tV_psi} If $V \in \Tpot$, then, for all $\beta \geq 0$ and $n \in \Npos$, the map $t \mapsto \psi_n(\cdot;tV)$ is a real-analytic map from $\Rpos$ to the Banach space $\banE^1_\beta$.
\end{enumerate}
\end{prp}
\begin{proof}
According to \cite[Chapter VII, Section 4.8, Example 4.24, p.\ 409]{kato}, the operators $\opH[tV]$ constitute a ``selfadjoint holomorphic family of type (B)'', hence the eigenvalues $E_n(tV)$ are analytic functions of $t$ for all $n \in \Npos$ \cite[Chapter VII, Section 1.3, Theorem 1.8, p.\ 370]{kato}. This proves part \ref{en:one_parameter_tV_E}.

Moreover, combining this with Theorem \ref{thm:smoothdependence} easily gives that, for all $n \in \NN$ and $\alpha\geq 0$, the map $t \mapsto \psi_n(\cdot;tV)$ is real-analytic from $\Rpos$ to the Banach space
\[
\banD_{V,\alpha} \defeq \{ u \in C^2(\RR;\RR) \tc u_{\oplus} \in \banD_{V_\oplus,\alpha}, \, u_{\ominus} \in \banD_{V_{\ominus},\alpha} \}.
\]
If $V \in \Tpot$, then by Proposition \ref{u'_prp} the space $\banD_{V,\alpha}$ is continuously embedded in $\banE_\beta^1$ for all $\beta \geq 0$, and part \ref{en:one_parameter_tV_psi} is proved.
\end{proof}

\begin{cor}\label{cor:one_parameter_Vt}
Let $V \in \pot$. 
\begin{enumerate}[label=(\roman*)]
\item For all $n \in \Npos$, the function $t \mapsto E_n(V(t\cdot))$ is real-analytic from $\Rpos$ to $\Rpos$.
\item If $V \in \Tpot$, then, for all $\beta \geq 0$ and $n \in \Npos$, the map $t \mapsto \psi_n(\cdot;V(t\cdot))$ is a $C^1$ map from $\Rpos$ to the Banach space $\banE^0_\beta$.
\end{enumerate}
\end{cor}

\begin{proof}
Both assertions follow from Proposition \ref{prp:one_parameter_tV} and the relations
\begin{align}
	E_n(V(t\cdot)) &= t^2E_n(t^{-2}V), \label{eq:scaling_energy}\\
	\psi_n(x;V(t\cdot)) &= t^{1/2} \psi_n(tx;t^{-2} V),
\end{align}
which are obtained rescaling the variable in the equation \eqref{eq:ODE_RR} with $E=E_n(V(t\cdot))$ and $v(x)=\psi_n(x;V(t\cdot))$.
\end{proof}

By the previous results, if $V_t$ is either of the families of potentials $tV$ or $V(t\cdot)$, then the corresponding eigenvalues $E_n(V_t)$ are analytic functions of $t$ for any $n \in \Npos$. In the next proposition, we obtain a formula for the derivative of the eigenvalues $\partial_t(E_n(V_t))$ involving $\partial_t V_t$. 

\begin{prp}\label{diff_parameter_prp}
Let $V \in \Tpot$. Then the identity
\begin{equation*}
\partial_t(E_n(V_t)) = \langle (\partial_t V_t) \psi_n(\cdot; V_t), \psi_n(\cdot; V_t) \rangle.
\end{equation*}
holds with $V_t = tV$. The same identity holds for $V_t = V(t\cdot)$ provided $V \in \Tpotone$.
\end{prp}
\begin{proof}
From Proposition \ref{prp:one_parameter_tV}, Corollary \ref{cor:one_parameter_Vt}, and
\begin{equation}\label{eq:pot_tderiv}
	\partial_t(tV(x)) = V(x), \qquad \partial_t(V(tx))=xV'(tx),
\end{equation}
we easily deduce that, under our assumptions on $V$, any of the functions $\psi_n(x;V_t)$, $V_t(x) \psi_n(x;V_t)$, and $\psi_n''(x;V_t) = (V_t(x)-E_n(V_t)) \psi_n(x,t)$, as well as their first-order $t$-derivatives, decay exponentially as $|x| \to \infty$ with locally uniform bounds in $t$. This justifies differentiations under the integral sign and integrations by parts in the following argument.

First, from the $L^2$-normalization of eigenfunctions we deduce that
\begin{equation*}
\langle\partial_t\psi_n(\cdot; V_t), \psi_n(\cdot ;V_t) \rangle=0.
\end{equation*}
Now, 
\[\begin{split}
\partial_t(E_n(V_t)) 
&= \partial_t \langle \opH[V_t]\psi_n(\cdot; V_t), \psi_n(\cdot; V_t)\rangle \\
&= \langle (\partial_t V_t) \psi_n(\cdot; V_t), \psi_n(\cdot; V_t)\rangle \\
&\quad+ \langle \opH[V_t]\partial_t\psi_n(\cdot; V_t), \psi_n(\cdot; V_t)\rangle \\
&\quad+ \langle \opH[V_t]\psi_n(\cdot; V_t),\partial_t \psi_n(\cdot; V_t)\rangle.
\end{split}\]
On the other hand, the last two summands vanish, since
\[\begin{split}
&\langle \opH[V_t]\partial_t\psi_n(\cdot; V_t), \psi_n(\cdot; V_t)\rangle \\
&= \langle \partial_t\psi_n(\cdot; V_t),\opH[V_t] \psi_n(\cdot; V_t)\rangle  \\
&= E_n(V_t)\langle \partial_t\psi_n(\cdot; V_t),\psi_n(\cdot; V_t)\rangle = 0,
\end{split}\]
and we are done.
\end{proof}

We now have all the ingredients to prove Theorem \ref{thm:virial}. Formula \eqref{virial_F} is an immediate consequence of Proposition \ref{diff_parameter_prp} applied to $V_t = tV$.
To prove formula \eqref{virial_identity}, we apply Proposition \ref{diff_parameter_prp} to $V_t = V(t\cdot)$, combined with the identity
\begin{equation*}
\partial_t|_{t=1} E_n(V(t\cdot)) = 2E_n(V) - 2F_n(V),
\end{equation*}
which is an elementary consequence of \eqref{eq:scaling_energy}. 

\subsection{Elementary doubling properties of regular potentials}

We record here a couple of useful elementary properties of potentials in the class $\pot_1(\kappa)$. 

\begin{prp}\label{pot1_trivial_prp}
Let $V \in \pot_1(\kappa)$. Then, for all $a \in \Rpos$, 
\begin{equation}\label{pot_1_trivial_1}
	|\{ V_\oplus \leq a\}| \simeq_\kappa |\{ V_\ominus \leq a\}| \simeq_\kappa |\{V\leq a\}|,
\end{equation} 
and
\begin{equation}\label{pot_1_trivial_2}
	|\{V\leq E\}| \simeq_{\kappa, c} |\{V\leq cE\}| \qquad\forall E>0,\ c\in(0,1).
\end{equation}
\end{prp}

\begin{proof}
By the definition of $\pot_1(\kappa)$ and part \ref{en:sublevel_measure} of Proposition \ref{doubling_prp},
\[
	|\{ V_\oplus \leq a\}| \leq |\{ V_\ominus \leq \kappa a\}| \lesssim_\kappa |\{ V_\ominus \leq a\}|.
\]
The reverse approximate inequality is proved analogously. Since
\begin{equation}\label{eq:sum_sublevel}
		|\{V\leq a\}| = |\{ V_\oplus \leq a\}| + |\{ V_\ominus \leq a\}|,
\end{equation}
the approximate identities \eqref{pot_1_trivial_1} follow, while \eqref{pot_1_trivial_2} is a consequence of \eqref{pot_1_trivial_1} and part \ref{en:sublevel_measure} of Proposition \ref{doubling_prp}.
\end{proof}

\subsection{Eigenvalue estimates}

In this section we prove the first half of Theorem \ref{thm:eigenvalues}, namely the approximate identity \eqref{approx_bohr_sommerfeld}, under the assumption $V \in \pot_1(\kappa)$. This will follow from the inequalities of Bohr--Sommerfeld type contained in the theorem below, which are valid more generally for potentials in the class $\pot$. 

\begin{thm}[Bohr-Sommerfeld inequalities]\label{bohr_sommerfeld_thm}
Let $V\in \pot$. Then 
\begin{equation}\label{first_eigenvalue_bohr_sommerfeld}
\sqrt{E_1(V)} \,|\{V< \varepsilon^{-2}E_1(V)\}|\geq \frac{4}{27}(1-\varepsilon)^3 \pi\qquad\forall \varepsilon\in (0,1),
\end{equation}
\begin{equation}\label{lower_bohr_sommerfeld}
\sqrt{E_n(V)} \,|\{V< E_n(V)\}|\geq \pi(n-1)\qquad\forall n\geq 1,
\end{equation}
and
\begin{equation}\label{upper_bohr_sommerfeld}
\sqrt{\frac{E_n(V)}{1+t}}\left|\left\{V< \frac{E_n(V)}{1+t}\right\}\right|\leq \frac{\pi n}{\sqrt{t}}\qquad\forall n\geq 1, \ \forall t>0.
\end{equation}	
\end{thm}

It is clear that the desired approximate identity \eqref{approx_bohr_sommerfeld} can be obtained by combining the three inequalities above and Proposition \ref{pot1_trivial_prp}. 

\begin{proof}
Inequality \eqref{first_eigenvalue_bohr_sommerfeld} can be proved by using a formulation of the uncertainty principle due to Donoho and Stark \cite{donoho_stark_1989} (see also \cite[Section 8]{folland_sitaram_1997}). To state it, we need some terminology: $\psi\in L^2(\RR)$ is said to be $\varepsilon$-concentrated on a measurable set $A\subseteq \RR$ if $\int_{\RR\setminus A}|\psi|^2\leq \varepsilon^2\int_{\RR}|\psi|^2$. Then, the Donoho--Stark inequality says that if $\psi\in L^2(\RR)$ is $\varepsilon$-concentrated on $A$ and $\widehat{\psi}$ is $\delta$-concentrated on $B$, where $\delta,\varepsilon>0$ and $\delta+\varepsilon < 1$, then
\begin{equation}\label{donoho_stark}
|A|\, |B|\geq 2\pi (1-\varepsilon-\delta)^2;
\end{equation}
here $\widehat\psi(\xi) = (2\pi)^{-1/2} \int_{\RR} \psi(x) \, e^{-i \xi \cdot x} \,dx$. Let $V\in \pot$ and take $\psi =  \psi_1$, the ground-state eigenfunction of $\opH[V]$. Then, for any fixed $\varepsilon\in (0,1)$, 
\begin{equation*}
E_1 \int_\RR \psi^2 
= \int_\RR (\psi')^2 
+\int_\RR V\psi^2\geq \varepsilon^{-2} E_1 \int_{V\geq \varepsilon^{-2} E_1}\psi^2,
\end{equation*}
that is, the ground-state eigenfunction is $\varepsilon$-concentrated on the sublevel set $A=\{V< \varepsilon^{-2} E_1\}$. By \eqref{donoho_stark}, if $\widehat{\psi}$ is $\delta$-concentrated on $(-b,b)$, where $\delta\in (0,1-\varepsilon)$, then necessarily $|A| \, b\geq \pi (1-\varepsilon-\delta)^2$. Therefore, for any $b<\pi (1-\varepsilon-\delta)^2/|A|$,
\[\begin{split}
E_1 \int_\RR \psi^2 
&= \int_\RR (\psi')^2 
+\int_\RR V \psi^2 \\
&\geq \int_\RR \xi^2 |\widehat{\psi}(\xi)|^2 \,d\xi \\
&\geq b^2\int_{\RR\setminus (-b,b)} |\widehat{\psi}(\xi)|^2 \,d\xi \\
&\geq b^2\delta^2 \int_{\RR} |\widehat{\psi}(\xi)|^2 \,d\xi = b^2 \delta^2 \int_{\RR} \psi^2.
\end{split}\]
Thus,
\begin{equation*}
\sqrt{E_1} \,|\{V< \varepsilon^{-2} E_1\}|\geq \pi \delta(1-\varepsilon-\delta)^2
\end{equation*}
and optimizing in $\delta$ gives \eqref{first_eigenvalue_bohr_sommerfeld}.

\smallskip

For the remaining bounds, we make use of the classical inequality
\begin{equation}\label{eq:wirtinger}
\int_0^1 (f')^2 \geq \pi^2 \int_0^1 f^2,
\end{equation}
valid for all $f \in W^{1,2}([0,1];\RR)$ which are either vanishing at the endpoints (see, e.g., \cite[p.\ 47]{dym_mckean_1972}) or with zero mean (see, e.g., \cite[Sect.\ 1.1]{kuznetsov_nazarov_2015}).

Given $n\geq 1$ and $E>0$, we decompose $\{V<E\}$ into $n$ subintervals of equal length $I_1,\ldots, I_n$. We define $M(n,E)$ as the linear subspace of $L^2(\RR;\RR)$ defined by the $n$ conditions $\int_{I_j}\psi=0$ ($j=1,\ldots, n$). Then, by \eqref{eq:wirtinger},
\begin{multline*}
\int_\RR (\psi')^2  + \int_\RR V \psi^2 \geq \sum_{j=1}^n \int_{I_j} (\psi')^2 + E \int_{V\geq E} \psi^2 \\
\geq \sum_{j=1}^n \frac{\pi^2}{|I_j|^2} \int_{I_j} \psi^2  + E \int_{V\geq E} \psi^2 \geq \frac{\pi^2n^2}{|\{V<E\}|^2} \int_{V<E} \psi^2  + E \int_{V\geq E} \psi^2.
\end{multline*}
for every $\psi\in M(n, E)$. Now, for every $n$, there is a unique $E=\widetilde{E}_n>0$ such that $\pi^2 n^2 /|\{V<E\}|^2=E$. So, by the max-min theorem \cite[Chapter XII]{lieb_loss_2001}, 
\begin{equation*}
E_{n+1} = \sup_{\substack{M\subseteq L^2(\RR;\RR) \\ \codim M=n}}\inf_{\substack{\psi\in M \\ \int\psi^2=1}} \int_\RR (\psi')^2  + \int_\RR V \psi^2 \geq \widetilde{E}_n,
\end{equation*}
where the latter inequality follows by taking $M = M(n,\widetilde{E}_n)$.
This implies that
\begin{equation*}
E_{n+1} \, | \{V<E_{n+1}\} |^2 \geq \widetilde{E}_n \, |\{V<\widetilde{E_n}\}|^2 = \pi^2 n^2 \qquad\forall n\geq 1,
\end{equation*}
and \eqref{lower_bohr_sommerfeld} is proved.

\smallskip

Next, let us prove \eqref{upper_bohr_sommerfeld}. Let $\eta(x) \defeq \sqrt{2}\sin(\pi x) \chr_{[0,1]}(x)$ and notice that $\int \eta^2=1$ and $\int(\eta')^2=\pi^2$ (that is, $\eta$ extremises \eqref{eq:wirtinger}). Given any finite interval $I=[s,t]$, we let $\eta_I(x) \defeq \frac{1}{\sqrt{t-s}}\eta\left(\frac{x-s}{t-s}\right)$. Of course, $\int \eta_I^2=1$ and $\int(\eta_I')^2=\pi^2/|I|^2$.
	
Given $n\geq 1$ and $E>0$, we decompose $\{V<E\}=\bigcup_{j=1}^nI_j$ as before and denote by $L(n, E)$ the $\RR$-linear span of $\{\eta_{I_j} \tc j=1,\ldots, n\}$. Then, it is easily checked that
\begin{equation*}
\int_\RR (\psi')^2  +\int_\RR V \psi^2 \leq \left( \frac{\pi^2n^2}{|\{V<E\}|^2} + E \right) \int_\RR \psi^2
\end{equation*}
for every $\psi\in L(E,n)$. Now, for every $n\geq 1$, there is a unique $E=\widetilde{E}_n>0$ such that $\pi^2 n^2 /|\{V<E\}|^2=t E$, where $t>0$.
So, by the min-max theorem,
\begin{equation*}
E_n=\inf_{\substack{L\subseteq L^2(\RR;\RR) \\ \dim L=n}} \sup_{\substack{\psi\in L \\ \int\psi^2=1}} \int_\RR (\psi')^2  + \int_\RR V\psi^2 \leq (1+t)\widetilde{E}_n,
\end{equation*}
where the inequality follows by taking $L = L(E,n)$.
This readily implies that
\begin{equation*}
\frac{E_n}{1+t}\left|\left\{V<\frac{E_n}{1+t}\right \}\right|^2 \leq \widetilde{E}_n \,|\{V<\widetilde{E}_n\}|^2 = \frac{\pi^2 n^2}{t},
\end{equation*}
and we are done.
\end{proof}

We record here a consequence of Theorem \ref{bohr_sommerfeld_thm} for doubling potentials, which will be used later.

\begin{prp}\label{prp:eigenvalue_doub}
	Let $V\in \pot_1(\kappa)$.
\begin{enumerate}[label=(\roman*)]	
\item\label{en:energy_doubling_eigen} For every $m\leq n$,
\begin{equation*}
		\left(\frac{m}{n}\right)^{2\kappa/(2+\kappa)} \lesssim_\kappa \frac{E_m(V)}{E_n(V)}\lesssim_\kappa \left(\frac{m}{n}\right)^{2/(2\kappa+1)}. 
\end{equation*}
\item\label{en:energy_doubling_lbd} For all $E \in \Rpos$, if $E \geq E_1(V)$ then
\[
V(E^{-1/2}) \lesssim_\kappa E.
\]
\end{enumerate}
\end{prp}
\begin{proof}
We first prove part \ref{en:energy_doubling_eigen}. Consider the function
\[
W(x) \defeq \sqrt{x}V_\oplus^{\inv}(x) + \sqrt{x}V_\ominus^{\inv}(x) \qquad(x>0).
\]
By Proposition \ref{trivial_doubling_prp}, $W\in \doub(1/2+1/\kappa,1/2+\kappa)$ and therefore the doubling and reverse doubling inequalities of Proposition \ref{doubling_prp} imply that
\begin{equation*}
		\left(\frac{E_n}{E_m}\right)^{1/2+1/\kappa} W(E_m)\leq W(E_n)\leq \left(\frac{E_n}{E_m}\right)^{1/2+\kappa} W(E_m)
\end{equation*}
for all $m,n \in \Npos$ with $m \leq n$. On the other hand, by the approximate Bohr--Sommerfeld identity \eqref{approx_bohr_sommerfeld}, 
\begin{equation}\label{eq:W_est}
W(E_n)=\sqrt{E_n} \,|\{V\leq E_n\}| \simeq_\kappa n
\end{equation}
for all $n \in \Npos$, and the desired estimates follow.

We now prove that, for all $n \in \Npos$, 
\begin{equation}\label{eq:energy_doubling_inv}
V(n E_n^{-1/2}) \simeq_{\kappa} E_n.
\end{equation}
Indeed, the estimate \eqref{eq:W_est} and  Proposition \ref{pot1_trivial_prp} also give that
\[
V_\oplus^{\inv}(E_n) = |\{V_\oplus \leq E_n\}| \simeq_{\kappa} n E_n^{-1/2}.
\]
Proposition \ref{doubling_prp} applied to $V_\oplus$ then yields  \eqref{eq:energy_doubling_inv}.

Finally, the case $n=1$ of  \eqref{eq:energy_doubling_inv} gives part \ref{en:energy_doubling_lbd} when $E = E_1$, and the general case follows from the monotonicity of $V$.
\end{proof}

\subsection{Eigenfunction estimates}

In this section we prove Theorem \ref{thm:pointwise_eigenfcts}. Let us begin with the pointwise bounds of part \ref{en:pointwise_eigenfcts}. Let $V\in \pot_1(\kappa)$ and $n\in\NN$. By the already established Bohr--Sommerfeld type approximate identity \eqref{approx_bohr_sommerfeld} of Theorem \ref{thm:eigenvalues}, we know in particular that 
\begin{equation*}
\sqrt{E_n} \, |\{V \leq E_n\}|\gtrsim_\kappa 1.
\end{equation*}
Hence, by Proposition \ref{pot1_trivial_prp}, we deduce that
\begin{equation}\label{delta_eigenfunctions}
	\sqrt{E_n} \,  |\{V_\oplus \leq E_n\}|, \sqrt{E_n} \,|\{V_\ominus \leq E_n\}|\gtrsim_\kappa 1.
\end{equation}
Thus, by applying the exponential decay inequality \eqref{exponential_decay_eigen} of Theorem \ref{thm:pointwise_eigen} to $\psi_n(x)$ and $\psi_n(-x)$ (where $x>0$), we obtain the estimates
\begin{align*}
	\psi_n(x)^2 + \frac{\psi_n'(x)^2}{E_n} &\lesssim_\kappa \frac{1}{|\{V_\oplus \leq E_n\}|} e^{-c(\kappa) |x| \sqrt{V(x)}} \qquad \forall x>0\tc V(x)\geq 4E_n,\\
	\psi_n(x)^2 + \frac{\psi_n'(x)^2}{E_n} &\lesssim_\kappa \frac{1}{|\{V_\ominus \leq E_n\}|} e^{-c(\kappa) |x| \sqrt{V(x)}} \qquad \forall x<0\tc V(x)\geq 4E_n.
\end{align*}
Another application of Proposition \ref{pot1_trivial_prp} finally gives the desired pointwise bound 	
\begin{equation*}
	\psi_n(x)^2 + \frac{\psi_n'(x)^2}{E_n} \lesssim_{\kappa}\frac{1}{|\{V\leq E_n\}|} e^{-c(\kappa)|x|\sqrt{V(x)}}
\end{equation*} for every $x$ in the range $V(x)\geq 4E_n$. 

Parts \ref{en:integral_eigenfcts} and \ref{en:ass_weight_eigenfcts} of Theorem \ref{thm:pointwise_eigenfcts} follow similarly from parts \ref{en:Vpowers_upp} and \ref{en:Vpowers_ass} of Theorem \ref{thm:Vpowers_prp}, by using \eqref{delta_eigenfunctions} and Proposition \ref{pot1_trivial_prp}.

\subsection{Eigenvalue gap estimates}\label{ss:gaps}

We now prove the eigenvalue gap estimate of Theorem \ref{thm:eigenvalues}, namely the approximate identity
\begin{equation*}
	E_n - E_m \simeq_\kappa \frac{E_n}{n} (n-m)\qquad n \geq m \geq 1,
\end{equation*}
under the assumption $V\in \pot_1(\kappa)$. Let us first show that this follows from
\begin{equation}\label{eigenvalue_gaps}
	E_{n+1}-E_n \simeq_{\kappa} \frac{E_n}{n}\qquad\forall n\geq 1.
\end{equation} 
In fact, by Proposition \ref{prp:eigenvalue_doub}, there exists $T_0=T_0(\kappa) \geq 2$ so large that $n\geq T_0 m$ implies $E_n\geq 2E_m$. Hence 
\[
E_n-E_m \simeq E_n \simeq \frac{E_n}{n} (n-m)\qquad\forall n \geq T_0 m.
\]
If instead $m \leq n \leq T_0 m$, then for $m \leq \ell < n$ we have $\ell \simeq_\kappa n$ and, by Proposition \ref{prp:eigenvalue_doub}, $E_\ell \simeq_{\kappa} E_n$. Thus
\[
E_n-E_m = \sum_{\ell=m}^{n-1} (E_{\ell+1}-E_\ell) \simeq_\kappa \sum_{\ell=m}^{n-1} \frac{E_\ell}{\ell} \simeq_\kappa \frac{E_n}{n} (n-m),
\]
as desired. 

\smallskip

We are now reduced to proving \eqref{eigenvalue_gaps}. The proof that we present follows the lines of arguments in \cite{kirsch_simon_1985}.

Assume, to begin with, that $V \in \Tpot$. Given $E>0$, let $v(x;E)$ be the solution defined in Section \ref{sec:sturm}. If $\lambda>0$ is a parameter, we introduce the Pr\"ufer variable $\theta(x;E,\lambda)$, defined as ``the'' argument of the complex number $\lambda v(x;E)+iv'(x;E)$. Notice that $v(x;E)$ and $v'(x;E)$ cannot vanish simultaneously, so $\theta(x;E,\lambda)$ is defined modulo a constant integer multiple of $2\pi$. We normalise it requiring that $\theta(x;E,\lambda)\in (-\pi/2,0)$ for $x$ in a neighbourhood of $+\infty$. This is possible, since $v'(x;E)<0<v(x;E)$ in a neighbourhood of $+\infty$. The key feature of $\theta(x;E,\lambda)$ is that $v(x;E)$ vanishes if and only if $\theta(x;E,\lambda)\in \pi/2+\pi\ZZ$.  

The relevance of the Pr\"ufer variable $\theta = \theta(x;E,\lambda)$ stems from the identities contained in the following proposition. 

\begin{prp}\label{der_prufer_prp}
Let $V \in \Tpot$.
Then the following identities hold.
\begin{equation}\label{eq:theta_der_iden}
\partial_x \theta = \lambda \frac{(V-E)v^2-(\partial_xv)^2}{\lambda^2v^2+(\partial_xv)^2}, \qquad
\partial_E \theta = \lambda \frac{\int_x^{+\infty} v(y;E)^2 \,dy}{\lambda^2v^2+(\partial_xv)^2}.
\end{equation}
\end{prp}
\begin{proof}
If $t \mapsto z(t)$ is a $C^1$ complex-valued function without zeros, it is easily checked that $\arg(z)'(t)=\Im\left(\frac{z'(t)}{z(t)}\right)$. We now apply this to $z(x;E)=\lambda v(x;E) + i \partial_x v(x;E)$ and use the differential equation \eqref{eq:ODE_RR} satisfied by $v$ to compute the derivatives $\partial_x z$ and $\partial_E z$. The former identity in \eqref{eq:theta_der_iden} is now straightforward: 
\[
\partial_x \theta = \Im\left(\frac{\partial_x z}{z} \right) 
= \Im\left(\frac{\lambda \partial_x v + i (V-E) v}{\lambda v + i \partial_x v} \right)
= \lambda \frac{(V-E) v^2 - (\partial_x v)^2}{\lambda^2 v^2 + (\partial_x v)^2} ,
\]
as desired. As for the latter,
\[
\partial_E \theta = \Im\left(\frac{\partial_Ez}{z} \right)
= \Im\left(\frac{\lambda \partial_E v + i \partial_E \partial_x v}{\lambda v + i\partial_x v} \right) 
= \lambda \frac{w}{\lambda^2 v^2+(\partial_x v)^2},
\]
where $w \defeq v \partial_E \partial_x v - \partial_x v \partial_E v$. Now, the differential equation \eqref{eq:ODE_RR} implies that $\partial_x w =-v^2$. If we show that $\lim_{x \to +\infty} w(x)=0$, then we can conclude that $w(x)=\int_x^{+\infty} v^2$ and derive the desired identity.

To verify the vanishing of the limit of $w$, notice that by Theorem \ref{thm:smoothdependence}, for all closed upper half-lines $I \subseteq \RR$ and all $\alpha \geq 0$, the map $E \mapsto v(\cdot;E)|_I$ is real-analytic from $\RR$ to the Banach space $\banD_{V|_I,\alpha}$. In particular, if $V \in \Tpot$, then, by Proposition \ref{u'_prp},
\begin{equation}\label{eq:v_decay}
	\lim_{x \to +\infty} e^{\beta x} \partial_E^k v(x;E) = \lim_{x \to +\infty} e^{\beta x} \partial_E^k \partial_x v(x;E) = 0
\end{equation}
for all $k \in \NN$ and $\beta \geq 0$, and so $\lim_{x \to +\infty} w(x)=0$. 
\end{proof}

We can apply exactly the same construction to the reflected potential $\widetilde{V}(x) \defeq V(-x)$. Let us denote by $\widetilde{v}(x;E)$ the solutions of Section \ref{sec:sturm} associated to the reflected potential, and by $\widetilde{\theta}(x;E,\lambda)$ the corresponding Pr\"ufer variable. The next proposition describes the relationship between $\theta(x;E,\lambda)$ and $\widetilde{\theta}(x;E,\lambda)$, and between $v(x;E)$ and $\widetilde{v}(x;E)$. 

\begin{prp}\label{prp:gap_tilde}
Let $V \in \Tpot$ and $n \in \Npos$. Then, for all $x \in \RR$ and $\lambda > 0$,
\begin{equation}\label{eq:sum_theta}
\theta(x;E_n,\lambda)+\widetilde{\theta}(-x;E_n,\lambda) = (n-1)\pi,
\end{equation}
where $E_n = E_n(V) = E_n(\widetilde{V})$. Moreover, there exists $E_n' \in [E_n,E_{n+1}]$ such that
\begin{equation}\label{eq:gap_theta}
\frac{\pi}{(E_{n+1}-E_n)\sqrt{E_n}} = \frac{\int_0^{+\infty} v(y;E'_n)^2 \,dy}{E_nv^2(0;E'_n) + \partial_x v(0;E'_n)^2}+\frac{\int_0^{+\infty} \widetilde{v}(y;E'_n)^2 \,dy}{E_n\widetilde{v}^2(0;E'_n)+\partial_x \widetilde{v}(0;E'_n)^2}. 
\end{equation}
\end{prp}
\begin{proof}
The first identity of Proposition \ref{der_prufer_prp} implies in particular that whenever $v(x; E)$ vanishes, $\partial_x \theta(x;E,\lambda) = -\lambda < 0$. Recalling our normalization of $\theta$, we see that if $x_0(E)>x_1(E)>\ldots$ are the zeros of $v(\cdot ,E)$ arranged in decreasing order, then $\theta(x_k(E);E,\lambda)=\pi/2+k\pi$.

By Sturm--Liouville theory, any nonzero eigenfunction of $\opH[V]$ of eigenvalue $E_n=E_n(V)$ has exactly $n-1$ zeros. Therefore, $\theta(x;E_n,\lambda)=\pi/2$ at the rightmost zero of $v(\cdot; E_n)$, and $\theta(x;E_n,\lambda)=\pi/2+(n-2)\pi$ at the leftmost zero of $v(\cdot; E_n)$. 

Clearly similar considerations hold for the solutions $\widetilde{v}(x;E)$ associated with the reflected potential $\widetilde{V}$.
Note now that $v$ solves \eqref{eq:ODE_RR} if and only if $\widetilde{v}(x) \defeq v(-x)$ solves the reflected equation $-u''+\widetilde{V}u=Eu$. Therefore, $\opH[V]$ and $\opH[\widetilde{V}]$ have the same eigenvalues and reflected eigenfunctions. Hence, when $E$ is an eigenvalue, $\widetilde{v}(x;E)=\mu(E) v(-x,E)$ for some $\mu(E)\in \RR\setminus\{0\}$. This means that 
\begin{equation}\label{reflecting_prufer}
\theta(x;E_n,\lambda)+\widetilde{\theta}(-x;E_n,\lambda)\in \pi\ZZ\qquad \forall x\in \RR.
\end{equation}
By continuity, the multiple of $\pi$ is independent of $x$. To compute it, we  evaluate the left-hand side of \eqref{reflecting_prufer} at the rightmost zero of $v(\cdot; E_n)$, so that $-x$ equals to the leftmost zero of $\widetilde{v}(x;E)$. This gives \eqref{eq:sum_theta}.

In particular, if we choose $\lambda = \sqrt{E_n}$, then \eqref{eq:sum_theta} and Lagrange's Mean Value Theorem imply that
\[\begin{split}
\pi 
&= \theta(0;E_{n+1},\sqrt{E_n}) + \widetilde{\theta}(0;E_{n+1},\sqrt{E_n}) - \theta(0;E_n,\sqrt{E_n}) - \widetilde{\theta}(0;E_n,\sqrt{E_n}) \\
&= (E_{n+1} - E_n) \partial_E(\theta + \widetilde{\theta})(0;E'_n,\sqrt{E_n})
\end{split}\]
for some $E'_n\in [E_n,E_{n+1}]$. The formula for $\partial_E\theta$ in Proposition \ref{der_prufer_prp} finally yields \eqref{eq:gap_theta}.
\end{proof}

Now that all the ingredients are in place, we can prove \eqref{eigenvalue_gaps}, under the assumption $V \in \pot_1(\kappa)$. Let $n \in \Npos$, and let $E_n'$ be as in Proposition \ref{prp:gap_tilde}.
By Proposition \ref{prp:eigenvalue_doub}, $E_n'\simeq_\kappa E_n$. Thus, the identity of Proposition \ref{prp:elementaryids}\ref{en:v(0)_prp}, combined with \eqref{delta_eigenfunctions} and Theorem \ref{thm:Vpowers_prp}, gives that
\[\begin{split}
E_n v^2(0;E'_n) + \partial_x v(0;E'_n)^2 
&\simeq_\kappa E_n' v^2(0;E'_n) + \partial_x v(0;E'_n)^2 \\
&= \int_0^{+\infty} V'(y) \, v(y;E'_n)^2 \,dy \\
&\simeq_\kappa \frac{E_n'}{|\{ V_\oplus \leq E_n'\}|} \int_0^{+\infty} v(y;E'_n)^2 \,dy.
\end{split}\]
Notice that the condition $V\in \pot_1(\kappa)$ has been used in the last step. 
At this point, Proposition \ref{pot1_trivial_prp} and \eqref{approx_bohr_sommerfeld} yield
\begin{equation*}
\frac{\int_0^{+\infty} v(y;E'_n)^2 \,dy}{E_n v^2(0;E'_n) + \partial_x v(0;E'_n)^2} \simeq_\kappa \frac{|\{V\leq E_n\}|}{E_n} \simeq_\kappa \frac{n}{E_n^{3/2}} .
\end{equation*}
Of course, a completely analogous argument gives the same estimate for the term involving $\widetilde{v}$, and the desired bound \eqref{eigenvalue_gaps} follows from \eqref{eq:gap_theta}.

The proof of Theorem \ref{thm:eigenvalues} is complete.

\section{Matrix and spectral projector bounds}\label{s:matrixbounds}

\subsection{Summary of the results}

Let $V$ be in the class $\Tpot$ (see Definition \ref{dfn:potentials_singlewell}). Define the ``differentiated eigenfunctions''
\begin{equation*}
\sigma_n(x;V) \defeq \partial_t|_{t=1} \psi_n(x;t V)
\end{equation*} 
for $n \geq 1$
(this makes sense by Proposition \ref{prp:one_parameter_tV}).
We define the real-valued matrices $\matP(V)$ and $\matA(V)$ associated to the potential $V$ as
\begin{equation*}
\matP_{nm}(V) = \langle V\psi_n(\cdot;V),\psi_m(\cdot;V) \rangle,
\end{equation*}
and 
\begin{equation*}
\matA_{nm}(V)=\langle \sigma_n(\cdot;V),\psi_m(\cdot;V)\rangle,
\end{equation*}
where $n,m \geq 1$ and $\langle\cdot,\cdot\rangle$ stands for the $L^2(\RR)$-inner product. These two matrices are crucially related as follows.

\begin{prp}\label{prp:A_formula}
Let $V \in \Tpot$. The matrix $\matP(V)$ is symmetric and the matrix $\matA(V)$ is antisymmetric. For all $n \neq m$,
\begin{equation}\label{AV_identity}
\matP_{nm}(V) = \matA_{nm}(V) \, (E_n(V)-E_m(V)) .
\end{equation}
\end{prp}

The relation \eqref{AV_identity} is only meaningful away from the diagonal. As shown in Theorem \ref{thm:virial}, the on-diagonal terms of $\matP(V)$ are given by the ``differentiated eigenvalues'' $F_n(V) = \partial_t|_{t=1} E_n(tV)$.

An important part of this section is devoted to proving estimates for the matrix coefficients of $\matP(V)$ and $\matA(V)$ under the assumption $V \in \pot_1(\kappa)$. In light of the relation \eqref{AV_identity} and the gap estimates in Theorem \ref{thm:eigenvalues}, off-diagonal estimates for $\matP(V)$ imply corresponding estimates for $\matA(V)$ and vice versa, so we only need to state and prove our estimates for one of the two matrices.

\smallskip 

It is convenient to introduce some notation. For $T > 1$, let $\matN_T$ be the matrix whose $(n,m)$-entry is $1$ if $m \in [T^{-1} n, T n]$ and $0$ otherwise, and $\matF_T$ the matrix whose $(n,m)$-entry $0$ if $m \in [T^{-1} n, T n]$ and $1$ otherwise (here $\matN$ and $\matF$ stand for ``near'' and ``far'' from the diagonal). 
Moreover, let $\schur$ denote the \emph{Schur product} between matrices; namely, if $\matB$ and $\matB'$ are matrices, then $\matB \schur \matB'$ is given by
\begin{equation}\label{eq:schur_def}
(\matB \schur \matB')_{nm}=\matB_{nm} \matB'_{nm}.
\end{equation}
Note that, for all $T > 1$, every matrix $\matB$ decomposes as
\[
\matB = \matB \schur \matN_T + \matB \schur \matF_T;
\]
we can think of $\matB \schur \matN_T$ and $\matB \schur \matF_T$ as the near-diagonal and far-diagonal parts of $\matB$ respectively.
We will also write $|\matB|$ to denote the matrix whose components are the moduli $|\matB_{nm}|$ of the components of the matrix $\matB$.

\smallskip 

The following statement includes a useful ``a priori'' bound for the coefficients of $\matA(V)$, as well as a sharper bound in the far-diagonal region.

\begin{thm}[Matrix bounds: $\matA$]\label{thm:estimates_A}
Let $V \in \pot_1(\kappa)$. Then, for every $n\geq 1$,
\begin{equation}\label{eq:A_apriori}
\int_{\RR} \sigma_n(x;V)^2 \,dx = \sum_{m} \matA_{nm}(V)^2 \lesssim_{\kappa} n^2 .
\end{equation}
Moreover, for all $T > 1$,
\begin{equation}\label{far_diag_bound}
|\matA_{nm}(V)| \lesssim_{\kappa,T} \frac{1}{\sqrt{nm}} \left(\frac{E_m(V)}{E_n(V)}\right)^{3/4}, \qquad \forall n, m \tc n \geq T m,
\end{equation} 
Consequently,
\[
\||\matA(V)| \schur \matF_T\|_{\ell^2 \to \ell^2} \lesssim_{\kappa,T} 1
\]
for all $T > 1$.
\end{thm}

We also obtain bounds in the ``near-diagonal'' region, which for convenience will be stated in terms of the matrix $\matP(V)$. In order to prove bounds with the required summability properties, here we make an additional regularity assumption on the potential $V$, corresponding to the assumption \eqref{eq:2dassumptions_holder} of the introduction.

\begin{dfn}
For $\theta \in (0,1)$ and $\kappa \geq 1$, we define $\pot_{1+\theta}(\kappa)$ as the class of the $V \in \pot_1(\kappa)$ such that
\begin{equation}\label{eq:potholderclass}
\left|V'(e^h x) - V'(x) \right| \leq \kappa \, |V'(x)| \, |h|^\theta \qquad \forall x \in \RR \setminus \{0\}  \quad  \forall h \in [-1,1].
\end{equation}
\end{dfn}

We can now state the near-diagonal bounds for $\matP(V)$.

\begin{thm}[Matrix bounds: $\matP$]\label{thm:near_diag}
Let $V\in \pot_1(\kappa)$. Then, for all $T \geq 1$,
\begin{equation}\label{eq:near_diag_bound}
|\matP_{nm}(V)| \lesssim_{\kappa,T} \frac{E_n(V)}{1+|m-n|} \qquad\forall n,m \tc T^{-1} n\leq m\leq Tn.
\end{equation}
Assume moreover that $V \in \pot_{1+\theta}(\kappa)$ for some $\theta \in (0,1)$. 
Then there exists $\epsilon = \epsilon(\theta,\kappa) >0$ such that, for all $T \geq 1$,
\begin{equation}\label{eq:near_diag_bound_sharper}
|\matP_{nm}(V)| \lesssim_{\theta,\kappa,T} \frac{E_n(V)}{1+|m-n|^{1+\epsilon}} \qquad\forall n,m \tc T^{-1} n\leq m\leq Tn.
\end{equation}
\end{thm}

Finally, we state some pointwise bounds for ``clusters'' of eigenfunctions and differentiated eigenfunctions of bounded energy. We note that the supremum in $x \in \RR$ of the left-hand side of \eqref{eq:spprojbd} below is the squared $L^2 \to L^\infty$ norm of the spectral projector of the Schr\"odinger operator $\opH[V]$ associated to the interval $[0,E_0]$; for this reason, by a slight abuse of language, we will refer to the estimates below as ``spectral projector bounds''.

\begin{thm}[Spectral projector bounds]\label{thm:spectral_proj}
Let $V\in\pot_1(\kappa)$. Then, for all $E_0>0$,
\begin{equation}\label{eq:spprojbd}
\sum_{E_n(V) \in [0,E_0]} \psi_n(x;V)^2 \lesssim_{\kappa} \sqrt{E_0} \left( \chr_{V \leq 4 E_0} +  e^{-c(\kappa) |x| \sqrt{V(x)}} \chr_{V \geq 4 E_0} \right).
\end{equation}
Let $T_0>1$, and define, for all $n \in \Npos$ and $x \in \RR$, the ``modified differentiated eigenfunctions''
\begin{equation}\label{eq:rho_def_elem}
\begin{split}
\rho_n(x;V) &\defeq \sigma_n(x;V) - \sum_m (\matA(V) \schur \matN_{T_0})_{nm}\psi_m(x;V) \\
&= \sum_m (\matA(V) \schur \matF_{T_0})_{nm}\psi_m(x;V) .
\end{split}
\end{equation} 
Then there exists $T_1=T_1(\kappa,T_0)$ such that, for all $E_0>0$,
\begin{equation*}
\sum_{E_n(V)\in [0,E_0]} \rho_n(x;V)^2 \lesssim_{\kappa,T_0} \sqrt{E_0} \left(\chr_{V\leq T_1 E_0} +  e^{-c(\kappa) |x| \sqrt{V(x)}} \chr_{V\geq T_1 E_0} \right).
\end{equation*}	
\end{thm}

The rest of the section is devoted to the proofs of the above results.

\subsection{Elementary properties of the matrices \texorpdfstring{$\matA$}{A} and \texorpdfstring{$\matP$}{P}}
Here we prove Proposition \ref{prp:A_formula}. Due to the assumption $V \in \Tpot$, all the differentiations used in the  proof are allowed in light of Proposition \ref{prp:one_parameter_tV}.

First, the symmetry of $\matP(V)$ is obvious, and the antisymmetry of $\matA(V)$ follows by differentiating in $t=1$ the identity $\langle \psi_n(\cdot;tV) , \psi_m(\cdot;tV) \rangle = \delta_{nm}$.

Differentiating in $t=1$ the equation $\opH[t V] \psi_n(x;t V) = E_n(t V) \, \psi_n(x;t V)$ gives the inhomogeneous equation
\begin{equation}\label{sigma_eq}
-\sigma_n''(x;V) + V(x) \, \sigma_n(x;V) = E_n(V) \, \sigma_n(x;V) + (F_n(V)-V(x)) \, \psi_n(x;V),
\end{equation}
where $F_n(V)$ is the differentiated eigenvalue, as in \eqref{virial_F}.
Multiplying both sides of \eqref{sigma_eq} by $\psi_m(x;V)$ with $m \neq n$ and integrating gives 
\[
E_n \matA(V)_{nm} - \matP(V)_{nm} = \langle \opH[V] \sigma_n, \psi_m \rangle = \langle \sigma_n, \opH[V] \psi_m \rangle = E_m \matA(V)_{nm} ,
\]
and rearranging yields the desired identity \eqref{AV_identity}.

\subsection{Differentiated eigenfunctions: \texorpdfstring{$L^2$}{L2} bound}\label{ss:sigma_l2bd}
Here we work under the assumption $V \in \pot_1(\kappa)$ and obtain the
$L^2$ bound \eqref{eq:A_apriori} for the differentiated eigenfunctions $\sigma_n = \sigma_n(\cdot;V)$, which proves the first part of Theorem \ref{thm:estimates_A}.

By Proposition \ref{prp:A_formula} and Theorem \ref{thm:eigenvalues}, for all $n \geq 1$,
\begin{multline*}
\int_{\RR} \sigma_n^2 
= \sum_{m\neq n} \matA_{nm}(V)^2 \\
= \sum_{m\neq n}\frac{\matP_{nm}(V)^2}{|E_m-E_n|^2} 
\lesssim_{\kappa} \frac{n^2}{E_n^2} \sum_{m\neq n} \matP_{nm}(V)^2 
\leq \frac{n^2}{E_n^2} \int_{\RR} V^2 \psi_n^2 
\lesssim_{\kappa} n^2.
\end{multline*}
In the last step we used Theorem \ref{thm:pointwise_eigenfcts}.

\subsection{Differentiated eigenfunctions: exponential decay}
Thanks to the above $L^2$ bound, we can now apply Agmon's theory (Corollary \ref{pointwise_agmon_cor}) to show that the differentiated eigenfunctions $\sigma_n(\cdot;V)$ have a similar exponential decay far from the classical region as the one proved for the eigenfunctions $\psi_n(\cdot;V)$ in Theorem \ref{thm:pointwise_eigenfcts}.

\begin{prp}\label{sigma_pointwise_prp}
Let $V \in \pot_1(\kappa)$. Then, for every $n \in \Npos$ and $x \in \RR$ such that $V(x) \geq 4 E_n(V)$,
\begin{equation*}
	\sigma_n(x;V)^2 \lesssim_\kappa \frac{1}{|\{V\leq E_n(V)\}|} e^{-c(\kappa)|x|\sqrt{V(x)}}.
\end{equation*}
\end{prp}
\begin{proof}
We prove the statement for $x>0$. The statement for $x<0$ is proved identically, by considering $V(-x)$ in place of $V(x)$.

Write $E_n = E_n(V)$, $F_n = F_n(V)$, $\psi_n(x) = \psi_n(x;V)$ and $\sigma_n(x) = \sigma_n(x;V)$ for simplicity.
We know that $\sigma_n$ solves the differential equation \eqref{sigma_eq}.
Recall the definition of $V_\oplus$ from \eqref{eq:pm_parts}.
If we apply Corollary \ref{pointwise_agmon_cor} with potential $V_\oplus-E_n$ and $w=(F_n-V_\oplus)\psi_n$, with parameters
\[
A=E_n/2, \ B=E_n, \ C=3E_n, \ \beta=1/2, \  \varepsilon=1/12, \ \gamma=3/4,
\]
then we obtain, for all $x>0$ such that $V(x) \geq 4E_n$, that
\begin{equation}\label{eq:prel_est_sigma}
\sigma_n(x)^2 \lesssim \auxC_0^2 \auxC_1 \auxD \exp\left(-\int_{(V_\oplus)^{\inv}(2E_n)}^x \sqrt{V-E_n}\right),
\end{equation}
where 
\[
\auxC_j = \int_{V_\oplus > 4E_n} \exp\left(-\frac{1}{6} \int_{(V_\oplus)^\inv(2E_n)}^y \sqrt{V-E_n} \right) (V(y)-E_n)^j \,dy 
\]
and
\[\begin{split}
\auxD
&= \int_{V_\oplus>2E_n} \exp\left(\frac{3}{2} \int_{(V_\oplus)^{\inv} (2E_n)}^y \sqrt{V-E_n}\right) \frac{|(F_n-V(y)) \psi_n(y)|^2}{V(y)-E_n} \,dy \\
&\qquad+ \int_{3E_n/2<V_\oplus<2E_n} \frac{|(F_n-V(y)) \psi_n(y)|^2}{V(y)-E_n} \,dy \\
&\qquad+ \frac{1}{|\{3E_n/2<V_\oplus<2E_n\}|^2} \int_{3E_n/2<V_\oplus<2E_n} \sigma_n(y)^2 \,dy .
\end{split}\]

Since $\sqrt{E_n} \,|\{V_\oplus \leq E_n\}| \gtrsim_\kappa 1$ 
by \eqref{delta_eigenfunctions},
we obtain as in Section \ref{ss:half_exponential} that
\begin{equation*} 
\auxC_0 \lesssim_\kappa |\{V_\oplus \leq E_n\}| \quad\text{and}\quad \auxC_1 \lesssim_\kappa |\{V_\oplus \leq E_n\}| \, E_n.
\end{equation*}
Moreover, 
\begin{equation}\label{virial}
0 \leq F_n \leq E_n
\end{equation}
by Theorem \ref{thm:virial}, hence
\[\begin{split}
\auxD
&\leq 4\int_{V_\oplus>2E_n} \exp\left(\frac{3}{2} \int_{(V_\oplus)^{-1}(2E_n)}^y \sqrt{V-E_n}\right)(V(y)-E_n)| \psi_n(y)|^2 \,dy \\
&\qquad+ 3\int_{3E_n/2<V_\oplus<2E_n} V(y) \psi_n(y)^2 \,dy \\
&\qquad+\frac{1}{|\{3E_n/2<V_\oplus<2E_n\}|^2} \int_{3E_n/2<V_\oplus<2E_n} \sigma_n(y)^2 \,dy .
\end{split}\]
We can bound the first summand by Theorem \ref{agmon_thm} (applied with potential $V_\oplus - E_n$, $w=0$, $\gamma=3/4$, $A=0$, and $B = E_n$)  
 and the last summand by \eqref{eq:A_apriori}, thus obtaining
\[
\auxD
\lesssim_\kappa \frac{1}{|\{E_n<V_\oplus<2E_n\}|^2} + E_n + \frac{n^2}{|\{3E_n/2<V_\oplus<2E_n\}|^2} 
\lesssim_\kappa E_n;
\]
here we used that
\[
|\{E_n<V_\oplus<2E_n\}|^2 \simeq_\kappa |\{3E_n/2 <V_\oplus<2E_n\}|^2 \simeq_\kappa |\{V \leq E_n\}|^2 \simeq n^2 / E_n
\]
by
Propositions \ref{doubling_prp}
and \ref{pot1_trivial_prp}
and Theorem \ref{thm:eigenvalues}.
Hence \eqref{eq:prel_est_sigma} gives, for all $x>0$ such that $V(x) \geq 4 E_n$, that
\[\begin{split}
|\sigma_n(x)|^2 
&\lesssim_\kappa |\{V_\oplus \leq E_n\}|^3 E_n^2 \exp\left(-\int_{(V_\oplus)^{\inv}(2E_n)}^x \sqrt{V-E_n}\right) \\
&\leq \frac{1}{|\{V_\oplus \leq E_n\}|} x^4 V(x)^2 e^{-c(\kappa) x \sqrt{V(x)}} \\
&\lesssim_\kappa \frac{1}{|\{V\leq E_n\}|} e^{-\frac{c(\kappa)}{2} x \sqrt{V(x)}} ,
\end{split}\]
where Proposition \ref{pot1_trivial_prp} was again used in the last step.
\end{proof}

\subsection{Off-diagonal decay: the commutator argument}
\label{ss:matrixbounds}

This and the following two sections are aimed at proving the bounds for the matrix coefficients of $\matA(V)$ and $\matP(V)$ stated in Theorems \ref{thm:estimates_A} and \ref{thm:near_diag}. We point out that the bound \eqref{eq:A_apriori} for $\matA(V)$ has already been proved in Section \ref{ss:sigma_l2bd}, and that moreover the on-diagonal bounds ($n=m$) for $\matP(V)$ in Theorem \ref{thm:near_diag} are trivial. What effectively remains to prove is the off-diagonal decay that those bounds entail.

Fix $V$ and write $\psi_n(x)=\psi(x; V)$, $E_n = E_n(V)$, and $\opH=\opH[V]$ for simplicity. The main heuristics at the basis of this section is contained in the following formal chain of identities:
\begin{equation}\label{eq:commutator_idea}
\begin{split}
\langle [\opH,V] \psi_n, \psi_m \rangle 
&= \langle \opH V\psi_n, \psi_m \rangle - \langle V \opH \psi_n, \psi_m \rangle \\
&= \langle V \psi_n , \opH \psi_m \rangle - \langle V \opH \psi_n, \psi_m \rangle \\
&= ( E_m - E_n ) \, \langle V \psi_n, \psi_m \rangle.
\end{split}
\end{equation}
This looks promising, because it allows to convert any (yet to be proved) upper bound on $|\langle [\opH,V]\psi_n,\psi_m\rangle|$ into an off-diagonal decay for $\matP_{nm}(V)$ (and $\matA_{nm}(V)$, in virtue of \eqref{AV_identity}). Moreover, by iterating the above argument, one could potentially obtain an even faster off-diagonal decay by considering iterated commutators.

To rigorously justify the above identities \eqref{eq:commutator_idea}, one would need to ensure that eigenfunctions are in the natural domain of the operator $[\opH,V]$, namely that $V\psi_n$ is in the domain of $\opH$. This would require in particular some control on the second derivative $V''$. Correspondingly, higher-order derivatives of $V$ would need to be controlled in order to deal with iterated commutators.

Nevertheless, the theory developed below allows us to obtain off-diagonal decay for an arbitrary potential $V$ in $\pot_1(\kappa)$ or $\pot_{1+\theta}(\kappa)$, for which bounds on the second and higher-order derivatives may not be available.

As we shall see, the estimates for $\matP(V)$ will be derived from estimates for the matrix coefficients $\langle U \psi_n, \psi_m \rangle$ associated to a more general function $U : \RR \to \CC$. We begin by introducing some terminology that will be convenient when justifying integrations by parts.

\begin{dfn}
Let $U : \RR \to \CC$. We say that $U$ has \emph{moderate growth} if there exist $M,K >0$ such that $|U(x)| \leq e^{M|x|}$ whenever $|x| \geq K$. Moreover, we say that $U$ has \emph{fast decay} if for all $M > 0$ there exists $K>0$ such that $|U(x)| \leq e^{-M|x|}$ whenever $|x| \geq K$.
\end{dfn}

The following statement, applied with $U=V$, provides a replacement of \eqref{eq:commutator_idea} that does not require $V$ to be twice differentiable.

\begin{lem}\label{commutator_identities_lem}
Assume that $V \in \Tpot$. Let $U \in C^0(\RR) \cap C^1(\RR \setminus \{0\})$ be such that $U'$ is locally integrable at $0$, and moreover $U,U'$ have moderate growth. Then
\begin{equation}\label{first_commutator}
(E_m-E_n) \, \langle U \psi_n, \psi_m \rangle = \int_{\RR} U' (\psi_n\psi_m'-\psi_n'\psi_m) ,
\end{equation}
where the integrals are absolutely convergent.
\end{lem}
\begin{proof}
Note that $\psi_n$, $\psi_n'$ and $\psi_n'' = (V-E)\psi_n$ are continuous and have fast decay (see Proposition \ref{prp:one_parameter_tV}).
So
\[\begin{split}
(E_m-E_n) \, \langle U \psi_n, \psi_m \rangle
&= \langle U \psi_n, \opH \psi_m \rangle - \langle U \opH \psi_n, \psi_m \rangle \\
&= \int_\RR U \psi_n (-\psi_m'' + V \psi_m)-\int_\RR U(-\psi_n'' + V \psi_n) \psi_m \\
&= -\int_\RR U \psi_n \psi_m''+ \int_\RR U \psi_n''\psi_m,
\end{split}\]
where the integrals are absolutely convergent because of the fast decay.
Since $U' \in L^{1}_\loc$, integration by parts (cf.\ \cite[eq.\ (7.18)]{gilbarg_trudinger} and \cite[Proposition 3.6]{cowling_martini_2013}) gives that
\begin{equation*}
\int_{\RR} U \psi_n'' \psi_m = -\int_{\RR} U' \psi_n' \psi_m - \int_{\RR} U \psi_n' \psi_m'.
\end{equation*} 
Clearly a similar identity holds with $n$ and $m$ swapped.
As a consequence, 
\begin{equation*}
-\int_{\RR} U \psi_n \psi_m'' + \int_{\RR} U \psi_n'' \psi_m = \int_{\RR} U' \psi_n \psi_m' - \int_{\RR} U' \psi_n' \psi_m.
\end{equation*}
The proof of \eqref{first_commutator} is complete.
\end{proof}

As we shall see, the off-diagonal decay provided by \eqref{first_commutator} will not be enough for our purposes. To obtain a faster decay, we can iterate the above argument, under additional  assumptions on $V$ and $U$.

\begin{lem}\label{commutator_identities2_lem}
Assume that $V \in \Tpotone$ and $V'$ is locally integrable at $0$.
Let $U \in C^1(\RR) \cap C^2(\RR \setminus \{0\})$ be such that $U''$ is locally integrable at $0$ and $U,U',U''$ have moderate growth.
Then
\begin{multline}\label{second_commutator}
(E_m-E_n)^2 \, \langle U \psi_n,\psi_m \rangle \\
= \int_{\RR} \left(2 V' U' + 2 V U''-(E_n+E_m) U'' \right) \psi_n \psi_m - 2 \int_{\RR} U'' \psi_n' \psi_m',
\end{multline}
where the integrals are absolutely convergent.
\end{lem}
\begin{proof}
Note that, under our assumptions, $\psi_n,\psi_n',\psi_n'' = (V-E)\psi_n$, are all continuous on $\RR$ and have fast decay (see Proposition \ref{prp:one_parameter_tV}), while $\psi_n''' = V'\psi_n + (V-E)\psi_n'$ is continuous on $\RR \setminus \{0\}$, locally integrable at $0$, and has fast decay.

Invoking \eqref{first_commutator} yields
\[\begin{split}
&\left(E_m-E_n\right)^2 \left \langle U \psi_n, \psi_m\right\rangle \\
&= (E_m-E_n) \left(\int_{\RR} U' \psi_n \psi_m' - \int_{\RR} U' \psi_n' \psi_m \right) \\
&= \int_{\RR} U'\psi_n(-\psi_m''+V\psi_m)'-\int_{\RR} U'\psi_n'(-\psi_m'' + V \psi_m) \\
&+ \int_{\RR} U' (-\psi_n'' + V\psi_n)' \psi_m - \int_{\RR} U' (-\psi_n'' + V\psi_n) \psi_m' \\
&= 2\int_{\RR} V' U' \psi_n \psi_m + \int_{\RR} U' (\psi_n' \psi_m'' + \psi_n'' \psi_m'- \psi_n \psi_m''' - \psi_n''' \psi_m).
\end{split}\]
All the above integrals are absolutely convergent because of the aforementioned continuity, integrability and decay properties of the eigenfunctions and their derivatives, the local integrability of $V'$, the continuity of $U'$ and the moderate growth of $V'$ and $U'$.

As in the proof of Lemma \ref{commutator_identities_lem}, since $U'',\psi_m'''\in L^1_\loc$, integration by parts gives that
\begin{equation*}
-\int_{\RR} U'\psi_n\psi_m''' = \int_{\RR} U' \psi_n' \psi_m'' + \int_{\RR} U''\psi_n\psi_m'',
\end{equation*} 
and similarly
\begin{equation*}
\int_{\RR} U' \psi_n' \psi_m'' = -\int_{\RR} U' \psi_n'' \psi_m' - \int_{\RR} U'' \psi_n' \psi_m'.
\end{equation*} 
Moreover analogous identities hold with $n$ and $m$ swapped.
All in all, 
\[\begin{split}
&(E_m-E_n)^2 \, \langle U \psi_n, \psi_m \rangle \\
&= 2\int_{\RR} V' U' \psi_n\psi_m + 2 \int_{\RR} U'(\psi_n' \psi_m'' + \psi_n''\psi_m') + \int_{\RR} U'' (\psi_n''\psi_m + \psi_n \psi_m'') \\
&= 2\int_{\RR} V' U' \psi_n\psi_m - 2 \int_{\RR} U'' \psi_n' \psi_m' + \int_{\RR} U'' (\psi_n'' \psi_m + \psi_n \psi_m'') \\
&= \int_{\RR} \left(2 V'U' + 2 V U''- (E_n+E_m) U'' \right) \psi_n \psi_m - 2\int_{\RR} U'' \psi_n' \psi_m', 
\end{split}\]
and we are done.
\end{proof}

We now combine the identity in Lemma \ref{commutator_identities_lem} and the weighted $L^2$ eigenfunction bounds in Theorem \ref{thm:pointwise_eigenfcts} to obtain, under the assumption $V \in \pot_1(\kappa)$, the following bound for the matrix elements associated to $U$.

\begin{prp}\label{prp:5/4_lem}
Let $V\in \pot_1(\kappa)$ and $U \in C^0(\RR) \cap C^1(\RR \setminus \{0\})$. For all $\alpha > 0$ and for all $n, m\geq 1$ with $n \neq m$,
\begin{equation*}
|\langle U \psi_n, \psi_m \rangle| \lesssim_{\kappa,\alpha} \left\| \frac{xU'}{ V^\alpha }\right\|_\infty  \frac{1}{|E_n-E_m|} \sqrt{\frac{(E_n+E_m) E_n^{1/2+\alpha} E_m^{1/2+\alpha}}{nm}}.
\end{equation*}
Thus, for every $T\geq 1$,
\begin{equation*}
|\langle U \psi_n, \psi_m \rangle| \lesssim_{\kappa,T,\alpha} \max_{\ell=0,1} \left\| \frac{x^\ell U^{(\ell)}}{ V^\alpha }\right\|_\infty   \frac{E_n^{\alpha}}{1+|m-n|} \quad \forall n,m \tc T^{-1} n \leq m \leq Tn.
\end{equation*}
\end{prp}
\begin{proof}
We may assume $M = \| xU' / V^\alpha \|_\infty < \infty$. 
Since $|U'(x)| \leq M |x|^{-1} V(x)^\alpha$ and $V \in \pot_1(\kappa)$ with $\alpha>0$, from Proposition \ref{doubling_prp} we deduce that $U'$ is locally integrable at $0$.
Hence, by Lemma \ref{commutator_identities_lem}, the Cauchy--Schwarz inequality and Theorem \ref{thm:pointwise_eigenfcts},
\[\begin{split}
&|(E_m-E_n) \, \langle U \psi_n,\psi_m \rangle|^2 \\
&\lesssim \left(\int_\RR |U'\psi_n\psi_m'|\right)^2 +\left(\int_\RR |U'\psi_n'\psi_m|\right)^2 \\
&\leq M^2 \left(\int_\RR |x|^{-1} V^{\alpha} \, \psi_n^2 \cdot \int_\RR |x|^{-1} V^{\alpha}\, (\psi_m')^2 
+\int_\RR |x|^{-1} V^{\alpha} \, (\psi_n')^2 \cdot \int_\RR |x|^{-1} V^{\alpha}  \,\psi_m^2 \right) \\
&\lesssim_{\kappa,\alpha}
M^2 \frac{E_n^{1+\alpha} E_m^{\alpha} + E_n^{\alpha} E_m^{1+\alpha}}{|\{V\leq E_n\}| \, |\{V\leq E_m\}|} .
\end{split}\]
Since $\sqrt{E_n} \,|\{V\leq E_n\}|\simeq_{\kappa} n$ by Theorem \ref{thm:eigenvalues}, we finally deduce that
\begin{equation*}
|(E_m-E_n) \, \langle U\psi_n,\psi_m \rangle| \lesssim_{\kappa,\alpha} M \sqrt{\frac{(E_n+E_m) E_n^{1/2+ \alpha} E_m^{1/2+\alpha}}{nm}}.
\end{equation*}
If $T^{-1} n\leq m\leq Tn$, then Theorem \ref{thm:eigenvalues} and Proposition \ref{prp:eigenvalue_doub} yield the estimate
\begin{equation*}
\frac{E_n}{n} | m-n | \, |\langle U\psi_n,\psi_m\rangle| \lesssim_{\kappa,T, \alpha} M \frac{E_n^{1+ \alpha}}{n},
\end{equation*}
which proves the desired near-diagonal bound, at least for $n\neq m$. On the other hand, if $n=m$, we have simply 
\[
|\langle U \psi_n,\psi_n \rangle| 
\leq \| U / V^\alpha \|_\infty \int_\RR V^{\alpha} \psi_n^2  \lesssim_{\kappa,\alpha}  \| U / V^\alpha \|_\infty \, E_n^{\alpha},
\]
by Theorem \ref{thm:pointwise_eigenfcts}.
\end{proof}

A faster decay for matrix coefficients is provided by the following lemma, which exploits the identity of Lemma \ref{commutator_identities2_lem} and consequently requires additional conditions on $U$ and $\alpha$.

\begin{prp}\label{prp:near_diag_lem_old}
Let $V\in \pot_1(\kappa)$. Let $U \in C^0(\RR) \cap C^2(\RR \setminus \{0\})$.
For all $\alpha > \kappa$, and all $n, m \geq 1$ with $n \neq m$,
\begin{equation*}
|\langle U \psi_n,\psi_m \rangle| \lesssim_{\kappa,\alpha} \max_{\ell=1,2} \left\| \frac{x^\ell U^{(\ell)}}{V^\alpha}\right\|_\infty  \frac{1}{(E_m-E_n)^2} \frac{(E_n+E_m)\sqrt{E_n^{1+\alpha} E_m^{1+\alpha}}}{nm} .
\end{equation*}
Thus, for every $T \geq 1$,
\begin{equation*}
|\langle U \psi_n,\psi_m \rangle| \lesssim_{\kappa,T,\alpha} \max_{\ell=0,1,2} \left\| \frac{x^\ell U^{(\ell)}}{V^\alpha}\right\|_\infty \frac{E_n^{\alpha}}{1+|m-n|^2} \quad \forall n,m \tc T^{-1} n\leq m\leq Tn.
\end{equation*}
\end{prp}
\begin{proof}
We may assume $M = \| xU'/V^\alpha \|_\infty + \|x^2 U''/V^\alpha\|_\infty < \infty$. 
Note that, under the assumption $V \in \pot_1(\kappa)$, $V'$ is locally integrable at $0$, because $|x|^{-1} V$ is (see Propositions \ref{trivial_doubling_prp} and \ref{doubling_prp}). 
Moreover, $|U'(x)| \leq M |x|^{-1} V(x)^{\alpha}$, which tends to $0$ as $x \to 0$ because $ \alpha/\kappa > 1$, and consequently $U'$ extends to a function in $C^0(\RR)$. Furthermore $|U''(x)| \leq M |x|^{-2}V(x)^{\alpha}$, which is locally integrable at $0$, because $\alpha/\kappa > 1$.
Hence, by Lemma \ref{commutator_identities2_lem}, the Cauchy--Schwarz inequality and Theorem \ref{thm:pointwise_eigenfcts}, 
\[\begin{split}
&|(E_m-E_n)^2 \, \langle U \psi_n, \psi_m \rangle|^2 \\
&\lesssim \left(\int_{\RR} (|V'| |U'|+ V |U''|) |\psi_n \psi_m|  + (E_n + E_m) \int |U'' \psi_n \psi_m| + \int_{\RR} |U'' \psi_n' \psi_m'|\right)^2 \\
&\lesssim M^2 \Biggl[ \int_{\RR} |x|^{-2} \, V^{\alpha+1} \, \psi_n^2 \cdot \int_{\RR} |x|^{-2} \, V^{\alpha+1} \, \psi_m^2\\
&\quad+ (E_n + E_m)^2 \int_{\RR} |x|^{-2} \, V^{\alpha} \, \psi_n^2 \cdot \int_{\RR} |x|^{-2} \, V^{\alpha} \, \psi_m^2 \\
&\quad+ \int_{\RR} |x|^{-2} \, V^{\alpha} \, (\psi_n')^2 \cdot \int_{\RR} |x|^{-2} \, V^{\alpha} \, (\psi_m')^2 \Biggr] \\
&\lesssim_{\kappa,\alpha} M^2 \frac{E_n^{1+\alpha} E_m^{1+\alpha} + (E_n + E_m)^2 E_n^{\alpha} E_m^{\alpha}}{|\{V \leq E_n\}|^2 \, |\{V \leq E_m\}|^2} \\
&\simeq_{\kappa} M^2 \frac{(E_n + E_m)^2 E_n^{1+\alpha} E_m^{1+\alpha}}{n^2 m^2},
\end{split}\]
since $\sqrt{E_n} \, |\{V\leq E_n\}|\simeq_{\kappa} n$ by Theorem \ref{thm:eigenvalues}; the fact that $\alpha/\kappa > 1$ was used in the application of Theorem \ref{thm:pointwise_eigenfcts}.

If $T^{-1} n \leq m \leq T n$, then, by arguing as in the proof of Proposition \ref{prp:5/4_lem}, we get
\[
\frac{E_n^2}{n^2} |m-n|^2 \, |\langle U \psi_n, \psi_m \rangle|
\lesssim_{\kappa,\alpha,T} 
M
\frac{E_n^{2+\alpha}}{n^2},
\]
and the remaining bound follows for $n\neq m$; the on-diagonal bound for $n=m$ is already proved in Proposition \ref{prp:5/4_lem}.
\end{proof}

\begin{rem}
The assumption $\alpha>\kappa$ in Proposition \ref{prp:near_diag_lem_old} may be relaxed for particular potentials $V$. For example, if $V$ satisfies, for some exponent $d > 0$, the reverse doubling condition
\begin{equation}\label{eq:reverse_doubling_remark}
V(\lambda x) \gtrsim \lambda^{d} V(x) \qquad\forall x \in \RR,\, \forall \lambda \geq 1,
\end{equation}
then we can run the same proof for any $\alpha > 1/d$ (indeed under \eqref{eq:reverse_doubling_remark} the function $W(x) = |x|^{-2} V(x)^\alpha$ satisfies the condition \eqref{eq:xV_int_assumption_bil} of Theorem \ref{thm:pointwise_eigenfcts} whenever $\alpha>1/d$).
The choice $d = 1/\kappa$ in \eqref{eq:reverse_doubling_remark} is always possible for any $V \in \pot_1(\kappa)$ in light of Proposition \ref{doubling_prp}, but a specific $V$ may satisfy \eqref{eq:reverse_doubling_remark} for some larger exponent $d$ as well (e.g., in the case $V(x) \simeq |x|^d$).
However the implicit constants in the resulting matrix bounds would then also depend on $d$ and the implicit constant in \eqref{eq:reverse_doubling_remark}.
\end{rem}

\subsection{Off-diagonal decay: the interpolation argument}
We note that Proposition \ref{prp:5/4_lem} can be applied to $U = V^\alpha$ for any $\alpha>0$; in particular, when $\alpha=1$, we obtain some bounds for the matrix $\matP(V)$, which however will not be enough for our purposes in the near-diagonal region.

Proposition \ref{prp:near_diag_lem_old} in principle would provide better bounds. However it cannot be directly applied to $U=V^\alpha$ without additional smoothness assumptions on $V$; moreover, one cannot take $\alpha = 1$ when $\kappa$ is too large, thus preventing in general its direct application to the matrix $\matP(V)$. Nevertheless, an interpolation argument allows us to weaken the smoothness requirements and remove the restrictions on $\alpha$, while still yielding an improved near-diagonal decay compared to the one given by Proposition \ref{prp:5/4_lem}.

In order to run the interpolation argument, it is convenient to introduce the following quantity, which, roughly speaking, should be thought of as interpolating between the ``weighted $C^m$ norms'' $\max_{\ell=0,\dots,m} \left\| x^\ell U^{(\ell)}/V^\alpha\right\|_\infty$ of $U$ used in Propositions \ref{prp:5/4_lem} and \ref{prp:near_diag_lem_old} with $m=1$ and $m=2$ respectively.

\begin{dfn}
Let $V \in \pot_1(\kappa)$. Let $\eta \in C^\infty_c(\Rpos)$ be any nontrivial cutoff. Define, for $s>0$ and $U : \RR \to \CC$,
\[
\|U\|_{V,s} = \sup_{t \neq 0} V(t)^{-1} \| U(t \cdot) \eta \|_{B^s_{\infty,\infty}},
\]
where $B_{\infty,\infty}^s(\RR)$ denotes the Besov space of Lebesgue indices $\infty,\infty$ and order $s$. 
\end{dfn}

\begin{rem}\label{rem:retract}
The norm $\|\cdot \|_{V,s}$ is defined analogously to the local scale-invariant Sobolev norm appearing in smoothness conditions \eqref{eq:mhcond_intro} of Mihlin--H\"ormander type; the main difference is the presence of the weighting factor $V(t)$, which, while not constant, nevertheless satisfies $V(t) \simeq_{\kappa} V(t')$ whenever $|t| \simeq |t'|$.
In particular, it is easily seen that different choices of the cutoff $\eta$ give rise to equivalent norms $\|\cdot\|_{V,s}$, and moreover one can restrict the supremum to a discrete set of scales, namely
\[
\|U\|_{V,s} \simeq_{\kappa,s,\eta} \sup_{t \in D} V(t)^{-1} \| U(t \cdot) \eta \|_{B^s_{\infty,\infty}},
\]
where $D = \{ \pm 2^n \tc n \in \ZZ \}$, provided the set $\{\eta \neq 0\}$ where $\eta$ does not vanish is sufficiently large that $\bigcup_{n \in \ZZ} 2^n \{ \eta \neq 0\} = \Rpos$ (cf., e.g., \cite[Proposition 2.4.1]{martini_phdthesis}). As a consequence, the Banach space determined by the norm $\| \cdot \|_{V,s}$ is a retract of the sequence space $\ell^\infty( (V(t)^{-1} B^s_{\infty,\infty}(\RR))_{t \in D})$ via the maps
\[
U \mapsto (U(t \cdot) \eta)_{t \in D}, \qquad (U_t)_{t \in D} \mapsto \sum_{t \in D} (\tilde \eta U_t)(t^{-1} \cdot),
\]
where $\eta,\tilde \eta \in C^\infty_c(\Rpos)$ are chosen so that $\sum_{n \in \ZZ} \eta(2^n \cdot) = 1$ on $\Rpos$, and $\tilde\eta \eta = \eta$ (see \cite[Section 6.4]{bergh_lofstrom} and \cite[Section 1.18.1]{triebel} for information on retracts and vector-valued sequence spaces in the context of interpolation theory).
\end{rem}

\begin{rem}\label{rem:embedding}
If $s>0$ is not an integer and $m = \lfloor s \rfloor$, then
\[
 \|U\|_{V,s} \simeq_{\kappa,s}  \max_{\ell=0,\dots,m} \left\| \frac{x^\ell U^{(\ell)}}{V} \right\|_\infty + \sup_{t \neq 0, \, |h| \leq 1} \frac{|U^{(m)}(e^h t) - U^{(m)}(t)|}{|h|^{s-m} |t|^{-m} V(t)} .
\]
This is easily proved as a consequence of the characterisation of the Besov space $B_{\infty,\infty}^s(\RR)$ in terms of finite differences \cite[Theorem 6.2.5]{bergh_lofstrom}.
\end{rem}

We can now state the near-diagonal bound for the matrix associated to $U$, which interpolates between those in Propositions \ref{prp:5/4_lem} and \ref{prp:near_diag_lem_old}.

\begin{prp}\label{prp:near_diag_sharper}
Let $V \in \pot_1(\kappa)$. Then, for all $\alpha>0$ and $s>1$, there exists $\epsilon = \epsilon(\kappa,\alpha,s) >0$ such that, for all $U : \RR \to \CC$ and $T \geq 1$,
\begin{equation*}
|\langle U \psi_n,\psi_m \rangle| \lesssim_{\kappa,\alpha,s,T} \|U\|_{V^\alpha,s} \frac{E_n^{\alpha}}{1+|m-n|^{1+\epsilon}} \qquad\forall n,m \tc T^{-1} n\leq m\leq Tn.
\end{equation*}
\end{prp}
\begin{proof}
Let $\alpha> 0$ and $s>1$. Choose any $\alpha_0 \in (0,\alpha)$ and $s_0 \in (1,s)$. Then we can find $\theta \in (0,1)$ sufficiently small that, if $\alpha_1,s_1$ are defined by the equations
\[
\alpha = (1-\theta) \alpha_0 + \theta \alpha_1, \qquad s = (1-\theta) s_0 + \theta s_1,
\]
then $\alpha_1> \kappa$ and $s_1>2$. 

Let $T \geq 1$. Since $s_0>1$ and $\alpha_0>0$, from Proposition \ref{prp:5/4_lem} and Remark \ref{rem:embedding} we deduce the bound
\[
|\langle U \psi_n,\psi_m \rangle| \lesssim_{\kappa,\alpha_0,s_0,T} \|U\|_{V^{\alpha_0},s_0} \frac{E_n^{\alpha_0}}{1+|m-n|} \qquad\forall n,m \tc T^{-1} n\leq m\leq Tn
\]
for all $U : \RR \to \CC$. Similarly, since $s_1>2$ and $\alpha_1>\kappa$, from Proposition \ref{prp:near_diag_lem_old} and Remark \ref{rem:embedding} we deduce the bound 
\[
|\langle U \psi_n,\psi_m \rangle| \lesssim_{\kappa,\alpha_1,s_1,T} \|U\|_{V^{\alpha_1},s_1} \frac{E_n^{\alpha_1}}{1+|m-n|^2} \qquad\forall n,m \tc T^{-1} n \leq m \leq T n
\]
for all $U : \RR \to \CC$. Interpolation of these two bounds (using the upper complex method) yields the bound
\[
|\langle U \psi_n,\psi_m \rangle| \lesssim_{\kappa,\alpha,s,T} \|U\|_{V^\alpha,s} \frac{E_n^{\alpha}}{1+|m-n|^{1+\theta}} \qquad\forall n,m \tc T^{-1} n \leq m \leq T n,
\]
as desired.

To justify the previous interpolation, in light of Remark \ref{rem:retract} and \cite[Theorem 6.4.2]{bergh_lofstrom}, it is enough to prove that
\begin{multline*}
(\ell^\infty( (V(t)^{-\alpha_0} B^{s_0}_{\infty,\infty}(\RR))_{t \in D}),\ell^\infty( (V(t)^{-\alpha_1} B^{s_1}_{\infty,\infty}(\RR))_{t \in D}))^{[\theta]} \\
= \ell^\infty( (V(t)^{-\alpha} B^{s}_{\infty,\infty}(\RR))_{t \in D}).
\end{multline*}
Here we write $(\cdot,\cdot)_{[\theta]}$ and $(\cdot,\cdot)^{[\theta]}$ for the lower and upper complex interpolation functors.
The previous identity follows by duality \cite[Theorems 4.5.1]{bergh_lofstrom} from the fact that
\[\begin{split}
& (\ell^1( (V(t)^{\alpha_0} B^{-s_0}_{1,1}(\RR))_{t \in D}),\ell^1( (V(t)^{\alpha_1} B^{-s_1}_{1,1}(\RR))_{t \in D}))_{[\theta]} \\
&= \ell^1(( (V(t)^{\alpha_0} B^{-s_0}_{1,1}(\RR),V(t)^{\alpha_1} B^{-s_1}_{1,1}(\RR))_{[\theta]})_{t \in D}) \\
& = \ell^1( (V(t)^{\alpha} B^{-s}_{1,1}(\RR))_{t \in D}),
\end{split}\]
which in turn is a consequence of \cite[\S 1.18.1]{triebel} and \cite[Theorem 6.4.5]{bergh_lofstrom}.
\end{proof}

\subsection{Off-diagonal decay: bounds for \texorpdfstring{$\matA$}{A} and \texorpdfstring{$\matP$}{P}}

As a consequence of the theory developed in the previous sections, we can prove the far-diagonal estimate \eqref{far_diag_bound} for $\matA(V)$ stated in Theorem \ref{thm:estimates_A}, and the near-diagonal bound for $\matP(V)$ stated in Theorem \ref{thm:near_diag}.

We first prove that, if $V\in \pot_1(\kappa)$, then, for all $T > 1$,
\begin{equation}\label{eq:far_diag_bound_cor}
|\matA_{nm}(V)| \lesssim_{\kappa,T} \frac{1}{\sqrt{nm}} \left(\frac{E_m}{E_n}\right)^{3/4}, \qquad \forall n,m \tc n \geq T m,
\end{equation} 
and
\begin{equation}\label{eq:near_diag_bound_cor}
|\matP_{nm}(V)| \lesssim_{\kappa,T} \frac{E_n}{1+|m-n|} \qquad\forall n,m \tc T^{-1} n \leq m \leq T n.
\end{equation}
Indeed, by Proposition \ref{prp:5/4_lem}, applied with $U=V$ and $\alpha=1$, we directly obtain the bound \eqref{eq:near_diag_bound_cor}, as well as the bound
\begin{equation}\label{eq:PV_est1}
|\matP_{nm}(V)| \lesssim_{\kappa} \frac{1}{|E_n-E_m|} \sqrt{\frac{(E_n+E_m) E_n^{3/2} E_m^{3/2}}{nm}}
\end{equation}
for all $m \neq n$. By Theorem \ref{thm:eigenvalues}, if $n \geq Tm$, then $E_n-E_m \simeq_{\kappa,T} E_n$.
Hence \eqref{eq:PV_est1} and Proposition \ref{prp:A_formula} give \eqref{eq:far_diag_bound_cor}.

We now prove that, under the stronger assumption $V \in \pot_{1+\theta}(\kappa)$ for some $\theta \in (0,1)$, there exists $\epsilon = \epsilon(\theta,\kappa) >0$ such that, for all $T \geq 1$,
\[
|\matP_{nm}(V)| \lesssim_{\kappa,T} \frac{E_n}{1+|m-n|^{1+\epsilon}} \qquad\forall n,m \tc T^{-1} n \leq m \leq T n.
\]
For this it is enough to note that, if $V \in \pot_{1+\theta}(\kappa)$, then $\| V \|_{V,1+\theta} \lesssim_{\kappa,\theta} 1$ (see Remark \ref{rem:embedding}). So the desired bound follows from Proposition \ref{prp:near_diag_sharper} applied with $U = V$, $s=1+\theta$, and $\alpha=1$.

\subsection{\texorpdfstring{$\ell^2$}{l2}-boundedness of the far-diagonal part of \texorpdfstring{$|\matA|$}{|A|}}
Here we complete the proof of Theorem \ref{thm:estimates_A}, by showing how the $\ell^2$-boundedness of $|\matA(V)| \schur \matF_{T}$ can be deduced from the estimate \eqref{far_diag_bound} on its matrix components.

To this purpose, it is convenient to introduce the following norms for complex-valued sequences.

\begin{dfn}
Let $V \in \pot_1(\kappa)$ and $\alpha \in \RR$. Define
\[
\|\vec v\|_{(V,\alpha)} = \sup_{E>0} E^{-\alpha} \sqrt{\sum_{n \tc E_n(V) \in [E,2E]} |v_n|^2}
\]
for all $\vec v = (v_n)_{n \geq 1} \in \CC^{\Npos}$.
\end{dfn}

Some useful properties of the above norms are collected below.

\begin{lem}\label{lem:Valpha_norm}
Let $V \in \pot_1(\kappa)$.
\begin{enumerate}[label=(\roman*)]
\item\label{en:Valpha_norm_equiv} If $\alpha > 0$, then
\[
\| \vec v \|_{(V,\alpha)} \simeq_{\alpha} \sup_{E>0} E^{-\alpha} \sqrt{\sum_{n \tc E_n(V) \in [0,E]} |v_n|^2}
\]
for all $\vec v \in \CC^{\Npos}$.
\item\label{en:Valpha_Eigen} For all $\alpha \in \RR$,
\[
\| (E_n(V)^\alpha/\sqrt{n})_{n \in \Npos} \|_{(V,\alpha)} \lesssim_{\kappa,\alpha} 1.
\]
\item\label{en:matrix_minmax_bds} Let $\delta,C>0$. Let the matrix $\matM$ satisfy the bounds
\[
|\matM_{nm}| \leq \frac{C}{\sqrt{nm}} \left(\frac{\min\{E_n(V),E_m(V)\}}{\max\{E_n(V),E_m(V)\}}\right)^\delta
\]
for all $n,m \in \Npos$. Then, for all $\alpha \in (-\delta,\delta)$, the bounds
\begin{align}
\label{eq:matrix_comp_bd}
|(\matM . \vec v)_n| &\lesssim_{\kappa,\delta,\alpha} C \|\vec v\|_{(V,\alpha)} \frac{E_n(V)^\alpha}{\sqrt{n}}, \\ 
\label{eq:matrix_sum_bd}
\| \matM . \vec v\|_{(V,\alpha)} &\lesssim_{\kappa,\delta,\alpha} C \|\vec v\|_{(V,\alpha)} , \\
\label{eq:matrix_l2_bd}
\| \matM . \vec v\|_{\ell^2} &\lesssim_{\kappa,\delta} C \|\vec v\|_{\ell^2}
\end{align}
hold for all $\vec v \in \CC^{\Npos}$ and $n \in \NN$.
\end{enumerate}
\end{lem}
\begin{proof}
Let us first prove part \ref{en:Valpha_norm_equiv}. The estimate $\lesssim_{\alpha}$ is trivial. As for the opposite estimate, for all $E>0$,
\[
\sum_{n \tc E_n \in [0,E]} |v_n|^2 = \sum_{j \in \NN}\sum_{n \tc E_n \in [2^{-j-1} E, 2^{-j} E]} |v_n|^2 \lesssim_{\alpha} E^{2\alpha} \|\vec v\|_{(V,\alpha)}^2 \sum_{j \in \NN} 2^{-2\alpha j},
\]
and the desired estimate follows because $\alpha>0$.

We now prove part \ref{en:Valpha_Eigen}. Clearly, for all $E >0$,
\[
\sum_{E_n \in [E,2E]} \frac{E_n^{2\alpha}}{n} \simeq_\alpha E^{2\alpha} \sum_{E_n \in [E,2E]} \frac{1}{n},
\]
hence the desired estimate follows once we check that
\begin{equation}\label{eq:dyadic_1n}
\sup_{E>0} \sum_{E_n \in [E,2E]} \frac{1}{n} \lesssim_\kappa 1.
\end{equation}
this however is an immediate consequence of the fact that, by Proposition \ref{prp:eigenvalue_doub}, if $E_n,E_m \in [E,2E]$ for some $E>0$, then $n/m \simeq_\kappa 1$.

Finally, let us prove part \ref{en:matrix_minmax_bds}. Clearly we may assume $C = 1$.
For all $n \in \Npos$,
\[\begin{split}
|(\matM . \vec v)_n| 
&\leq \sum_m \frac{1}{\sqrt{mn}} \left(\frac{\min\{E_n,E_m\}}{\max\{E_n,E_m\}}\right)^\delta |v_m| \\
&\lesssim_\delta \frac{1}{\sqrt{n}} \sum_{j \in \ZZ} 2^{-\delta|j|} \sum_{m \tc E_m \in [2^j E_n,2^{j+1} E_n]} \frac{|v_m|}{\sqrt{m}}.
\end{split}\]
Hence, by applying the Cauchy--Schwarz inequality to the inner sum and \eqref{eq:dyadic_1n},
\[\begin{split}
|(\matM . \vec v)_n| &\lesssim_{\kappa,\delta} \frac{1}{\sqrt{n}} \sum_{j \in \ZZ} 2^{-\delta|j|} \sqrt{\sum_{m \tc E_m \in [2^j E_n,2^{j+1} E_n]} |v_m|^2} \\
&\leq \frac{E_n^\alpha}{\sqrt{n}} \|\vec v\|_{(V,\alpha)} \sum_{j \in \ZZ} 2^{-|j|\delta+j\alpha} \\
&\lesssim_{\delta,\alpha} \frac{E_n^\alpha}{\sqrt{n}} \|\vec v\|_{(V,\alpha)},
\end{split}\]
which proves \eqref{eq:matrix_comp_bd}.

The bound \eqref{eq:matrix_sum_bd} is an immediate consequence of \eqref{eq:matrix_comp_bd}  and part \ref{en:Valpha_Eigen}.

To prove \eqref{eq:matrix_l2_bd}, we apply Schur's Test to $\matM$ with testing sequence $v_n = 1/\sqrt{n}$; since $\| \vec v\|_{(V,0)} \lesssim_{\kappa} 1$ by part \ref{en:Valpha_Eigen},
 the desired bound follows from \eqref{eq:matrix_comp_bd} applied with $\alpha=0$.
\end{proof}

In light of the far-diagonal decay \eqref{far_diag_bound} of $\matA(V)$, we can apply Lemma \ref{lem:Valpha_norm}\ref{en:matrix_minmax_bds} with $\delta=3/4$ and $\matM = |\matA(V)| \schur \matF_T$, and deduce the
bound $\| |\matA(V)| \schur \matF_T \|_{\ell^2 \to \ell^2} \lesssim_{\kappa,T} 1$ for all $T > 1$ and $V \in \pot_1(\kappa)$. This concludes the proof of Theorem \ref{thm:estimates_A}.

\subsection{Spectral projector bounds: eigenfunctions}\label{ss:proj_bd_eigen}
Here we prove the first estimate in Theorem \ref{thm:spectral_proj}, 
namely, the bound \eqref{eq:spprojbd} for the eigenfunctions $\psi_n$ of $\opH[V] = -\partial_x^2 + V$.

The integral kernel of the spectral projector $\chr_{[0,E_0]}(\opH[V])$ is given by
\begin{equation*}
K(x,y) = \sum_{E_n \in [0,E_0]} \psi_n(x) \, \psi_n(y).
\end{equation*}	
Thus the quantity
\begin{equation*}
\sup_{x\in\RR} \int_\RR K(x,y)^2 \,dy
= \sup_{x\in\RR} \sum_{E_n \in [0, E_0]} \psi_n(x)^2
\end{equation*}
represents the square of the $L^2 \to L^\infty$ operator norm of $\chr_{[0,E_0]}(\opH[V])$. Since this operator is self-adjoint, this is the same as its $L^1 \to L^2$ operator norm, which can be estimated as follows: 
\[\begin{split}
\|\chr_{[0,E_0]}(\opH[V]) f\|_{L^2(\RR)}
&\leq e \|e^{-\opH[V]/E_0} f \|_{L^2(\RR)} \\
&\leq e \|e^{\partial_x^2/E_0}f\|_{L^2(\RR)} \\
&\leq e\left(\frac{8\pi}{E_0}\right)^{-1/4} \|f\|_{L^1(\RR)} ,
\end{split}\]
where in the second line we used the Feynman--Kac formula (see, e.g., \cite[Theorem 6.2]{simon_quantum}) and in the third line a Euclidean heat kernel bound \cite[p.\ 60]{davies_heat_1989}. This completes the proof of the uniform bound
\[
\sum_{E_n \in [0,E_0]}|\psi_n(x)|^2\lesssim \sqrt{E_0}.
\]
It remains to prove the exponential decay for large $|x|$.
	
By Theorem \ref{thm:pointwise_eigenfcts}, if $V(x)\geq 4 E_0$, then
\[
\sum_{E_n \in [0,E_0]}| \psi_n(x)|^2 
\lesssim_{\kappa} \left( \sum_{E_n \in [0,E_0]} \frac{1}{|\{V\leq E_n\}|} \right) e^{-c(\kappa)|x|\sqrt{V(x)}}.
\]
Moreover $|\{V\leq E_n\}| \sqrt{E_n} \simeq_\kappa n$ by Theorem \ref{thm:eigenvalues},
 hence
\begin{equation}\label{eq:sqrtEsum}
\sum_{E_n \in [0,E_0]} \frac{1}{|\{V\leq E_n\}|}
\simeq_\kappa \sum_{E_n \in [0,E_0]} \frac{\sqrt{E_n}}{n}
\lesssim_\kappa \sqrt{E_0} ,
\end{equation}
where Lemma \ref{lem:Valpha_norm}\ref{en:Valpha_norm_equiv}-\ref{en:Valpha_Eigen} was used, and the desired estimate follows.

\subsection{Spectral projector bounds: differentiated eigenfunctions}\label{ss:proj_bd_diffeigen}
Here we complete the proof of Theorem \ref{thm:spectral_proj}, by proving the bound for the ``modified differentiated eigenfunctions'' $\rho_n = \rho_n(\cdot;V)$ defined in \eqref{eq:rho_def_elem}. 

We first prove the uniform bound
\[
\sum_{E_n\in [0,E_0]} \left| \rho_n(x) \right|^2 \lesssim_{\kappa,T_0} \sqrt{E_0}.
\]
By \eqref{eq:rho_def_elem} and Lemma \ref{lem:Valpha_norm}\ref{en:Valpha_norm_equiv}, this would follow from the bound
\[
\|(\matF_{T_0} \schur \matA(V)) .\vec\psi(x)\|_{(V,1/4)} \lesssim_{\kappa,T_0} 1,
\]
where $\vec{\psi}(x) = (\psi_n(x))_{n\geq 1}$.
 In light of Theorem \ref{thm:estimates_A} and Lemma \ref{lem:Valpha_norm}\ref{en:matrix_minmax_bds}, we can apply the estimate \eqref{eq:matrix_sum_bd} with $\matM = \matF_{T_0} \schur \matA(V)$, $\delta=3/4$ and $\alpha=1/4$; hence the desired uniform bound is a consequence of the bound 
$\|\vec\psi(x)\|_{(V,1/4)} \lesssim_{\kappa} 1$,
which is the uniform bound for the $\psi_n$ proved in Section \ref{ss:proj_bd_eigen}.

To prove the exponential decay part of the statement, we decompose $\rho_n$ as $\sigma_n + (\rho_n-\sigma_n)$, and consider the two summands separately.

First, the same argument used in Section \ref{ss:proj_bd_eigen} to prove the exponential decay for the $\psi_n$, using Proposition \ref{sigma_pointwise_prp} in place of Theorem \ref{thm:pointwise_eigenfcts}, yields 
\begin{equation*}
\sum_{E_n\in [0,E_0]} |\sigma_n(x)|^2 \lesssim_{\kappa} \sqrt{E_0} \, e^{-c(\kappa)|x|\sqrt{V(x)}}
\end{equation*}
for every $x$ such that $V(x)\geq 4E_0$. Next, by \eqref{eq:rho_def_elem},
\[
\sigma_n(x) - \rho_n(x) = \sum_m (\matN_{T_0} \schur \matA(V))_{nm} \psi_m(x),
\]
so
\[
\begin{split}
&\sum_{E_n \in [0,E_0]} |\rho_n(x) - \sigma_n(x)|^2 \\
&= \sum_{E_n \in [0,E_0]} \left|\sum_{T_0^{-1} n \leq m \leq T_0 n} \matA_{nm}(V) \psi_m(x) \right|^2 \\
&\lesssim_{\kappa} \sum_{E_n \in [0,E_0]} n^2 \sum_{T_0^{-1} n \leq m \leq T_0 n} \psi_m(x)^2 .
\end{split}
\]
by the Cauchy--Schwarz inequality and Theorem \ref{thm:estimates_A}.

By Proposition \ref{prp:eigenvalue_doub}, $m\in [T_0^{-1} n, T_0 n]$ implies $E_m\in [2^{-\ell_1}E_n, 2^{\ell_1} E_n]$ for some $\ell_1=\ell_1(\kappa,T_0) \in \NN$. If $V(x)\geq  2^{\ell_1+2}E_0$, then
\begin{equation}\label{eq:n_estimate}
n \simeq_{\kappa} \sqrt{E_n} \, |\{V \leq E_n\}| \lesssim_{\kappa} |x| \sqrt{V(x)}
\end{equation}
whenever $E_n \leq E_0$, by Theorem \ref{thm:eigenvalues}.
Moreover, if $E_n \in [0,E_0]$ and $V(x)\geq  2^{\ell_1+2}E_0$, then
\[\begin{split}
\sum_{T_0^{-1} n \leq m \leq T_0 n}\psi_m(x)^2 
\leq \sum_{E_m \in [0,2^{\ell_1} E_n]}\psi_m(x)^2
\lesssim_{\kappa,T_0} \sqrt{E_n}  \, e^{-c(\kappa) |x|\sqrt{V(x)}}  
\end{split}\]
by \eqref{eq:spprojbd}.
Hence, if $V(x)\geq  2^{\ell_1+2}E_0$, then
\[\begin{split}
&\sum_{E_n \in [0,E_0]} n^2 \sum_{T_0^{-1} n \leq m \leq T_0 n} \psi_m(x)^2\\
 &\lesssim_{\kappa,T_0} e^{-c(\kappa) |x|\sqrt{V(x)}} \sum_{E_n \in [0,E_0]} n^2 \sqrt{E_n} \\
 &\lesssim_{\kappa}  e^{-c(\kappa) |x|\sqrt{V(x)}} \, \left(|x|\sqrt{V(x)}\right)^3 \sum_{E_n \in [0,E_0]} \frac{\sqrt{E_n}}{n} \\
&\lesssim_{\kappa} \sqrt{E_0} \, e^{-\frac{c(\kappa)}{2} |x| \sqrt{V(x)}},
\end{split}\]
where \eqref{eq:n_estimate} and \eqref{eq:sqrtEsum} were applied in the last two inequalities.

This proves the desired estimate with $T_1 = 2^{\ell_1+2}$.

\section{Proof of the weighted Plancherel estimate}\label{s:proofweightedplancherel}

The aim of this section is to prove the ``weighted Plancherel estimate'' of Theorem \ref{thm:conditional_weightedplancherel} and Remark \ref{rem:conditional_weightedplancherel} in the case where $q=\infty$ and $V$ satisfies the assumptions \eqref{eq:2dassumptions}, that is, $V$ belongs to one of the the classes $\pot_{1+\theta}(\kappa)$ introduced in Section \ref{s:matrixbounds}. Note that any $V \in \pot_{1+\theta}(\kappa)$ (more generally, any $V \in \pot_1(\kappa)$) satisfies the assumption \eqref{eq:V_doubling} with $D=\kappa$ (see Proposition \ref{doubling_prp}); so Theorem \ref{thm:conditional_weightedplancherel} applies to such $V$ and, combined with Theorem \ref{thm:abstract} and the result below, proves Theorem \ref{thm:main}.

\begin{thm}\label{thm:weightedplancherel}
Assume that $V \in \pot_{1+\theta}(\kappa)$ for some $\theta \in (0,1)$ and $\kappa \geq 1$.
Let $\opL$ be the Grushin operator on $\RR^2$ associated to $V$, as in \eqref{eq:grushin_def}.
 Then, for all $\vartheta \in [0,1/2)$ and $r>0$, the estimate
\begin{multline*}
\esssup_{z' \in \RR^2} \, r^{2-2\vartheta} \max\{V(r),V(x')\}^{1/2-\vartheta} \int_{\RR^2} |y-y'|^{2\vartheta} \left|\Kern_{\mm(r^2 \opL)}(z,z') \right|^2 \,dz 
\\\lesssim_{\theta,\kappa,\vartheta} \|\mm\|_{\sobolev{\vartheta}{\infty}}^2
\end{multline*}
holds for all continuous $\mm : \RR \to \CC$ with $\supp \mm \subseteq [-1,1]$. The analogous estimate with $\vartheta=0$ holds more generally for any $V \in \pot_1(\kappa)$.
\end{thm}

We note that a sharper result, where $\sobolev{\vartheta}{\infty}$ is replaced by $\sobolev{\vartheta}{2}$, is proved in \cite{martini_sharp_2014} in the case where $V(x) = x^2$.

\subsection{Preliminaries}

We denote by $\left(\cdot,\cdot\right)$ the $\ell^2(\Npos)$-scalar product and by $\|\cdot\|$ the associated norm. For a vector $\vec{a}=(a_n)_{n\geq 1}$ and a function $F$ of one variable, we write $F(\vec{a}) \defeq (F(a_n))_{n\geq 1}$. We also set 
\begin{equation}\label{eq:vec_inc_def}
\diag \vec{a} = (\delta_{nm}a_n)_{n,m\geq 1} \quad \text{and}\quad  \inc \vec{a} = (a_n-a_m)_{n,m\geq 1}
\end{equation}
(respectively, the \emph{diagonal matrix} and the \emph{increment matrix} associated to $\vec{a}$). Recall that, if $\matB$ and $\matB'$ are matrices, $\matB \schur \matB'$ is their Schur product \eqref{eq:schur_def}, while $|\matB|$ denotes the componentwise modulus of $\matB$.

Let $\opL$ be the Grushin operator defined in \eqref{eq:grushin_def}, associated to a function $V \in \pot$.
The self-adjoint operator $\opL$ commutes with the differential operator $-\partial_y^2$ on $\RR^2$, and the two operators have a joint functional calculus on $L^2(\RR^2)$ in the sense of the spectral theorem, which can be conveniently analysed by taking the partial Fourier transform in the variable $y$.
Indeed, if we write, for any sufficiently regular function $f$ on $\RR^2$ and $x,\xi \in \RR$,
\[
f^{\xi}(x) \defeq \int_{\RR} f(x,y) \,e^{-i\xi y} \,dy,
\]
then
\[
(\opL f)^\xi = \opH[\xi^2 V] f^\xi, \qquad (-\partial_y^2 f)^\xi = \xi^2 f^\xi,
\]
where, for all $\tau \in \Rpos$, $\opH[\tau V]$ is the Schr\"odinger operator defined in \eqref{eq:def_opH}.
Consequently
\begin{equation}\label{eq:jointcalculus_four}
(G(\opL,-\partial_y^2) f)^\xi = G(\opH[\xi^2 V],\xi^2) f^\xi
\end{equation}
for all bounded Borel functions $G : \Rnon \times \Rnon \to \CC$.

As in Section \ref{basic_sec}, we consider the eigenvalues $E_n(\tau V)$ and eigenfunctions $\psi_n(\cdot;\tau V)$ associated to $\opH[\tau V]$.
Moreover, as in Section \ref{ss:proj_bd_diffeigen}, it is convenient to introduce the infinite vector 
\begin{equation*}
\vec{\psi}(x;\tau V) = (\psi_n(x;\tau V))_{n\geq 1}.
\end{equation*}
Whenever $\matM(\tau)=(\matM_{nm}(\tau))_{n,m\geq 1}$ is a $\tau$-dependent infinite matrix, we define, at least formally,
\begin{equation}\label{eq:k_mat}
k_{\matM}(x', x;\tau)=\left( \matM(\tau).\vec{\psi}(x;\tau V),\vec{\psi}(x';\tau V)\right).
\end{equation}
In other words, $k_{\matM}(\cdot,\cdot;\tau)$ is the integral kernel of the operator on $L^2(\RR)$ whose matrix is $\matM(\tau)$ with respect to the basis $(\psi_n(\cdot;\tau V))_{n \geq 1}$.

By using the above matrix notation, we can conveniently express the formula for the integral kernel of an operator $G(\opL,-\partial_y^2)$ in the joint functional calculus of $\opL$ and $-\partial_y^2$, which is an elementary consequence of \eqref{eq:jointcalculus_four}.

\begin{prp}\label{kernel_prp}
Given $G : \Rnon \times \Rnon \to \CC$, the integral kernel $K_G(z',z)$ of $G(\opL,-\partial_y^2)$ is given by 
\begin{equation*}
K_G(z',z) = \frac{1}{2\pi} \int_\RR k_{\matM}(x', x;\xi^2) \, e^{i\xi(y'-y)} \,d\xi \qquad(z=(x,y), \ z'=(x',y')),
\end{equation*}
where
\begin{equation}\label{eq:m_def}
\matM(\tau) \defeq \diag G(\vec{E}(\tau V), \tau).
\end{equation}
\end{prp}

In light of the above formula, multiplication by $(y'-y)$ of the integral kernel $K_G(z',z)$ corresponds to differentiation in $\xi$ of $k_{\matM}(x', x;\xi^2)$. Thanks to the matrix notation, the result of this differentiation can be expressed in a particularly concise form. Recall that $\matA(\tau V)$ is the matrix associated to $\tau V$ defined in Section \ref{s:matrixbounds}.

\begin{prp}\label{Mderivative_prp}
If $k_\matM$ is given by \eqref{eq:k_mat} for some $\matM = \matM(\tau)$, then
\begin{equation}\label{eq:DM_deriv}
\tau\partial_\tau k_{\matM}(x', x;\tau)=k_{D[\matM]}(x', x;\tau),
\end{equation}
where
\begin{equation}\label{eq:DM_def}
D[\matM] \defeq \tau\partial_\tau \matM + [\matM, \matA(\tau V)]. 
\end{equation}
Moreover, if $\matM(\tau) = \diag \vec{a}(\tau)$, then 
\begin{equation}\label{eq:DM_diag}
D[\matM](\tau) = \diag \tau\partial_\tau\vec{a}(\tau)  + \matA(\tau V) \schur \inc \vec{a}(\tau) .
\end{equation}
\end{prp}
\begin{proof}
In light of \eqref{eq:vec_inc_def}, formula \eqref{eq:DM_diag} is just a rewriting of \eqref{eq:DM_def} in the case $\matM$ is a diagonal matrix. We are left with the proof of \eqref{eq:DM_deriv} for an arbitrary $\matM$.

Note that, by the definition of $\matA(\tau V)$,
\[
\tau\partial_\tau \vec\psi(\cdot,\tau V) = \matA(\tau V) . \vec\psi(\cdot,\tau V).
\]
Hence, by \eqref{eq:k_mat} and the Leibniz rule,
\[\begin{split}
\tau \partial_\tau k_{\matM}(x', x;\tau)
&=\left( (\tau \partial_\tau \matM(\tau)).\vec{\psi}(x;\tau V),\vec{\psi}(x';\tau V)\right) \\
&\quad+\left( \matM(\tau). \matA(\tau V) . \vec{\psi}(x;\tau V)),\vec{\psi}(x';\tau V)\right) \\
&\quad+\left( \matM(\tau).\vec{\psi}(x;\tau V),\matA(\tau V). \vec{\psi}(x';\tau V)\right).
\end{split}\]
Since $\matA(\tau V)$ is skew-adjoint (see Proposition \ref{prp:A_formula}), the desired formula follows.
\end{proof}

By combining the previous formulas, the Plancherel theorem for the Fourier transform in $y$, and the orthonormality of the eigenfunctions $\psi_n(\cdot,\tau V)$, we immediately obtain the following ``weighted Plancherel identities'' for the integral kernel $K_G$ of an operator in the joint functional calculus of $\opL$ and $-\partial_y^2$, which will be the starting point for our analysis.

\begin{prp}\label{plancherel_identity_prp}
If $K_G$ and $\matM$ are as in Proposition \ref{kernel_prp}, then
\begin{equation*}
\int_{\RR^2} |K_G(z',z)|^2 \,dz' = \frac{1}{2\pi}\int_0^{\infty} \| \matM(\tau).\vec{\psi}(x;\tau V)\|^2 \tau^\frac{1}{2} \frac{d\tau}{\tau}
\end{equation*}
and
\begin{equation*}
\int_{\RR^2} (y'-y)^2 \, |K_G(z',z)|^2 \,dz' = \frac{2}{\pi}\int_0^{\infty} \| D[\matM](\tau).\vec{\psi}(x;\tau V)\|^2 \tau^{-\frac{1}{2}} \frac{d\tau}{\tau}.
\end{equation*}
\end{prp}

\subsection{Decompositions}

Let $V\in \pot_1(\kappa)$ and $\mm\in C^\infty_c([-1,1])$. In order to analyse the operator $\mm(\opL)$ and its integral kernel, it is convenient to introduce a dyadic decomposition along the spectrum of $-\partial_y^2$.

Namely, let $\chi\in C^\infty_c([1,4])$ be such that $\sum_{j\in \ZZ}\chi(2^j\tau)=1$. We set
\begin{equation}\label{en:GA_def}
G_A(\lambda, \tau) \defeq \mm(\lambda) \,\chi(A\tau),
\end{equation}
where $A>0$ is a parameter. Here $\lambda$ is the variable in the spectrum of $\opL$ and $\tau=\xi^2$ is the variable in the spectrum of $-\partial_y^2$.

In this case the matrix given by Proposition \ref{kernel_prp} is 
\begin{equation}\label{eq:matrix_G_A}
\matM(\tau)=\diag{G_A(\vec{E}(\tau V), \tau)}=\chi(A\tau) \cdot \matM_1(\tau),
\end{equation}
where
\begin{equation}\label{eq:m1_def}
\matM_1(\tau) \defeq \diag{\mm(\vec{E}(\tau V))},
\end{equation}
and Proposition \ref{plancherel_identity_prp} yields
\begin{equation}\label{plancherel_decomposition}
\int_{\RR^2} |K_{G_A}(z',z)|^2 \,dz' 
\lesssim A^{-\frac{1}{2}} \int_{A^{-1}}^{4 A^{-1}}  \| \matM_1(\tau) .\vec{\psi}(x;\tau V) \|^2 \frac{d\tau}{\tau} .
\end{equation}

We now obtain a similar expression for the analogous integral with the weight $(y'-y)^2$. By the Leibniz and chain rules,
\begin{equation*}
\tau\partial_\tau(G_A(E_n(\tau V), \tau)) = \widetilde{\chi}(A\tau) \cdot \mm(E_n(\tau V)) + \chi(A\tau) \cdot \mm'(E_n(\tau V)) \cdot F_n(\tau V),
\end{equation*} 
where $\widetilde\chi(x)=x\chi'(x)$, and $F_n(\tau V)$ is as in \eqref{virial_F}. Thus, by Proposition \ref{Mderivative_prp},
\[
D[\matM](\tau)
= \widetilde{\chi}(A\tau)\cdot \matM_1(\tau) + \chi(A\tau) \sum_{k=2}^4 \matM_k(\tau),
\]
where $\matM_1(\tau)$ is as in \eqref{eq:m1_def}, while
\begin{align}
 \matM_2(\tau)&\defeq \diag \mm'(\vec{E}(\tau V)) \schur \diag \vec{F}(\tau V), \label{eq:m2_def}\\
	\matM_3(\tau)&\defeq \matN \schur \matA(\tau V) \schur \inc{\mm(\vec{E}(\tau V))} , \label{eq:m3_def}\\
	 \matM_4(\tau)&\defeq \matF \schur \matA(\tau V) \schur \inc{\mm(\vec{E}(\tau V))} . \label{eq:m4_def}
\end{align}
Here $\matN$ and $\matF$ denote the cutoff matrices $\matN_T$ and $\matF_T$ of Section \ref{s:matrixbounds}, with $T = 2$.
Then Proposition \ref{plancherel_identity_prp} yields
\begin{equation}\label{plancherel_decomposition2}
\int_{\RR^2}(y'-y)^2 \, |K_{G_A}(z',z)|^2 \,dz' 
\lesssim A^\frac{1}{2} \sum_{k=1}^4\int_{A^{-1}}^{4 A^{-1}} \|\matM_k(\tau).\vec{\psi}(x;\tau V)\|^2 \frac{d\tau}{\tau} .
\end{equation}

In light of \eqref{plancherel_decomposition} and \eqref{plancherel_decomposition2}, we are reduced to estimating the quantities
\begin{equation}\label{eq:planch_terms}
\sup_{\tau\in [A^{-1}, 4 A^{-1}]} \|\matM_k(\tau).\vec{\psi}(x;\tau V)\|^2 .
\end{equation}
for $k=1,2,3,4$. This is discussed in Sections \ref{ss:diagonal_bound} to \ref{ss:neardiagonal} below, where the estimates proved in the previous Sections \ref{basic_sec} and \ref{s:matrixbounds} will play a fundamental role. In applying those estimates, it is crucial to remember that, if $V \in \pot_{1+\theta}(\kappa)$ for some $\theta \in [0,1)$ and $\kappa > 1$, then the scaled potentials $\tau V$ belong to the same class $\pot_{1+\theta}(\kappa)$ for all $\tau>0$.
A useful tool in the upcoming discussion is an elementary strengthening of Theorem \ref{thm:spectral_proj}, which we state here for convenience.

\begin{lem}\label{spectral_proj_final_lem}
Let $V\in \pot_1(\kappa)$. Then, for every $A,E>0$,
\[
\sup_{\tau\geq A^{-1}}\sum_{E_n(\tau V) \in [0,E]} \psi_n(x;\tau V)^2 
\lesssim_{\kappa} \sqrt{E} \left(\chr_{V\leq 4 A E} + e^{-c(\kappa) A^{-1/2} |x| \sqrt{V(x)}} \chr_{V \geq 4 A E}\right)
\]
and, if $\rho_n$ is defined as in Theorem \ref{thm:spectral_proj} with $T_0=2$, then
\[
\sup_{\tau\geq A^{-1}}\sum_{E_n(\tau V) \in [0,E]} \rho_n(x;\tau V)^2 
\lesssim_{\kappa} \sqrt{E} \left(\chr_{V \leq T_1 A E} + e^{-c(\kappa) A^{-1/2} |x| \sqrt{V(x)}} \chr_{V \geq T_1 A E}\right),
\]
where $T_1 = T_1(\kappa)$. 
\end{lem}

\subsection{The diagonal bounds}\label{ss:diagonal_bound}
Here we continue to assume $V \in \pot_1(\kappa)$, and consider the terms \eqref{eq:planch_terms} with $k=1,2$.
First of all, since $\supp \mm \subseteq [-1,1]$,
\begin{equation}\label{M1}
\begin{split}
&\sup_{\tau\in [A^{-1}, 4 A^{-1}]} \|\matM_1(\tau).\vec{\psi}(x;\tau V)\|^2 \\
&=\sup_{\tau\in [A^{-1}, 4 A^{-1}]} \sum_{n\geq 1} |\mm(E_n(\tau V))|^2 \psi_n(x;\tau V)^2 \\
&\leq \|\mm\|_\infty^2\sup_{\tau\in [A^{-1}, 4 A^{-1}]}\sum_{E_n(\tau V) \in [0,1]} \psi_n(x;\tau V)^2 \\
&\lesssim_{\kappa} \|\mm\|_\infty^2\left(\chr_{V \leq 4 A} + e^{-c(\kappa) A^{-1/2} |x|\sqrt{V(x)}} \chr_{V \geq 4 A}\right) ,
\end{split}
\end{equation} 
by Lemma \ref{spectral_proj_final_lem}.

We can treat similarly the term involving $\matM_2$. In fact, by Theorem \ref{thm:virial}, $0\leq F_n(\tau V)\leq E_n(\tau V)$ and thus, if we define $\tilde\mm$ by
\[
\widetilde{\mm}(\lambda) \defeq \lambda \mm'(\lambda),
\]
then
\begin{equation}\label{M2}
\begin{split}
&\sup_{\tau\in [A^{-1}, 4 A^{-1}]} \|\matM_2(\tau).\vec{\psi}(x;\tau V) \|^2 \\
&\lesssim_{\kappa} \|\widetilde{\mm}\|_\infty^2\left(\chr_{V \leq 4 A} + e^{-c(\kappa) A^{-1/2} |x| \sqrt{V(x)}} \chr_{V \geq 4 A} \right) \\
&\lesssim_{\kappa} \|\mm'\|_\infty^2 \left(\chr_{V \leq 4 A} + e^{-c(\kappa) A^{-1/2} |x|\sqrt{V(x)}} \chr_{V \geq 4 A} \right) ,
\end{split}
\end{equation}
where we used again that $\supp\mm \subseteq [-1,1]$.

\subsection{Terms far from the diagonal}
Here we consider the term \eqref{eq:planch_terms} with $k=4$, still under the assumption $V \in \pot_1(\kappa)$.
In view of \eqref{eq:m4_def}, the modulus of the $n$th entry of $\matM_4(\tau).\vec{\psi}(x;\tau V)$ is
\[\begin{split}
&\left|\sum_{m\in [1, n/2)\cup (2n, \infty)} \matA_{nm}(\tau V) \left( \mm(E_n(\tau V)) - \mm(E_m(\tau V)) \right) \psi_m(x;\tau V) \right| \\
&\leq  \|\mm\|_\infty \Biggl(\chr_{E_n(\tau V)\in [0,1]} |\rho_n(x;\tau V)| \\
&\qquad\qquad+ \sum_{m\in [1, n/2) \cup (2 n, \infty)} |\matA_{nm}(\tau V) \chr_{E_m(\tau V) \in [0,1]} \psi_m(x;\tau V)|\Biggr),
\end{split}\]
where we used the definition \eqref{eq:rho_def_elem} of $\rho_n$.
Thus, from the (uniform in $\tau$) $\ell^2$-boundedness of $|\matA(\tau V)| \schur \matF$ proved in Theorem \ref{thm:estimates_A} we deduce that
\begin{multline*}
\|\matM_4(\tau).\vec{\psi}(x;\tau V)\|^2\\
\lesssim_{\kappa} \|\mm\|_\infty^2\left(\sum_{E_n(\tau V)\in [0,1]} |\rho_n(x;\tau V)|^2+ \sum_{E_n(\tau V)\in [0,1]} |\psi_n(x;\tau V)|^2\right).
\end{multline*}
Hence, by Lemma \ref{spectral_proj_final_lem}, 
\begin{multline}\label{M4}
\sup_{\tau\in [A^{-1}, 4 A^{-1}]} \|\matM_4(\tau).\vec{\psi}(x;\tau V)\|^2 \\
\lesssim_{\kappa}
\|\mm\|_\infty^2 \left(\chr_{V \leq T_2 A E} + e^{-c(\kappa) A^{-1/2} |x| \sqrt{V(x)}} \chr_{V \geq T_2 A E}\right), 
\end{multline}
where $T_2 = T_2(\kappa)$.

\subsection{The near-diagonal bound}\label{ss:neardiagonal}
Finally we consider the term \eqref{eq:planch_terms} with $k=3$. For this bound we use the stronger assumption $V \in \pot_{1+\theta}(\kappa)$ for some $\theta \in (0,1)$.

In view of \eqref{eq:m3_def}, the modulus of the $(n,m)$-entry of $\matM_3(\tau)$ is
\begin{equation*}
\chr_{n/2\leq m \leq 2n} \, |\matA_{nm}(\tau V)| \, |\mm(E_n(\tau V))-\mm(E_m(\tau V))|.
\end{equation*}
Notice that $\mm(E_n(\tau V)) - \mm(E_m(\tau V))$ vanishes unless $E_m(\tau V)\in [0,1]$ or $E_n(\tau V)\in [0,1]$. In the latter case Proposition \ref{prp:eigenvalue_doub} implies $E_m(\tau V)\in [0,S]$, where $S=S(\kappa) \geq 1$. Hence, by Proposition \ref{prp:A_formula} and Theorem \ref{thm:near_diag},
\[\begin{split}
&|(\matM_3)_{nm}(\tau)|\\
&\leq \|\mm'\|_\infty \, \chr_{n/2 \leq m \leq 2n} \, |\matA_{nm}(\tau V)| \left|E_n(\tau V) - E_m(\tau V) \right| \chr_{E_m(\tau V) \in [0,S]}\\
&= \|\mm'\|_\infty \, \chr_{n/2 \leq m \leq 2n} \, |\matP_{nm}(\tau V)| \, \chr_{E_m(\tau V) \in [0,S]}\\
&\lesssim_{\kappa,\theta}  \|\mm'\|_\infty \frac{1}{1+|m-n|^{1+\epsilon}} \chr_{E_m(\tau V) \in [0,S]},
\end{split}\]
where $\epsilon = \epsilon(\kappa,\theta) > 0$.
Applying Schur's Test yields
\begin{equation*}
\|\matM_3(\tau).\vec{\psi}(x;\tau V)\|^2\lesssim_{\kappa,\theta,T} \|\mm'\|_\infty^2 \sum_{E_m(\tau V)\in [0,S]}|\psi_m(x;\tau V)|^2.
\end{equation*} 
Notice that in applying Theorem \ref{thm:near_diag} we used the fact that the potentials $\tau V$ belong to the same class $\pot_{1+\theta}(\kappa)$ for all $\tau >0$. Another application of Lemma \ref{spectral_proj_final_lem} yields 
\begin{multline}\label{M3}
\sup_{\tau\in [A^{-1}, 4 A^{-1}]} \|\matM_3(\tau).\vec{\psi}(x;\tau V)\|^2 \\
 \lesssim_{\kappa,\theta} \|\mm'\|_\infty^2\left(\chr_{V \leq 4 A S} + e^{-c(\kappa) A^{-1/2} |x| \sqrt{V(x)}} \chr_{V \geq 4 A S}\right),
\end{multline}
where, as noted above, $S=S(\kappa)$.

\subsection{Conclusion}

Assume that $V \in \pot_{1+\theta}(\kappa)$ for some $\theta \in (0,1)$.
By combining \eqref{plancherel_decomposition} and \eqref{M1}, we obtain that
\begin{equation*}
\int_{\RR^2} |K_{G_A}(z',z)|^2 \,dz' \lesssim_{\kappa} A^{-\frac{1}{2}} \|\mm\|_\infty^2 \left(\chr_{V\leq T_3 A}+e^{-c(\kappa) A^{-1/2} |x| \sqrt{V(x)}} \chr_{V \geq T_3 A}\right).
\end{equation*}
Similarly, from \eqref{plancherel_decomposition2}, \eqref{M1}, \eqref{M2}, \eqref{M3}, and \eqref{M4} we deduce that
\begin{multline*}
\int_{\RR^2} (y'-y)^2 |K_{G_A}(z',z)|^2 \,dz' \\
\lesssim_{\kappa,\theta} A^\frac{1}{2} \left(\|\mm\|_\infty^2 + \|\mm'\|_\infty^2\right) \left(\chr_{V\leq T_3 A} + e^{-c(\kappa) A^{-1/2} |x| \sqrt{V(x)}} \chr_{V \geq T_3 A}\right).
\end{multline*}
Here $T_3 \defeq \max\{4S, T_2\}$ depends only on $\kappa$. Interpolation of the above two estimates yields
\begin{multline}\label{eq:interpolated_scale}
\int_{\RR^2} |y'-y|^{2\vartheta} |K_{G_A}(z',z)|^2 \,dz' \\
\lesssim_{\theta,\kappa,\vartheta} \|\mm\|_{\sobolev{\vartheta}{\infty}}^2 A^{\vartheta-\frac{1}{2}} \left(\chr_{V \leq T_3 A} + e^{-c(\kappa) A^{-1/2} |x| \sqrt{V(x)}} \chr_{V \geq T_3 A}\right)
\end{multline}
for every $\vartheta\in[0,1]$. We note that, as discussed in Section \ref{ss:diagonal_bound}, the estimate at the endpoint $\vartheta=0$ is valid more generally for $V \in \pot_1(\kappa)$.

When $\vartheta < 1/2$, we can sum the interpolated estimates \eqref{eq:interpolated_scale} corresponding to different scales $A$ to obtain an estimate for $\Kern_{\mm(\opL)}$. Indeed, by \eqref{en:GA_def},
\begin{equation}\label{eq:dyadic_kern}
\Kern_{\mm(\opL)} = \sum_{j \in \ZZ} K_{G_{2^j}} .
\end{equation}
Moreover, in light of  Proposition \ref{kernel_prp}, formula \eqref{eq:matrix_G_A} and the support conditions $\supp \mm \subseteq [-1,1]$ and $\supp \chi \subseteq [1,4]$, 
 the kernel $K_{G_A}$ vanishes identically 
 unless there is a $\tau \in [A^{-1}, 4 A^{-1}]$ such that $E_1(\tau V) \leq 1$; by Proposition \ref{prp:eigenvalue_doub}\ref{en:energy_doubling_lbd} (applied to the potentials $\tau V$ with $E=1$), this in turn implies that $K_{G_A} = 0$ unless $A \geq a V(1)$, where $a = a(\kappa)>0$, and the sum in \eqref{eq:dyadic_kern} can be restricted accordingly.

Hence, if $\vartheta<\frac{1}{2}$, by \eqref{eq:interpolated_scale} and the triangle inequality,
\[\begin{split}
&\left(\int_{\RR^2} |y'-y|^{2\vartheta} |\Kern_{\mm(\opL)}(z',z)|^2 \,dz' \right)^{1/2} \\
&\leq \sum_{j \tc 2^j \geq a V(1)} \left(\int_{\RR^2} |y'-y|^{2\vartheta} |K_{G_{2^j}}(z',z)|^2 \,dz' \right)^{1/2} \\
&\lesssim_{\theta,\kappa,\vartheta} \|\mm\|_{\sobolev{\vartheta}{\infty}}\sum_{j \tc 2^j \geq a V(1)} 2^{\frac{j}{2}(\vartheta-\frac{1}{2})} \left(\chr_{V \leq T_3 2^j} 
+ e^{-c(\kappa) 2^{-j/2} |x| \sqrt{V(x)}} \chr_{V\geq T_32^j}\right) \\
&\lesssim \|\mm\|_{\sobolev{\vartheta}{\infty}} \left(\sum_{j \tc 2^j \geq \max\{T_3^{-1} V(x), a V(1)\}} 2^{\frac{j}{2}(\vartheta-\frac{1}{2})} \right.\\
&\qquad\qquad\qquad\qquad\qquad + \left. e^{-c(\kappa) T_3^{1/2} |x|} \sum_{j \tc 2^j \geq a V(1)} 2^{\frac{j}{2}(\vartheta-\frac{1}{2})}\right)\\
&\lesssim_{\vartheta,\kappa} \|\mm\|_{\sobolev{\vartheta}{\infty}}\left( \max\{V(1),V(x)\}^{\frac{\vartheta}{2} - \frac{1}{4}} 
+ V(1)^{\frac{\vartheta}{2} - \frac{1}{4}}  e^{-c(\kappa) T_3^{1/2} |x|}\right) \\
&\lesssim_{\vartheta,\kappa} \|\mm\|_{\sobolev{\vartheta}{\infty}}  \max\{V(1),V(x)\}^{\frac{\vartheta}{2} - \frac{1}{4}} ,
\end{split}\]
where the last estimate follows from the fact that
\[
V(x) \simeq_\kappa V(|x|) \lesssim_\kappa V(1) \, (1+|x|)^{\kappa} \lesssim_{\kappa,\beta} V(1) \, e^{\beta |x|}
\]
for all $\beta>0$, by Proposition \ref{doubling_prp}. 

This proves the weighted Plancherel estimate of Theorem \ref{thm:weightedplancherel} in the case $r=1$. In the general case, define
\[
V_r(x) = r^2 V(rx), \qquad \opL_r = -\partial_x^2 - V_r(x) \partial_y^2
\]
and observe that $V_r \in \pot_{1+\theta}(\kappa)$ for any $r > 0$. Moreover, if $T_r$ is the isometry of $L^2(\RR^2)$ defined by
\[
T_r f(x,y) = r^{-1/2} f(x/r,y),
\]
then it is immediately checked that $(r^2 \opL) T_r = T_r \opL_r$,
whence
\[
\Kern_{\mm(r^2 \opL)}((x',y'),(x,y)) = r^{-1} \Kern_{\mm(\opL_r)}((x'/r,y'),(x/r,y)) 
\]
and the desired estimate for $\opL$ and arbitrary $r>0$ easily follows by applying the previous estimate to $\opL_r$.

\def\cprime{$'$}

\end{document}